%% file: Lgraphv3.tex
\documentclass[11pt]{amsart}
\theoremstyle{plain}
\usepackage{amsmath, amssymb,amscd, graphicx, enumitem, cancel}
\usepackage{nccmath}
\input{macros}

\title{L-space intervals for Graph Manifolds and Cables}
\author{Sarah Dean Rasmussen}
\email{S.Rasmussen@dpmms.cam.ac.uk}
\address{DPMMS\\ University of Cambridge\\ UK}
\thanks{The author was supported by EPSRC grant EP/M000648/1}

\begin{document}
\begin{abstract}
We present a graph manifold analog of the Jankins-Neumann 
classification of Seifert fibered spaces over $S^2\mkern-3mu$
admitting taut foliations,
providing a finite recursive formula to compute the
L-space Dehn-filling interval for any graph manifold
with torus boundary.
As an application of a generalization of this result to Floer simple
manifolds, 
we compute the L-space interval
for any cable of a Floer simple knot complement in
a closed three-manifold in terms of the original L-space interval,
recovering a result of Hedden and Hom as a special case.
\end{abstract}
\maketitle

\section{Introduction}

In the late 1990's, Thurston showed 
\cite{Thurstonfolact, CalegariDunfieldact}
that any taut foliation on an atoroidal three-manifold $M$ makes
$\pi_1(M)$ act faithfully on the circle.
This result came almost two decades after
Eisenbud, Hirsch, and Neumann \cite{EHN}
encountered a complementary phenomenon:
they proved 
that an oriented three-manifold $M$ Seifert fibered over $S^2$
admits a co-oriented foliation transverse to the fiber
if and only if $\pi_1(M)$ admits a representation in
$\widetilde{\mathrm{Homeo}_+}S^1$
sending $\phi \mkern-2mu\mapsto\mkern-2mu \mathrm{sh}(1)$, where
$\widetilde{\mathrm{Homeo}_+}S^1$ is the
universal cover of the group of 
orientation-preserving homeomorphisms of the circle, 
$\phi$ is the regular fiber class,
and $\mathrm{sh}(s) \mkern-1.5mu:\mkern-1.5mu t \mapsto t\mkern-2mu+\mkern-2mu s$ for any 
$s \mkern-2mu\in\mkern-2mu {\mathbb{R}}$,
making $\widetilde{\mathrm{Homeo}_+}\mkern-1muS^1$ the centralizer of
$\mathrm{sh}(1)$ in $\mathrm{Homeo}_+{\mathbb{R}}$.

\subsection{Jankins-Neumann Classification}
Inspired by this observation,
Jankins and Neumann 
used Poincar{\'e}'s ``rotation number'' invariant
to generalize the criterion of \cite{EHN} to a more local
representation-theoretic condition in terms of
meridians of exceptional fibers.  This new formulation of the problem,
in addition to a correct conjecture that the 
necessary representation-theoretic conditions could be met in
$\widetilde{\mathrm{Homeo}_+}S^1$
if and only if they could also be met in a smooth Lie subgroup thereof,
allowed them to work out a complete, explicit classification \cite{JankinsNeumann},
which they proved in all but one special case, later proven by Naimi \cite{Naimi}.

\begin{theorem}[\cite{JankinsNeumann, Naimi}]
For $n>1$, the manifold $M_{S^2}(y_0; y_1, \ldots, y_n)$ 
Seifert fibered over $S^2$
admits a co-oriented taut foliation if and only if
$\mkern.8mu0 = y_-\mkern-2.5mu = y_+\mkern-2.5mu$ or 
$\mkern.8mu0 \in \left<y_+, y_-\right>$, where
\begin{equation*}
y_-
:=
\max_{k>0}
-\mfrac{1}{k}\mkern-1.8mu\left(\mkern2mu
1
+ {{\textstyle{\sum\limits_{i=0}^n  \left\lfloor y_i k \right\rfloor}}} \right),
\;\;\;\;\;
y_+ 
:=
\min_{k>0}
-\mfrac{1}{k}\mkern-1.8mu\left(\mkern-3mu
-1
+ {{\textstyle{\sum\limits_{i=0}^n   \left\lceil y_i k \right\rceil}}} \right).
\end{equation*}
\end{theorem}

\vspace{-.15cm}
{\noindent{(In the above, and henceforth in this paper, we always regard $k$ as an integer, writing $k>0$ as shorthand for the restriction $k \in \mathbb{Z}_{>0}$.)}} 

\vspace{.1cm}
Since then, the Jankins-Neumann-Naimi classification 
has served as a Rosetta stone for certain
{\em a priori} unrelated properties.
\begin{theorem}[\cite{EHN, JankinsNeumann, Naimi, 
ziggurat, OSGen, LiscaMatic, LSIII, BGW, lslope}]
If $M$ is a closed oriented Seifert fibered space, then the following are equivalent:
\begin{enumerate}
\item[(1a)]
$\pi_1(M)$ admits a non-trivial representation in $\mathrm{Homeo}_+{\mathbb{R}}$.
\item[(1b)]
$\pi_1(M)$ is left-orderable.
\item[(2)]
$M$ admits a co-oriented $C^{\mkern1mu 0}\mkern-2mu$ taut foliation.
\item[(3)]
$M$ has non-trivial Heegaard Floer homology, {\em{i.e.}}, $M$ fails to be an L-space.
\end{enumerate}
\end{theorem}

One often uses $(1a)$ as a proxy for $(1b)$, since
a result of Boyer, Wiest, and Rolfsen \cite[Theorem 1.1.1]{BRW},
combined with the well-known fact
\cite{LinnellLO} that the
set of countable left-orderable groups coincides
with the set of countable nontrivial subgroups of 
$\mathrm{Homeo}_+{\mathbb{R}}$, shows that 
$(1a) \!=\! (1b)$ for every prime compact oriented three-manifold.
Boyer, Gordon, and Watson have conjectured that 
$(1) \!=\! (3)$
for any prime compact oriented three-manifold \cite{BGW},
and quite recently, Kazez and Roberts \cite{KazezRobertsCzero},
and independently Bowden \cite{Bowden},
have extended a $C^2$ foliations result of 
Ozsv{\'a}th and Szab{\'o} \cite{OSGen} to show that
$(2) \!\Rightarrow\! (3)$ for any compact oriented three-manifold.
(For this reason, all foliations in this paper are assumed to be $C^0$ unless otherwise
stated.)

The implication $(3) \!\Rightarrow\! (2)$, however, is entirely more mysterious.
In particular, all known proofs 
\cite{LiscaMatic, LSIII, lslope}
that non-L-space oriented Seifert
fibered spaces admit co-oriented taut foliations rely on
an explicit comparison of sets of manifolds:
one works out the classification of Seifert fibered manifolds over
$S^2\!$ with non-trivial Heegaard Floer homology, and observes that this
classification coincides with the 
Jankins-Neumann-Naimi
classification of oriented
Seifert fibered spaces over $S^2\mkern-2mu$ admitting
co-oriented taut foliations.$\mkern-.7mu$
(The implication $\mkern-.5mu(3)\mkern-3.5mu\Rightarrow\mkern-4mu$~$(2)$ holds vacuously for closed oriented
Seifert fibered spaces with $b_1 \mkern-1.5mu>\mkern-1.5mu 0$, all of which
admit co-oriented taut foliations \cite{gabai}, and for oriented Seifert
fibered spaces over ${\mathbb{R}}{\mathbb{P}}^2\!$, all of which are L-spaces \cite{BGW}.)

\subsection{Graph Manifolds}

Boyer and Clay recently brought insight to this question by
introducing a relative version of the problem, 
studying the gluing behavior of properties $(1a)$, $(1b)$, and $(2)$
along the incompressible tori separating Seifert fibered components
of graph manifolds.  By showing that these three properties
glue in an identical manner along boundaries of JSJ components of
rational homology sphere graph manifolds, they were able to prove
the equivalence of these three properties for any closed
graph manifold \cite{BoyerClay}.
Boyer and Clay also conjectured that property $(3)$ should obey the same
gluing behavior.

In answer, Hanselman and Watson \cite{HanWat},
and independently J. Rasmussen and the author \cite{lslope}, 
were able to confirm this gluing conjecture
for a larger class of
three-manifolds with torus boundary, but subject to certain hypotheses,
which one can show are safe to remove in the case of graph manifolds.
The four of us \cite{HRRW} were therefore able to prove the following.
\begin{theorem}[\cite{HRRW}]
\label{thm: intro version of l = ntf}
A graph manifold is an L-space if and only if it fails
to admit a co-oriented taut foliation.
\end{theorem}

The current paper follows an independent trajectory
from the work of \cite{HRRW},
launched before the author joined the other collaboration.
Although the two papers overlap in one or two results,
including slightly variant proofs
of Theorem
\ref{thm: intro version of l = ntf}
and the below gluing criterion,
the main result of the current paper is the generalization of
the Jankins-Neumann classification formula to graph manifolds,
for which we now introduce some notation.

\begin{definition}
If $Y\!$ is a compact oriented three-manifold with torus boundary,
then the {\em{L-space interval}} of $Y$ is the space 
$\mathcal{L}(Y) \subset {\mathbb{P}}(H_1(\partial Y))$ of
L-space Dehn filling slopes of $Y\!$.
\end{definition}

We call $\mathcal{L}(Y)\mkern-.8mu$ an interval because if it contains more than
one point, then it is the intersection of ${\mathbb{P}}(H_1(\partial Y))$ with either
a closed interval in ${\mathbb{P}}(H_1(\partial Y;{\mathbb{R}}))$
or the complement of a single
point in ${\mathbb{P}}(H_1(\partial Y;{\mathbb{R}}))$.  
It therefore makes sense
to speak of the {\em{interior}} 
$\mathcal{L}^{\mkern-.4mu\circ}\mkern-1.5mu(Y)$ of $\mathcal{L}(Y)$.
If $\mathcal{L}^{\circ}(Y)$ is nonempty, we call $Y$ {\em{Floer simple}}.

\begin{prop}
\label{prop: intro gluing result for graph manifolds}
If $Y_1$ and $Y_2$ are non-solid-torus 
graph manifolds with torus boundary,
then the union $Y_1 \mkern-1mu\cup_{\varphi}\! Y_2$, with gluing map
$\varphi \mkern-1mu:\mkern-1mu \partial Y_1 \to -\partial Y_2$,
is an L-space if and only if
\begin{equation*}
\varphi_*^{{\mathbb{P}}}(\mathcal{L}^{\circ}\mkern-1.5mu(Y_1)) 
\cup \mathcal{L}^{\circ}\mkern-1.5mu(Y_2)
={\mathbb{P}}(H_1(\partial Y_2)).
\end{equation*}
\end{prop}
\vspace{-.1cm}
{{\noindent{In particular, Floer simplicity is not required.}}}

A graph manifold $Y$ with torus boundary and $b_1(Y) \mkern-1.5mu > \mkern-1.8mu 1$ has 
$\mathcal{L}(Y) = \emptyset$.  If $b_1(Y) \mkern-1mu = \mkern-1mu 1$, then the graph for
$Y$ is a tree, and we choose to root this tree at the
JSJ component $\hat{Y}$ containing $\partial Y$.
Writing $Y_1, \ldots, Y_{n_{\textsc{g}}}$ for the
$n_{\textsc{g}}$ components of 
$Y \setminus (\hat{Y} \setminus (\partial \hat{Y} \setminus \partial Y))$,
we then regard $Y$ as the union
\begin{equation*}
\label{eq: presentation of graph manifold in intro}
Y = \hat{Y} \cup_{\boldsymbol{\varphi}} {{\textstyle{\coprod\limits_{i=1}^{n_{{\textsc{g}}}} Y_i}}},
\;\;\;\;
\varphi_i : \partial Y_i \to -\partial_i \hat{Y},
\end{equation*}
with 
$\hat{Y}$ Seifert fibered over an 
$n_{{\textsc{g}}}\mkern-1mu+\mkern-.5mu1$-punctured $S^2$ or ${\mathbb{R}}{\mathbb{P}}^2$.
Note that each $Y_i$ is again a graph manifold with torus boundary and $b_1\!=\!1$,
hence is endowed with its own tree graph rooted at the JSJ component containing $\partial Y_i$,
but with the height of this tree strictly less than the height of the tree for $Y$,
so that a recursive computation of $\mathcal{L}(Y)$ in terms of
the $\mathcal{L}(Y_i)$ is a finite process.

For any (necessarily toroidal) boundary component of an oriented Seifert fibered space,
we fix the reverse-oriented homology basis 
$(\tilde{f}\mkern-2.8mu, \mkern-2.2mu-\tilde{h})$,
where $\tilde{h}$ is the meridian of the excised regular fiber,
and $\tilde{f}$ is the lift dual to $\tilde{h}$ of the regular fiber class,
so that we can express any slope 
$r\mkern-.5mu\tilde{f} - s\tilde{h} \in {\mathbb{P}}(H_1(\partial Y))$
as $\frac{r}{s} \in {\mathbb{Q}} \cup \{\infty\}$.
For any $Y_i$ with nonempty $\mathcal{L}(Y_i)$, we then write
\begin{equation*}
\varphi_{i*}^{{\mathbb{P}}}(\mathcal{L}(Y_i)) =:
\begin{cases}
\mkern1mu[[y_{i-}^{\textsc{g}}, y_{i+}^{\textsc{g}}]]
& 
\mathcal{L}^{\circ}\mkern-1.5mu(Y_i) \neq \emptyset
   \\
\mkern-.5mu\{y_{i-}^{\textsc{g}}\} \mkern-1.5mu=\mkern-1.5mu \{y_{i+}^{\textsc{g}}\}
& 
\mathcal{L}^{\circ}\mkern-1.5mu(Y_i) = \emptyset,
\end{cases}
\end{equation*}
where we use the notation 
$[[y_-, y_+]] \mkern-1mu\subset\mkern-1mu {\mathbb{Q}} \mkern1mu\cup\mkern-1mu \{\infty\}$
to denote the L-space interval with left-hand endpoint $y_-$
and right-hand endpoint $y_+$,
since any L-space interval with nonempty interior
is uniquely specified by its endpoints.

We also write $y_0^{\textsc{d}}, \ldots, y_{n_{{\textsc{d}}}}^{\textsc{d}}$
for the Seifert data of $\hat{Y}$, so that
$\hat{Y}$ is the
complement of $n_{{\textsc{g}}}\!+\!1$ regular fibers in either 
$M_{S^2}(y_0^{\textsc{d}};y_1^{\textsc{d}}, \ldots, y_{n_{{\textsc{d}}}}^{\textsc{d}})$ or
$M_{{\mathbb{R}}{\mathbb{P}}^2}(y_0^{\textsc{d}};y_1^{\textsc{d}}, \ldots, y_{n_{{\textsc{d}}}}^{\textsc{d}})$,
depending on whether $\hat{Y}$ has orientable or non-orientable base.
(These $y_i^{\textsc{d}}$ can also be regarded as Dehn-filling slopes
in terms of the basis ($\tilde{f}_i^{\textsc{d}}, -\tilde{h}_i^{\textsc{d}})$
described above.
See Section~\ref{ss: Conventions for Seifert fibered spaces}
for notation and homology conventions for Seifert fibered spaces.)
We can now state our main result.

\begin{theorem}
\label{thm: intro version of JN graph}
Suppose that $Y$ is not a solid torus and that $\mathcal{L}(Y)$ is nonempty.  Then
\begin{equation*}
\mathcal{L}(Y) =
\begin{cases}
\left< -\infty, +\infty \right>
& 
\hat{Y}\;
\text{has non-orientable base}
   \\
[[y_-, y_+]]
& 
\hat{Y}\;
\text{has orientable base},\;
\mathcal{L}^{\circ}(Y) \neq \emptyset
    \\
\{y_-\} \mkern-1.5mu=\mkern-1.5mu \{y_+\}
&
\hat{Y}\;
\text{has orientable base},\;
\mathcal{L}^{\circ}(Y) = \emptyset,
\end{cases}
\end{equation*}
where
\begin{align*}
y_-
&:=\,
\max_{k > 0}
-\mfrac{1}{k}\mkern-2.5mu\left( 1
+ {{\textstyle{\sum\limits_{i=0}^{n_{{\textsc{d}}}}  \left\lfloor y^{\textsc{d}}_i k \right\rfloor}}}
+ {{\textstyle{\sum\limits_{i=1}^{n_{{\textsc{g}}}} \left(  \left\lceil y^{\textsc{g}}_{i+}k 
\right\rceil - 1 \right)}}}
\right),
\\ \nonumber
y_+ 
&:=\,
\min_{k > 0}
-\mfrac{1}{k}\mkern-2.5mu\left(\mkern-2mu -1
+ {{\textstyle{\sum\limits_{i=0}^{n_{{\textsc{d}}}} \left\lceil y^{\textsc{d}}_i k \right\rceil}}}
+ {{\textstyle{\sum\limits_{i=1}^{n_{{\textsc{g}}}} \left(  \left\lfloor y^{\textsc{g}}_{i-} k 
\right\rfloor + 1 \right)}}}
\right).
\end{align*}
\end{theorem}
{\noindent{(In the above, we define $y_- := \infty$ or $y_+ := \infty$, respectively,
if any infinite terms appear as summands of $y_-$ or $y_+$, respectively.)}}

Whereas every oriented Seifert fibered space over the
disk or M{\"o}bius strip is Floer simple,
{\em{i.e.}} has $\mathcal{L}^{\circ}\mkern-3.7mu \neq \mkern-1mu \emptyset$,
the story for graph manifolds is more complicated.  Consider the following examples,
for all of which we take $\hat{Y}$ to have orientable base.
\begin{align*}
\mathcal{L}(Y) &= [-\infty, 96]:
&&
n_{{\textsc{d}}} \mkern-3mu=\mkern-2mu 3, \;\;
(y_1^{\textsc{d}}, y_2^{\textsc{d}}, y_3^{\textsc{d}})
= \left({{\textstyle{\frac{1}{3}}}}, \mkern-1mu -{{\textstyle{\frac{2}{5}}}},
{{\textstyle{\frac{3}{2}}}}\right);
  \nonumber \\
&&&
n_{{\textsc{g}}} \mkern-3mu=\mkern-2mu 2, \;\;
\varphi_{1*}^{{\mathbb{P}}}(\mathcal{L}(Y_1)) = [-100, +\infty],\;
\varphi_{2*}^{{\mathbb{P}}}(\mathcal{L}(Y_2)) =
\left[\left[{{\textstyle{\frac{2}{5}}}}, -20\right]\right].
  \nonumber \\
\mathcal{L}(Y) &= \{0\}:
&&
n_{{\textsc{d}}} \mkern-3mu=\mkern-2mu 1, \;\;
y_1^{\textsc{d}} = {{\textstyle{\frac{1}{3}}}};
    \\
&&&
n_{{\textsc{g}}} \mkern-3mu=\mkern-2mu 1, \;\;
\varphi_{1*}^{{\mathbb{P}}}(\mathcal{L}(Y_1)) = \left[-{{\textstyle{\frac{1}{3}}}}, 0\right].
 \nonumber \\
\mathcal{L}(Y) &= \emptyset:
&&
n_{{\textsc{d}}} \mkern-3mu=\mkern-2mu 3, \;\;
(y_1^{\textsc{d}}, y_2^{\textsc{d}}, y_3^{\textsc{d}})
= \left({{\textstyle{\frac{1}{3}}}}, \mkern-1mu -{{\textstyle{\frac{2}{5}}}},
{{\textstyle{\frac{3}{2}}}}\right);
  \nonumber \\
&&&
n_{{\textsc{g}}} \mkern-3mu=\mkern-2mu 2, \;\;
\varphi_{1*}^{{\mathbb{P}}}(\mathcal{L}(Y_1)) = [-100, +\infty],\;
\varphi_{2*}^{{\mathbb{P}}}(\mathcal{L}(Y_2)) = \left[-{{\textstyle{\frac{1}{3}}}}, 0\right].
\end{align*}

Above, we see examples in which
$Y$ is Floer simple, has an isolated L-space filling, 
or has empty L-space interval.
One cannot use
Theorem~\ref{thm: intro version of JN graph} without first knowing
which of these three cases occurs for $Y$.
We therefore provide 
Proposition~\ref{prop: characterization of Floer simple graph manifolds},
which lists explicit criteria for the
multiple mutually exclusive cases in which $Y$ is Floer simple
or in which $Y$ has an isolated L-space filling.
In the complement of these criteria, $\mathcal{L}(Y)$ is empty.

\vspace{.07cm}
In fact, the validity of 
Theorem~\ref{thm: l space interval for graph manifolds}
extends beyond the realm of graph manifolds.
\vspace{-.1cm}
\begin{prop}
\label{prop: intro version of JN graph theorem applicable for non graph}
Theorem 
\ref{thm: intro version of JN graph}
holds for any boundary incompressible Floer simple three-manifolds
$Y_1, \ldots, Y_{n_{{\textsc{g}}}}$, provided that
$Y$ satisfies the criteria in
Proposition~\ref{prop: characterization of Floer simple graph manifolds}
for $\mathcal{L}(Y)$ to be nonempty.
\end{prop}

\vspace{-.1cm}
{\noindent{One immediate application of this generalization
is the computation of L-space intervals for
cables of Floer simple knot complements.}}

\subsection{Cables}
The $(p,q)$-cable $Y^{(p,q)}\mkern-2mu\subset\mkern-2mu X$ of a knot complement 
$Y \mkern-2mu:=\mkern-2mu X \mkern-1.5mu\setminus\mkern-1.5mu \nu(K) \subset X$
is given by 
the knot complement $Y^{(p,q)} \mkern-3mu:=\mkern-2mu 
X \mkern-1.5mu\setminus\mkern-1.5mu \nu(K^{(p,q)})$,
where $K^{(p,q)}\mkern-1mu \subset X$ is the image of 
the $(p,q)$-torus knot embedded in the boundary of $Y\!$.
Since one can realize any cable of $Y \subset X$ by gluing 
an appropriate Seifert fibered space onto $Y\!$,
we can use the above generalization of
Theorem~\ref{thm: intro version of JN graph}
to prove the following result.
\begin{theorem}
\label{thm: intro cabling theorem for L-space complements}
Suppose that $p, q \in {\mathbb{Z}}$ with $p>1$ and $\gcd(p,q) = 1$,
and that $Y = X\setminus \nu(K)$ is a boundary incompressible Floer simple knot complement
in an L-space $X$, 
with L-space interval 
$\mathcal{L}(Y) = [[\frac{a_-}{b_-}, \frac{a_+}{b_+}]]$, written in
terms of the surgery basis $\mu, \lambda \in H_1(\partial Y)$ for $K$, with 
$\mu$ the meridian of $K$ and $\lambda$ a choice of longitude.
Then in terms of the surgery basis produced by cabling,
the $(p,q)$-cable $Y^{(p,q)} \mkern-2mu\subset\mkern-2mu X$ of 
$Y \mkern-2mu\subset\mkern-2mu X$
has L-space interval
\begin{equation*}
\mathcal{L}(Y^{(p,q)}) =
\begin{cases}
\{\infty\}
&\mkern15mu
\mfrac{a_-\mkern-2mu}{b_-\mkern-2mu} \in
\left[\mfrac{p^*}{q^*}, \infty\right],\;
\mfrac{a_+\mkern-2mu}{b_+\mkern-2mu} \in
\left[\mfrac{q-p^*}{p-q^*}, \mfrac{q}{p} \right> \cup \{\infty\}
     \\
[[1/y_+, 1/y_-]]
&\mkern15mu
\text{otherwise},
\end{cases}
\end{equation*}
where $p^*, q^* \in {\mathbb{Z}}$ are defined to satisfy
$pp^* - qq^* = 1$ with $0 < q^* < p$,
and where we define 
$y_- := \max_{k>0} y_-(k)$
and
$\mkern3mu y_+ := \min_{k>0} y_+(k)$, with
\begin{align*}
y_-(k)
:= \mfrac{1}{k}\mkern-3mu\left(
\left\lceil
\mkern-2mu\mfrac{\mkern1.5muq^*\mkern-3.5mu}{p}k\mkern-1.5mu
\right\rceil
-
\left\lceil y_{1+}k\mkern-1.5mu
\right\rceil
\right),
\;\;\;
y_+(k)
:= \mfrac{1}{k}\mkern-3mu\left(
\left\lfloor
\mkern-2mu\mfrac{\mkern1.5muq^*\mkern-3.5mu}{p}k\mkern-1.5mu
\right\rfloor
-
\left\lfloor
y_{1-}k\mkern-1.5mu
\right\rfloor
\right),
   \\
\text{and}\;\;\;\;
y_{1\pm} := \frac{a_{\pm}q^* - b_{\pm}p^*}{a_{\pm}p - b_{\pm}q}
=\mkern1mu \frac{\mkern.5muq^*\mkern-2.5mu}{p} 
\mkern-5mu\left(1 - \frac{b_{\pm}}{q^*(a_{\pm}p-b_{\pm}q)}\right).
\;\;\;\;\;\;
\end{align*}
\end{theorem}
{\noindent{We also prove a slightly more general version of
Theorem~\ref{thm: intro cabling theorem for L-space complements}
which does not require that
$X$ be an L-space, and which
holds for any $\frac{p}{q} \in {\mathbb{Q}} \mkern-1mu\cup\mkern-3mu \{\infty\}$.}}

A brief application of
the theorem,
followed by an appropriate change of basis,
recovers the following result of 
Hedden \cite{Heddencableii} and Hom \cite{Homcable}.
\begin{cor}
Suppose $Y := S^3 \setminus \nu(K)$ is a boundary incompressible
Floer simple positive  knot complement in $S^3$.
If $p> 0$ and $\gcd(p,q)=1$, then in terms of the
conventional basis for knot complements in $S^3$,
$Y^{(p,q)}$ has L-space interval
\begin{equation}
\mathcal{L}(Y^{(p,q)}) = 
\begin{cases}
\{\infty\}
&2g(K) - 1 >
\frac{q}{p}
   \\
[\mkern.5mu pq \mkern-1.5mu-\mkern-1.5mu p 
\mkern-1.5mu-\mkern-1.5mu q \mkern-1.5mu+\mkern-1.5mu 2g(K)p , \;\infty\mkern.2mu]
&2g(K) - 1 \le
\frac{q}{p}.
\end{cases}
\end{equation}
\end{cor}
{\noindent{By {\em{positive}}, we simply mean that $Y$ has an L-space Dehn filling $Y(N)$
for some $N> 0$.}}

Note that equating $pq-p-q+2g(K)p$ with $2g(K^{(p,q)})-1$
recovers the formula for the 
genus of the $(p,q)$-cable of $K \subset S^3\mkern-1mu$.
Note also that since $\frac{p^*}{q^*} - \frac{q}{p} = \frac{1}{pq^*}$,
with $\frac{p^*}{q^*}, \frac{q}{p} \notin {\mathbb{Z}}$, the domain
specified for Floer simple cables in the above corollary is equivalent to the condition 
$2g(K) - 1  \notin [\frac{p^*}{q^*}, \infty]$,
matching
Theorem~\ref{thm: intro cabling theorem for L-space complements}.

\subsection{Generalized Solid Tori}
A recent result of Gillespie \cite{Gillespietorus}
states that a compact oriented three manifold $Y$ with torus boundary
satisfies $\mathcal{L}(Y) = {\mathbb{P}}(H_1(\partial Y)) \setminus \{l\}$
if and only if $Y$ has genus 0 and an L-space filling,
where $l$ denotes the rational longitude of $Y$.
Such manifolds are called {\em{generalized solid tori}}
in \cite{lslope} and are of independent interest
\cite{BergeST, GabaiTorus, Cebanu, BBL, HanWat}.

Using the
version of 
Theorem~\ref{thm: intro cabling theorem for L-space complements}
that does not require $X$ to be an L-space,
along with some incremental results from the proof
of 
Theorem~\ref{thm: intro version of JN graph}, we are able to show the following.
\begin{theorem}
\label{thm: intro version of cabled generalized tori}
If $Y$ is a generalized solid torus,
then any cable of $Y \subset Y(l)$ is a generalized solid torus.
If $Y\!$ is a graph manifold with torus boundary, $b_1(Y) \mkern-2.5mu=\mkern-2.5mu 1$, and rational
longitude other than the regular fiber, then
$Y$ is a generalized solid torus if and only if
it is homeomorphic to an iterated
cable of the regular fiber complement in $S^1\mkern-2.5mu\times\mkern-.8mu S^2\mkern-4mu$.
\end{theorem}
Similarly, for any class of manifolds for which the gluing result in
Proposition~\ref{prop: intro gluing result for graph manifolds}
holds without the requirement of Floer simplicity---such as graph manifolds---one
has the result that if $Y$ has an isolated L-space filling, {\em{i.e.}},
if $\mathcal{L}(Y) = \{\mu\}$ for some $\mu \in {\mathbb{P}}(H_1(\partial Y))$,
then any cable of $Y \subset Y(\mu)$ has $Y(\mu)$ as an isolated L-space filling.

\subsection{Floer simple knot complements}
Whereas the regular fiber complement in a rational homology sphere
Seifert fibered space could arguably be called the prototypical
Floer simple manifold, not every regular fiber complement in an
L-space graph manifold is Floer simple,
due to the existence of isolated L-space fillings.
However, the next best thing is true.

Given a closed graph manifold $X$, call an 
exceptional fiber $f_{\textsc{e}} \mkern-1.5mu\subset\mkern-2.5mu X$
{\em{invariantly exceptional}} if the JSJ component $\hat{Y} \subset X$
containing $f_{\textsc{e}}$ has more than one exceptional fiber.
Note that if $\hat{Y}$ has only one exceptional fiber,
then $\hat{Y}$ is either a lens space (if $X$ is Seifert fibered)
or a punctured solid torus.
Since the solid torus
has nonunique Seifert structure, one can show that if $X$ is not a lens space, 
it is homeomorphic
to a graph manifold $X'$ in which the image
$f_{\textsc{e}}'$ of $f_{\textsc{e}}$ is a regular fiber.
excluding this scenario allows us to show the following.

\begin{theorem}
\label{thm : intro exceptional fiber}
Every invariantly exceptional fiber complement 
in an L-space graph manifold is Floer simple.
\end{theorem}
There are also Floer simple knot complements
traversing the graph structure of $X$.
\vspace{-.04cm}
\begin{prop}
\label{prop: intro incompressible floer simple knot complements graph}
If $X$ is an L-space graph manifold, then for every incompressible
torus $T \subset X$, there is a knot $K \subset X$ transversely
intersecting $T$ for which the complement $X \setminus \nu(K)$
is Floer simple.
\end{prop}
\vspace{-.04cm}
{\noindent{The same occurs for an arbitrary L-space $X$,
provided that $X$ decomposes
as a union of Floer simple manifolds along $T$
(see 
Proposition~\ref{prop: floer simple knot across incompressible torus}).}}

\vspace{.05cm}
The above results, together with the evidence of various other classes of L-spaces,
and a certain degree of optimism, motivate the following:
\begin{conj}
Every L-space admits a Floer simple knot complement.
\end{conj}

\subsection{Organization}
In Section~\ref{s: Foliations on Seifert fibered spaces},
we introduce our conventions for Seifert fibered spaces
and provide a lengthy discussion of the Jankins-Neumann
problem, since we cannot hope for 
Theorem~\ref{thm: intro version of JN graph}
to provide insight if
the original theorem of Jankins, Neumann and Naimi is opaque to
the reader.

Section~\ref{s: L-space intervals}
reviews some basic facts about L-space intervals,
including the independent results of Hanselman and Watson \cite{HanWat}
and J. Rasmussen and the author \cite{lslope}
about L-space criteria for unions of Floer simple manifolds.

Section~\ref{s: main results section}
is where we prove
our main graph manifold results,
including
Theorems~\ref{thm: intro version of l = ntf}
and
\ref{thm: intro version of JN graph}
in the forms of
Theorems~\ref{thm: lntf equivalence}
and
\ref{thm: l space interval for graph manifolds}.
This section also derives
Proposition~\ref{prop: characterization of Floer simple graph manifolds}'s
classification of 
single-boundary-component graph manifolds
with nonempty L-space intervals.

Our main cabling results reside in 
Section~\ref{section: cabling},
although the proof of
Theorem~\ref{thm: intro version of cabled generalized tori},
for generalized solid tori,
is relegated to 
Section~\ref{s: Observations}.

Lastly,
Section~\ref{s: Observations}
justifies
Proposition~\ref{prop: intro version of JN graph theorem applicable for non graph}'s
generalization of
our Jankins-Neumann graph manifold result
to the union of a Seifert fibered space with Floer simple manifolds.
This final section also 
lists an array of applications of the paper's main results,
including the aforementioned generalized solid torus cabling result
and proofs of the
Floer simple knot complement results from
Theorem~\ref{thm : intro exceptional fiber}
and
Proposition~\ref{prop: intro incompressible floer simple knot complements graph}.

\subsection*{Acknowledgements}
I would like to thank 
Jonathan Bowden,
Thomas Gillespie,
Jake Rasmussen,
and Rachel Roberts
for helpful conversations,
Liam Watson and Jonathan Hanselman
for comments on an earlier draft,
and the referee for helpful suggestions.

\section{Foliations on Seifert fibered spaces}
\label{s: Foliations on Seifert fibered spaces}

A {\em{graph manifold}} is a prime compact
oriented three-manifold which admits a JSJ decomposition---which
in this case, we take to be a 
{\em{minimal}} cutting apart along incompressible tori into 
disjoint pieces---such 
that each JSJ component is an oriented Seifert fibered space.
The data for reassembling these components into the original manifold
are encoded in a labeled graph, where each vertex corresponds to a Seifert fibered
JSJ component, and each edge corresponds to a gluing of two Seifert fibered
pieces along an incompressible torus.

\subsection{Restricting taut foliations to JSJ components}
\label{Restricting taut foliations on graph manifolds to Seifert fibered spaces}

Questions about taut foliations on graph manifolds
can often be reduced to questions about taut foliations on Seifert
fibered spaces, due in part to the following result.
\begin{prop}[\cite{Roussarie, Thurston, BrittenhamRoberts}]
\label{prop: restrict to separated JSJ components}
If $Y$ is a compact oriented three-manifold
admitting a taut foliation $F$ transverse to $\partial Y$,
then every incompressible separating torus in $Y$
can be isotoped to be everywhere transverse to $F$.
\end{prop}
\begin{proof}
Roussarie \cite{Roussarie} showed
that if $F$ is $C^2$,
then each incompressible torus $T \subset Y$ can be isotoped to
be either everywhere transverse to $F$ or a leaf of $F$.
A later theorem of
Brittenham and Roberts \cite{BrittenhamRoberts}
extends the validity of this proposition
to $C^0$ foliations.
Thus,
since a taut foliation
has no compact separating leaves,
an incompressible separating torus cannot be isotoped 
to be a leaf of $F$,
and so it must be possible to
isotop any incompressible separating torus to be everywhere transverse to $F$.
This is also believed to have been known by Thurston \cite{Thurston}.
\end{proof}

As noted by Brittenham, Naimi, and Roberts \cite{BNR},
this result has major consequences for graph manifolds:
\begin{cor}
\label{cor: taut tree manifold foliations restrict}
If $Y\!$ is a graph manifold with tree graph
and $F$ is a taut foliation on $Y$ transverse to $\partial Y$,
then $F$ can be isotoped so that it restricts to 
boundary-transverse taut foliations
on the Seifert fibered JSJ components of $Y$.
\end{cor}

When a closed graph manifold has positive first Betti number,
the question of existence of taut foliations becomes trivial,
since a result of Gabai states that any such manifold
admits a co-oriented taut foliation \cite{gabai}.
Correspondingly, any closed oriented three manifold with $b_1 \!>\! 0$
has non-trivial Heegaard Floer homology, hence is not an L-space.
We therefore restrict attention to rational homology sphere graph manifolds,
hence to
oriented Seifert fibered spaces over $S^2\mkern-1mu$ or $\mkern.5mu{\mathbb{R}}{\mathbb{P}}^2\!$,
and regular fiber complements thereof.

\subsection{Conventions for Seifert fibered spaces}
\label{ss: Conventions for Seifert fibered spaces}
If $\hat{M}$ denotes the trivial circle fibration over the
$n+1$-punctured two-sphere,
\vspace{-.1cm}
\begin{equation}
\hat{M} := S^1 \times \left(S^2 \setminus {\textstyle{\coprod\limits_{i=0}^n D^2_i}}\right),
\;\;\;\;\;\;
\partial_i \hat{M} := -\partial (S^1 \times D^2_i),\;\;i \in \{0, \ldots, n\},
\\[-3pt]
\end{equation}
then writing $-\tilde{h}_i: \in H_1(\partial_i \hat{M})$
for the meridian of each excised solid torus $S^1 \times D^2_i$, we have
\vspace{-.1cm}
\begin{equation}
-{\textstyle{\sum\limits_{i=1}^n \tilde{h}_i}}
= p_{S^1} \times \partial(S^2 \setminus D^2_0) = p_{S^1} \times -\partial D^2_0
= \tilde{h}_0
\\[-3pt]
\end{equation}
for any point class $p_{S^1} \in H_0(S^1)$ of the circle fiber.
For each $i \in \{0, \ldots, n\}$, if we write
$\iota_i : H_1(\partial_i \hat{M}) \to H_1(\hat{M})$
for the map induced by inclusion, 
then there is a lift 
$\tilde{f}_i \in \iota_i^{-1}(f)$ 
of the regular fiber class $f \in H_1(\hat{M})$ satisfying
$(-\tilde{h}_i \cdot \tilde{f}_i)|_{\partial_i \hat{M}} = 1$.
The reverse-oriented basis
$(\tilde{f}_i, -\tilde{h}_i)$
for $H_1(\partial_i \hat{M})$
induces a projectivization map
\begin{equation}
\label{eq: convention for slopes for Seifert fiber complement}
\pi_i : H_1(\partial_i \hat{M}) \setminus \{0\} \to
{\mathbb{P}}(H_1(\partial_i \hat{M})) \stackrel{\sim}{\to}
{\mathbb{Q}} \cup \{\infty\},\;\;\;
r_i \tilde{f}_i - s_i \tilde{h}_i \mapsto {\textstyle{\frac{r_i}{s_i}}},
\end{equation}
by which we identify Seifert invariants with slopes,
and slopes with ${\mathbb{Q}} \cup \{\infty\}$.

The Seifert fibered space
$M_{\mkern-1.5mu S^2\mkern-1.6mu}({\textstyle{\frac{r_*}{s_*}}}) \mkern-2.5mu:=\mkern-2.5mu
M_{\mkern-1.5mu S^2\mkern-1.6mu}(e_0\mkern-4mu=\mkern-4mu\frac{r_0}{s_0}; 
\frac{r_1}{s_{{\mkern-1mu}1}}, \ldots, \frac{r_n}{s_n})$ over $S^2\mkern-.8mu$
is the Dehn filling of $\hat{M}$ along the slopes
$\mu_i := r_i \tilde{f}_i - s_i \tilde{h}_i$, with 
$\pi_i(\mu_i) = \frac{r_i}{s_i}$,
subject to the convention that $e_0 = \frac{r_0}{s_0} \in {\mathbb{Z}}$
and $\frac{r_i}{s_i} \notin {\mathbb{Z}} \cup \{\infty\}$.
Setting each $h_i := \iota_i(\tilde{h}_i)$ then gives the presentation
\vspace{-.1cm}
\begin{equation}
H_{\mkern-1.5mu 1}\mkern-1.2mu(
M_{\mkern-1.5mu S^2\mkern-1.6mu}({\textstyle{\frac{r_*}{s_*}}}))
= \left< f, h_0, \ldots, h_n 
\mkern-2.5mu 
\vphantom{{{\textstyle{\sum\limits_{i=1}^n}}}}
\mkern3mu
\right|
\mkern-2mu
\left.
{{\textstyle{\sum\limits_{i=0}^n \mkern-2.2mu h_i}}}
\mkern-1mu=\mkern-1mu 
e_0\mkern-.7mu f \mkern-2.5mu-\mkern-1.8mu h_0 
\mkern-1mu=\mkern-1mu
r_1\mkern-.7mu  f \mkern-2.5mu-\mkern-1.8mu s_1\mkern-.7mu h_1 
\mkern-1mu=\mkern-1mu
\ldots
\mkern-1mu=\mkern-1mu
r_n\mkern-.7mu f \mkern-2.5mu-\mkern-1.8mu s_n \mkern-.7muh_n 
\mkern-1mu=\mkern-1.5mu 
0\mkern.5mu
\right>.
\\[-3pt]
\end{equation}

Likewise, if we respectively lift $f$ and each $h_i$ to generators
$\phi$ and $\eta_i$ for 
$\pi_1\mkern-.5mu(M_{\mkern-1.5mu S^2\mkern-1.6mu}({\textstyle{\frac{r_*}{s_*}}}))$
and substitute $\phi^{e_0}$ for $\eta_0$, then we obtain the 
fundamental group presentation
\begin{equation}
\label{eq: presentation of pi1 for S^2 Seifert fibered space}
\pi_1\mkern-.5mu
(M_{\mkern-1.5mu S^2\mkern-1.6mu}({\textstyle{\frac{r_*}{s_*}}}))
= \left< \phi, \eta_1, \ldots, \eta_n \mkern-3.5mu 
\vphantom{{{\textstyle{\sum\limits_{i=0}^n}}}}
\right.\left|\mkern2mu
\phi\text{ central,}\mkern5mu
{{\textstyle{\prod\limits_{i=1}^n\! \eta_i}}} \!=\! \phi^{-e_0}\mkern-2mu,\,
\eta_1^{s_1} \!=\! \phi^{r_1}\mkern-2mu,\, \ldots,\, \eta_n^{s_n} \!=\! \phi^{r_n}
\right>.
\end{equation}

For a manifold 
$M_{{\mathbb{R}}{\mathbb{P}}^{\mkern.3mu 2\mkern-1.4mu}}({\textstyle{\frac{r_*}{s_*}}})\mkern-.5mu$
Seifert fibered over ${\mathbb{R}}{\mathbb{P}}^2$,
we adopt the same homology and slope conventions for
the boundary of a regular fiber complement,
but the global homology is slightly different.
Since, this time, $\hat{M}$ is the {\it{twisted}} circle bundle over a
punctured ${\mathbb{R}}{\mathbb{P}}^2$, the fiber class $f$ is now 2-torsion.
Also, since puncturing ${\mathbb{R}}{\mathbb{P}}^2$
once gives a M{\"o}bius strip
instead of a disk, the sum $\sum_{i=1}^n\mkern-2mu{\tilde{h}_i}$
differs from $p_{{\mathbb{R}}{\mathbb{P}}^2} \times \partial({\mathbb{R}}{\mathbb{P}}^2 \setminus D^2_0)$
by twice the one-cell $c$ glued to the disk to make ${\mathbb{R}}{\mathbb{P}}^2$,
yielding a homology presentation of the form
\begin{equation}
H_{\mkern-1.5mu 1}\mkern-.5mu
(M_{{\mathbb{R}}{\mathbb{P}}^2\mkern-1.2mu}({\textstyle{\frac{r_*}{s_*}}})\mkern-.5mu)
= \left< f, c, h_0, \ldots, h_n 
\mkern.5mu|\mkern1.5mu
2f \mkern-2mu=\mkern-2mu 0,\mkern.7mu
2c \mkern-1mu+ \mkern-2.8mu{\textstyle{\sum_{i=0}^n 
\mkern-2mu h_i \mkern-1.5mu=\mkern-1.5mu 0}},\, 
\iota_0(\mu_0) = ... = \iota_n(\mu_n) \mkern-1mu=\mkern-1mu 0
\right>.
\end{equation}

For either type of base, 
the Seifert fibration is invariant under any reparameterization
$M(\frac{r_0}{s_0}, \ldots, \frac{r_n}{s_n})
\to
M(\frac{r_0}{s_0} + z_0 , \ldots, z_n + \frac{r_n}{s_n})$
with $\sum_{i=0}^n z_i = 0$ and each $z_i \mkern-2mu\in\mkern-2mu {\mathbb{Z}}$.
The manifold also admits 
orientation-reversing 
homeomorphism,
$M(\frac{r_0}{s_0}, \ldots, \frac{r_n}{s_n})
\to
M(-\frac{r_0}{s_0}, \ldots, -\frac{r_n}{s_n})$.

\vspace{.1cm}
A {\em{regular fiber complement}} 
$Y := M \setminus \nu(f)$ in a rational homology sphere
Seifert fibered space has $b_1(Y) = 1$,
hence has a well-defined rational longitude.
\vspace{-.15cm}
\begin{definition}
\label{def: rational longitude}
Any compact oriented three manifold $Y\!$ with torus boundary and $b_1(Y) \mkern-2mu=\mkern-2mu1$
has a {\em{rational longitude}}, a unique class $l \in {\mathbb{P}}(H_1(\partial Y))$
such that representatives in $H_1(\partial Y)$ have torsion image in $H_1(Y)$
under the homomorphism induced by inclusion of the boundary.
\end{definition}
\vspace{-.15cm}
{\noindent{It is straightforward to show (see, {\em{e.g.}}, \cite{lslope}) that}
\vspace{-.12cm}
\begin{equation}
{{\textstyle{M_{S^2}(\frac{r_*}{s_*})}}} \setminus \nu(f)\;\;\;
\text{has rational longitude}\;\;\;
l = {{\textstyle{-\sum\limits_{i=0}^n \frac{r_i}{s_i}}}}.
\\[-3pt]
\end{equation}
A mild generalization of the calculation in \cite{lslope}
shows that the above result also holds if each solid torus
$S^1 \mkern-2.5mu\times\mkern-2mu D^2_i$ is replaced with an arbitrary
compact oriented three-manifold $Y_i$ with torus boundary and
$b_1(Y_i)=1$, with $\frac{r_i}{s_i}$ the image of the rational longitude of $Y_i$.
By contrast, if the JSJ component containing $\partial Y$
has non-orientable base, then 
$l = \pi(\tilde{f}) = \infty$.

\vspace{.15cm}
{\noindent{\bf{Remark.}}
The requirement that $\frac{r_i}{s_i} \neq \infty$ for
$i \in \{0, \ldots, n\}$ 
is  a necessary (assuming $n\mkern-2mu>\mkern-2mu1$) and sufficient condition
for the resulting Seifert fibered space to be prime---an important property
for manifolds serving as building blocks
in combinatorial constructions.
To understand necessity,
let $\hat{M}^{\infty}$ denote
the result of Dehn filling $\hat{M}$ with slope $\infty$ along $\partial_0 \hat{M}$.
If $\hat{M}$
is fibered over a punctured $S^2$,
then $\hat{M}^{\infty}$ is a connected sum of $n$ solid tori,
each with longitude of slope $\infty$.
Similarly, if $\hat{M}$ is fibered over a punctured ${\mathbb{R}}{\mathbb{P}}^2$,
then $\hat{M}^{\infty}$ is the connected sum of an $S^1 \mkern-2mu\times\mkern-1mu S^2$
with $n$ solid tori, each with longitude of slope $\infty$.}

Primality is especially important in the context of foliations,
since, by Novikov \cite{Novikov},
no reducible manifold except $S^1 \mkern-2mu\times\mkern-1mu S^2$
admits a co-oriented taut foliation.  On the other hand,
not all connected sums are L-spaces,
so any correspondence between being an L-space and
failing to admit a co-oriented taut foliation
breaks down beyond the realm of prime manifolds.

\subsection{Rotation Number, Shift, and Foliation Slope}

One of the key insights of Jankins and Neumann into the work of
Eisenbud, Hirsch, and Neumann on taut foliations on Seifert fibered spaces
was the need for a better invariant
on $\widetilde{\mathrm{Homeo}_+} S^1$\!.  Whereas the latter group relied
on the invariants 
$\underline{m}, \overline{m} : \widetilde{\mathrm{Homeo}_+} S^1 \to {\mathbb{R}}$,
with $\underline{m}(\gamma) := \min_{t\in {\mathbb{R}}} \gamma(t) - t$ and
$\overline{m}(\gamma) := \max_{t \in {\mathbb{R}}} \gamma(t) - t$,
Jankins and Neumann introduced
the problem to a more precise invariant of circle actions:
a conjugacy invariant called the
(Poincar{\'e}) {\em{rotation number}}, 
\begin{equation}
\mathrm{rot} : \widetilde{\mathrm{Homeo}_+}(S^1) \to {\mathbb{R}},\;\;\;
\mathrm{rot}(\gamma) = \lim_{k \to \infty} \mfrac{1}{k}(\gamma^k(t) - t),
\end{equation}
which is independent of $t \in {\mathbb{R}}$, and rational if and only if
$\gamma$ has some closed orbit \cite{Ghys}.
The rotation number is not, in general, a homomorphism.
However, it restricts to a homomorphism on any 
amenable, hence any abelian, subgroup \cite{Ghys}.
In particular, it restricts to a homomorphism
on any representation of the fundamental group of a torus.

The simplest element of 
$\widetilde{\mathrm{Homeo}_+}(S^1)$ is a {\em{rotation}}, or {\em{shift}},
\begin{equation}
\mathrm{sh}(s) : t \mapsto t+s,\;\;\;t \in {\mathbb{R}}.
\end{equation}
Whereas $\mathrm{rot} \circ \mkern.4mu\mathrm{sh} \mkern-1mu=\mkern-1mu \mathrm{id}$,
not every element of $\widetilde{\mathrm{Homeo}_+}(S^1)$
is conjugate to a rotation.
It is a classic result, however, that every element of
$\widetilde{\mathrm{Homeo}_+}(S^1)$
with irrational rotation number is
left and right semiconjugate to a shift of the same rotation number
\cite{Ghys}.

Rotation numbers can also be
used to associate slopes to taut foliations on tori.

\begin{definition}
For the two-torus $T$, there is a canonical map
\begin{equation}
\alpha: 
\left\{
\;C^0\mkern-2mu\text{ codimension-one foliations on }T \;\right\}
\;\to\;
{\mathbb{P}}(H_1(T;{\mathbb{R}})),
\end{equation}
constructed below, which respects isotopy.  We call $\alpha(F)$
the {\em{slope}} of $F$.
\end{definition}

If $F$ has Reeb compenents, then
$\alpha(F)$ is given by the class of any closed leaf of $F$.
All Reebless foliations on tori are taut.
Thus, if $F$ is Reebless, then there is a curve,
say $C_{\lambda}$ of primitive class $\lambda \in H_1(T)$, 
which intersects every leaf transversely, and $F$ can be realized
as the suspension of a circle homeomorphism
$\gamma_{F , \lambda} \in \mathrm{Homeo}_+ S^1$
from $C_{\lambda}$ to itself \cite{HectorHirsch}.
A choice of $\mu \in H_1(T)$ with $\mu \cdot \lambda = 1$
induces a lift of this suspension
to a suspension from a universal cover $\tilde{C}_{\lambda}$
of $C_{\lambda}$ to its translate by 
$\mu$ in the universal cover of $T$.
That is, if we regard $\tilde{C}_{\lambda}$ as the real vector space
$\{t\lambda\}_{t\in{\mathbb{R}}}$ spanned by $\lambda$,
with $C_{\lambda} \cong \tilde{C}_{\lambda}/\lambda{\mathbb{Z}}$,
then one can lift the 
foliation $F$ to the universal cover of $T$ 
by iteratively suspending the map
$t\lambda \mapsto 
\tilde{\gamma}_{\mkern.2muF , \lambda , \mu}(t)\lambda + \mu$,
for an appropriate lift 
$\tilde{\gamma}_{F , \lambda , \mu} \in 
\widetilde{\mathrm{Homeo}}_+ S^1 \subset \mathrm{Homeo}_+ {\mathbb{R}}\mkern.8mu$
of $\mkern.5mu\gamma_{F , \lambda} \in {\mathrm{Homeo}}_+ S^1$.

This lifted suspsension, in turn, induces a representation
\begin{equation}
\rho^F_{\lambda} : \pi_1(T) \to \widetilde{\mathrm{Homeo}}_+ S^1,
\;\;\;\;
[\lambda] \mapsto \mathrm{sh}(1), 
\;\;\;\;
[-\mu] \mapsto \tilde{\gamma}_{F , \lambda , \mu},
\end{equation}
where $[\lambda]$ and $[\mu]$ denote the lifts of $\lambda$ and $\mu$
to $\pi_1(T)$.  One can regard $\rho^F_{\lambda}$ as describing how to traverse the
line $\{t\lambda\}_{t\in {\mathbb{R}}}  \subset \left<\mu, \lambda \right>_{{\mathbb{R}}}$
by traveling only along foliation leaves or integer multiples of $\lambda$ or $\mu$.
That is, if one starts at some $t_0\lambda$, hops by $a\lambda + b\mu$ for some
$a, b \in {\mathbb{Z}}$, takes the foliation leaf
intersecting this new point, and follows this leaf
back to the line $\{t\lambda\}_{t\in {\mathbb{R}}}$,
then one will arrive at $\rho^F_{\lambda}(a\lambda +b\mu)(t_0)\lambda$.
Note that while $\rho^F_{\lambda}$ is independent of the choice of $\mu$,
and is determined up to conjugacy by a choice of $\lambda$, it still
depends on $\lambda$.

On the other hand, when we define the slope $\alpha(F)$ of $F$ to be
\begin{equation}
\alpha(F) := \ker( \mathrm{rot} \circ \rho^F_{\lambda})
= \left<(\mathrm{rot}(\rho^F_{\lambda}\!(-\mu))\lambda 
+ \mu\right> \in {\mathbb{P}}(H_1(T;{\mathbb{R}})),
\end{equation}
then the rotation number washes out all dependence on $\lambda$ and
choice of suspension.
That is, one can use the definition of rotation number to
compute $\mathrm{rot}(\rho^F_{\lambda}\!(-\mu))$
in terms of the rotation number associated to a different
choice of basis and suspension for $F$,
and obtain the same answer for $\alpha(F)$ in both cases.
Alternatively, any suspension homeomorphism with rational rotation number
has a periodic orbit, hence realizes a foliation with a compact leaf
of slope $\alpha(F)$.  If the suspension homeomorphism has irrational
rotation number, then it is semiconjugate to a shift of 
matching rotation number \cite{Ghys},
giving rise to a linear foliation of slope $\alpha(F)$.

\subsection{Restricting Seifert fibered space foliations to torus foliations}
If a compact oriented three-manifold $Y$ admits a co-oriented
taut foliation transverse to $\partial Y$,
then Gabai tells us that $\partial Y$ 
can only have toroidal components \cite{gabai}.
Thus, we often encounter foliations on tori as boundary
restrictions of foliations on three-manifolds.
Moreover, on a Seifert fibered space, any
taut foliation transverse to the boundary
restricts to taut foliations on boundary components.

Suppose $F$ is a co-oriented taut foliation
transverse to the fibration of the
Seifert fibered Dehn filling
$M_{S^2\mkern-2mu}(\frac{r_*}{s_*})$
along the slopes $\frac{r_*}{s_*} = (\frac{r_0}{s_0}, \ldots, \frac{r_n}{s_n})$
of the trivial circle fibration $\hat{M}$ over an $n\!+\!1$-punctured $S^2$,
according to the conventions of Section \ref{ss: Conventions for Seifert fibered spaces}.
For each boundary component $\partial_i \hat{M}$,
we regard the foliation $F \cap \partial_i \hat{M}$ 
as a suspension of a homemorphism of the curve of class $\tilde{f}_i$
to itself, and since the class $-\tilde{h}_i$ satisfies
$-\tilde{h}_i \cdot \tilde{f}_i = 1$,
it specifies a lift of this suspension
to a suspension of an element 
$\gamma_{F,\tilde{f}_i, -\tilde{h}_i} \in \widetilde{\mathrm{Homeo}_+}S^1$.
To this suspension we associate the representation
\begin{equation}
\rho_i:= \rho^{F\mkern.3mu\cap\mkern.3mu \partial_i \mkern-1.2mu\hat{M}}_{\mkern-2.3mu\tilde{f}_i}
: \pi_1(\partial_i \hat{M}) \to \widetilde{\mathrm{Homeo}_+}S^1\mkern-2mu,\;\;\;\;
[\tilde{f}_i] \to \mathrm{sh}(1),\;\;\;\;
[\tilde{h}_i] \to \gamma_{F,\tilde{f}_i, -\tilde{h}_i},
\end{equation}
allowing us to express the slope $\alpha(F \cap \partial_i \hat{M})$ of
$F \cap \partial_i \hat{M}$ as
\begin{equation}
\alpha(F \cap \partial_i \hat{M}) 
= \pi_i( \mathrm{rot}(\rho_i([\tilde{h}_i]))\tilde{f}_i - \tilde{h}_i )
= \mathrm{rot}(\rho_i([\tilde{h}_i])).
\end{equation}

The construction of Eisenbud, Hirsch, and Neumann \cite{EHN}
associating a representation
$\rho : \pi_1(M(\frac{r_i}{s_i})) \to 
\widetilde{\mathrm{Homeo}_+}S^1$ to $F$, sending the fiber class
$\phi = [f]$ to $\mathrm{sh}(1)$,
is sufficiently compatible with the construction of each $\rho_i$ above
that, possibly after conjugation of each $\rho_i$, $\rho$ can be chosen 
to satisfy $\rho_i = \rho \circ \iota^{\pi_1}_i$, with
$\iota^{\pi_1}_i: \pi_1(\partial_i \hat{M}) \to \pi_1(M_{S^2\mkern-2mu}(\frac{r_*}{s_*}))$
the homomorphism induced by inclusion.
The presentation
(\ref{eq: presentation of pi1 for S^2 Seifert fibered space})
for $\pi_1(M_{S^2}(\frac{r_*}{s_*}))$
then places the following restrictions on $\rho$,
as observed by Jankins and Neumann \cite{JankinsNeumann}:
\begin{align}
{\textstyle{\prod_{i=1}^n \mkern-3.5mu\eta_i \mkern-1.3mu=\mkern-1.8mu \phi^{-e_0}}}
&\implies
\mkern14.5mu
\bullet\;
\mathrm{rot}\mkern-2mu\left(
{\textstyle{\prod_{i=1}^n \mkern-3.5mu\rho(\eta_i) \mkern-1.3mu}}
\right) = -e_0 = \textstyle{-\frac{r_0}{s_0}},
         \\
i \in \{0, \ldots, n\}\!:\;\; \eta_i^{s_i} \mkern-4mu=\mkern-2.2mu \phi^{r_i}
\mkern13.5mu
&\implies
\begin{cases}
\bullet\;
\mathrm{rot}(\rho(\eta_i)) = \textstyle{\frac{r_i}{s_i}},
     \\
\bullet\;
\rho(\eta_i)\;\text{is conjugate to}\;
\mathrm{sh}(\textstyle{\frac{r_i}{s_i}}).
\end{cases}
\end{align}
Jankins and Neumann mostly focused on the case of
$n= 3$, $e_0 = -1$, and $0 < \frac{r_1}{s_1}, \frac{r_2}{s_2}, \frac{r_3}{s_3} < 1$,
but their above observation holds in general.

Whereas the first condition enforces a global restriction on $F$,
the latter two conditions provide local restrictions
at each $F \cap \partial_i \hat{M}$,
which we could recover simply by considering Dehn fillings.
The solid torus admits only one taut foliation, namely,
the product foliation with slope given by the rational longitude.
As a consequence, the co-oriented taut foliation $F\cap \partial_i \hat{M}$ extends
to a co-oriented taut foliation on the Dehn filling $\partial_i \hat{M}(\frac{r_i}{s_i})$
if and only if $F\cap \partial_i \hat{M}$ is the product foliation
of slope $\frac{r_i}{s_i}$,
which occurs if and only if
$\rho_i([\tilde{h}_i])\!=\! \rho(\eta_i)$ is conjugate to $\mathrm{sh}(\frac{r_i}{s_i})$,
a condition which requires
$\alpha(F \cap \partial_i \hat{M}) = \mathrm{rot}(\rho(\eta_i)) = \frac{r_i}{s_i}$.
In particular,
the $j^{\mathrm{th}}$ shift conjugacy condition,
that $\rho(\eta_j)$ be conjugate to $\mathrm{sh}(\frac{r_j}{s_j})$,
is due solely to the fact that $\partial_j \hat{M}$
is glued to a solid torus. As emphasized by Boyer and Clay \cite{BoyerClay},
when one relaxes the $j^{\mathrm{th}}$ shift conjugacy condition,
one can still find manifolds $Y$ with torus boundary for which $F$
extends to a taut foliation on
$\hat{M} \cup_{\partial_j \mkern-1.5mu\hat{M}} \mkern-2muY\!$,
for a suitable choice of gluing map.

It is presumably for this reason that Jankins and Neumann focused on the
more general condition of $J$-{\it{realizability}}
for an $n\!+\!1$-tuple 
$\frac{r_*}{s_*} \mkern-2mu := \mkern-2mu (e_0 \mkern-4mu=\mkern-4mu \frac{r_0}{s_0}, \frac{r_1}{s_1}, \ldots, \frac{r_n}{s_n})$,
given a subset $J \subset \{1, \ldots, n\}$.
They deem $\frac{r_*}{s_*}$ $J$-realizable 
if there is
a representation $\rho^J : \left<\eta_1, \ldots \eta_n\right> \to
\widetilde{\mathrm{Homeo}_+}S^1$
such that $\rho^J$ meets the $\frac{r_*}{s_*}$ rotation number condition,
that $\mathrm{rot}\mkern-2mu\left(
{\textstyle{\prod_{i=1}^n \mkern-3.5mu\rho(\eta_i) \mkern-1.3mu}}
\right) \mkern-2mu=\mkern-2mu -e_0$
with
$\mathrm{rot}(\rho(\eta_i)) \!=\! \textstyle{\frac{r_i}{s_i}}$
for each $i \in \{1, \ldots, n\}$,
and such that
$\rho^J$ meets the $j^{\mathrm{th}}$ shift conjugacy condition
for each $j \in J$.
We have already shown that $J$-realizability is a necessary condition for 
$\hat{M}$ to admit a taut foliation $F$ of slopes
$\alpha(F\cap \partial_i \hat{M}) = \frac{r_i}{s_i}$
which extends to a taut foliation on the partial Dehn filling
of $\hat{M}$ along the slopes $\frac{r_j}{s_j} \in {\mathbb{P}}(H_1(\partial_j \hat{M}))$
for all $j \in \{0\} \cup J$.
With the help of Naimi, Jankins and Neumann showed that
the condition is also sufficient \cite{JankinsNeumann, Naimi}.

\subsection{Solutions for $J$-realizability}

Jankins and Neumann conjectured that a slope
$\frac{r_*}{s_*}$ is $J$-realizable in 
$\widetilde{\mathrm{Homeo}_+}S^1$
if and only if it is
$J$-realizable in a smooth Lie subgroup of $\widetilde{\mathrm{Homeo}_+}S^1\mkern-2mu$.
Observing that any smooth Lie subgroup of $\widetilde{\mathrm{Homeo}_+}S^1\mkern-2mu$
is conjugate to $\widetilde{PSL}_k(2,{\mathbb{R}})$ for some $k \!\in\! {\mathbb{Z}}_{>0}$,
where $\widetilde{PSL}_k(2,{\mathbb{R}}) \mkern-2mu=\mkern-2mu 
\psi_k^{-1} \mkern-.5mu \widetilde{PSL}(2,{\mathbb{R}}) \mkern.5mu \psi_k$
for $(\psi_k \mkern-3mu:\mkern-2.8mu t \mkern-2.2mu\mapsto\mkern-2.5mu kt) 
\mkern-1mu\in\mkern-1mu \mathrm{Homeo}_+{\mathbb{R}}$,
they computed
\begin{equation}
\label{eq: max rot equation, PSLk}
\max \;\mathrm{rot}\left(\prod_{i=1}^n \gamma_i\right)
=
\frac{1}{k}\left(-1 + \sum_{i=1}^n
\left(\left\lfloor \frac{r_i k}{s_i} \right\rfloor +1 \right) \right)
\end{equation}
for the maximum rotation number 
of a product of elements 
$\gamma_i \in \widetilde{PSL}_k(2,{\mathbb{R}})$
with each $\mathrm{rot}(\gamma_i) = \frac{r_i}{s_i}$.
They then proved the above conjecture in all but one case,
later proven by Naimi \cite{Naimi}.
More recently,
in Theorem 3.9 of \cite{ziggurat} (appropriately generalized from 2 to $n$),
Calegari and Walker rederived
(\ref{eq: max rot equation, PSLk})
(with the maximum taken over $k \!\in\! {\mathbb{Z}}_{>0}$)
for $\widetilde{\mathrm{Homeo}_+}S^1$,
without appealing to $\widetilde{PSL}_k(2,{\mathbb{R}})$,
by using dynamical techniques
similar to those of Naimi.

One obtains the analogous minimum rotation number of a product by
sending $\frac{r_i}{s_i} \mapsto -\frac{r_i}{s_i}$
in (\ref{eq: max rot equation, PSLk}).
Demanding that $-e_0$ lie between the minimum and maximum
rotation numbers for $\prod_{i=1}^n\rho(\eta_i)$,
and multiplying the resulting inequality by $-1$, implies that
a representation in $\widetilde{\mathrm{Homeo}_+}S^1$  
can only satisfy the rotation number condition for
$\frac{r_*}{s_*} = (e_0,\frac{r_1}{s_1\mkern-1.5mu}, \ldots, \frac{r_n}{s_n})$ if
\begin{equation}
\label{eq: negative of rot satisfiers}
\min_{k > 0}
-\mfrac{1}{k}\mkern-1.5mu\left(-1 + \sum_{i=1}^n\mkern-1mu
\left(\left\lfloor \mkern-1.5mu\mfrac{r_i k}{s_i}\mkern-1.5mu \right\rfloor +1 \right) \right)
\le
e_0
\le
\max_{k > 0}
-\mfrac{1}{k}\mkern-1.5mu\left(1 + \sum_{i=1}^n\mkern-1mu
\left(\left\lceil \mkern-1.5mu\mfrac{r_i k}{s_i}\mkern-1.5mu \right\rceil  -1\right) \right),
\end{equation}
a criterion which Jankins and Neumann prove is also sufficient \cite{JankinsNeumann}.
Moreover, $\frac{r_*}{s_*}$ is $J$-realizable in
$\widetilde{PSL}_{k}(2,{\mathbb{R}})$, for some $k \in {\mathbb{Z}}_{>0}$,
if and only if
\begin{equation}
-\mkern-.8mu\mfrac{1}{k}\mkern-3.5mu\left(\mkern-2mu -1 +\mkern-2.7mu
\sum_{i=0}^n\mkern-3.5mu
\left(\left\lfloor\mkern-2mu \mfrac{r_i k}{s_i} \mkern-2mu\right\rfloor 
\mkern-3mu+\mkern-2.1mu 1 \right)\mkern-3mu \right)
\le \mkern-.7mu
0 \mkern-.9mu
\le
-\mkern-3.8mu\sum_{i=0\mkern.5mu}^n\mkern-2mu
\mfrac{r_i}{s_i}
\mkern4mu\text{ or }\mkern2mu
-\mkern-2.8mu\sum_{i=0\mkern.5mu}^n\mkern-2mu
\mfrac{r_i}{s_i}
\le \mkern-.7mu
0 \mkern-.9mu
\le
-\mfrac{1}{k}\mkern-3.5mu\left(1 +\mkern-2.7mu
\sum\limits_{i=0}^n\mkern-3.5mu
\left(\left\lceil\mkern-2mu \mfrac{r_i k}{s_i} \mkern-2mu\right\rceil  
\mkern-3mu-\mkern-2.1mu 1\right)\mkern-2.5mu \right).
\end{equation}

The shift conjugacy condition is easier to apply:
one can approximate an element of 
$\widetilde{\mathrm{Homeo}_+}S^1$
with a shift-conjugate element of arbitrarily close
rotation number.
Jankins and Neumann used this fact to show that if
one fixes all $\frac{r_i}{s_i}$ with $i \neq j$, for some
fixed $j \in \{1, \ldots, n\}$,
then imposing the $j^{\mathrm{th}}$ shift conjugacy condition
is equivalent to restricting to the interior of the
interval of $\frac{r_j}{s_j} \in {\mathbb{R}}$ for which
the $\frac{r_*}{s_*}$ rotation number condition is satisfied.
More generally, Calegari and Walker have shown that the same
principle holds for the 
the rotation number condition associated to any 
fixed positive word 
\cite[Lemma 3.31]{ziggurat}.

Since for any $r \in {\mathbb{R}}$ and $z \in {\mathbb{Z}}$, we have
\begin{equation}
\label{eq: closed interval to ceiling floor}
z \le \lfloor r \rfloor  \;\iff\;  z \le r,
\;\;\;\;\;\;\;\;\;
\lceil r \rceil \le z  \;\iff\;   r  \le z,
\end{equation}
it follows, for any $j \in \{1, \ldots, n\}$,
that if we fix all $\frac{r_i}{s_i}$ with $i \neq j$,
then the interval of $\frac{r_j}{s_j} \in {\mathbb{R}}$ satisfying
(\ref{eq: negative of rot satisfiers}) is closed.
On the other hand, for any $r \in {\mathbb{R}}$ and $z \in {\mathbb{Z}}$, we know that
\begin{equation}
\label{eq: open interval to ceiling floor}
z \le \lceil r \rceil - 1   \;\iff\;   z < r,
\;\;\;\;\;\;\;\;\;
\lfloor r \rfloor + 1 \le z \;\iff\;  r < z.
\end{equation}
These identities led Jankins and Neumann
to produce the formulas in the following result.

\begin{theorem}[Jankins, Neumann, Naimi \cite{JankinsNeumann, Naimi}; 
$\mkern-3mu${\em{c.f.}}$\mkern-1.5mu$ Calegari, $\mkern-3mu$Walker \cite{ziggurat}]
\label{thm: J Seifert fibered foliation classification}
For any $n\mkern-1.8mu\ge\mkern-1.8mu 2$, partition 
$J \mkern-1.2mu\amalg\mkern-1.2mu \bar{J} 
\mkern-2.5mu=\mkern-2.5mu \{0, \ldots, n\}$ with 
$0 \mkern-2.3mu\in\mkern-2.2mu J\mkern-1.5mu$,
$\mkern-.3mu$and $n\mkern-1.7mu+\mkern-2mu 1$-tuple 
$\frac{r_*}{s_*} \!:=\!  (\frac{r_0}{s_0}, \ldots, \frac{r_n}{s_n}) 
\mkern-2mu\in\mkern-2mu {\mathbb{Q}}^{n+1}\mkern-2mu$
with $\frac{r_0}{s_0} \mkern-2.6mu\in\mkern-2.4mu {\mathbb{Z}}$
and $\frac{r_i}{s_i} \mkern-2mu\notin\mkern-2mu {\mathbb{Z}}$ for $i\mkern-2mu>\mkern-2mu0$,
the trivial circle fibration $\hat{M}$ over an $n+1$-punctured $S^2\!$
admits a co-oriented taut foliation $F\mkern-2mu$ transverse to the boundary, with slopes
$\alpha(F \cap \partial _i \hat{M}) \mkern-2.5mu =\mkern-2.5mu \frac{r_i}{s_i}$ for 
$i \mkern-2.5mu \in \mkern-2.5mu \{0, \ldots, n\}$,
and with $F\mkern-2.5mu$ extending to a co-oriented taut foliation on the Dehn filling of
$\mkern1mu\partial_j \hat{M}\mkern-2.5mu$ of slope $\frac{r_j}{s_j}$ for each 
$j \mkern-2.5mu\in\mkern-2.5mu J\mkern-1.7mu$,
if and only if 
$\mkern1mu 0 \mkern-.4mu=\mkern-.4mu y_- \mkern-3mu=\mkern-.4mu y_+\mkern-1.5mu$ or 
$\mkern1mu 0 \mkern-1.3mu \in \mkern-2.1mu \left<y_+, y_-\right>$, with
\vspace{-.1cm}
\begin{align}
\label{eq: defs of y- and y+}
y_-
&:=
\max_{1 \le k \le s}
-\frac{1}{k}\mkern-3mu\left(\mkern2mu
1
\mkern2mu+\mkern1mu \sum\limits_{j\in J} \! \left\lfloor \!\frac{r_j k}{s_j}\! \right\rfloor
\mkern2mu+\mkern1mu \sum\limits_{\bar{\jmath}\in\bar{J}}\!
\left(  \left\lceil \!\frac{r_{\bar{\jmath}} k}{s_{\bar{\jmath}}}\! 
\right\rceil - 1 \right)\!
\right),
     \\ \nonumber
y_+ 
&:=
\min_{1 \le k \le s}
-\frac{1}{k}\mkern-3mu\left(\mkern-3mu
-1
\mkern2mu+\mkern1mu \sum\limits_{j\in J}\!  
\left\lceil \!\frac{r_j k}{s_j}\! \right\rceil
\mkern2mu+\mkern1mu \sum_{\bar{\jmath}\in\bar{J}}\!
\left(  \left\lfloor \!\frac{r_{\bar{\jmath}} k}{s_{\bar{\jmath}}}\! 
\right\rfloor + 1 \right)\!
\right),
\end{align}
where $s$ is the least common positive multiple of the $s_i$.
\end{theorem}

\subsection{Dehn fillings and $\bar{N}$-fillings}
For any particular $i \in \{1, \ldots, n\}$ in the above theorem,
if one fixes the remaining slopes, one finds that the space
of slopes in ${\mathbb{P}}(H_1(\partial_i\hat{M}))$ for which the
desired taut foliation exists is often an interval.
We now introduce some notation to describe such spaces of slopes
in general.

\begin{definition}
If $\mkern.5muY\mkern-3.8mu$ is a compact oriented three-manifold with torus boundary,
then we define the sets
$\mathcal{F}^L\mkern-1.6mu(Y)
\mkern-3mu\subset\mkern-2.4mu
\mathcal{F}^D\mkern-2.2mu(Y) 
\mkern-3mu\subset\mkern-2.4mu
\mathcal{F}\mkern-.5mu(Y) 
\mkern-3mu\subset\mkern-2mu
{\mathbb{P}}(H_{\mkern-.2mu1}\mkern-.5mu(\partial Y \mkern-.6mu;{\mathbb{Z}})\mkern-1.5mu)\mkern-1.8mu$
of rational foliation slopes as follows:
\begin{align*}
\mathcal{F}^L\mkern-2mu(Y)
&:= \left\{ \alpha(F \cap \partial Y)\left|\mkern-4mu
\begin{array}{c}
F\text{ is a co-oriented taut foliation on }Y\!\text{ transverse to }
\partial Y,
   \\
\text{restricting to a rational co-oriented {\em{linear}} foliation on }
\partial Y
\end{array}
\mkern-10mu
\right.
\right\},
       \\
\mathcal{F}^D\mkern-2mu(Y)
&:= \left\{ 
\mkern39.2mu \mu \mkern39.2mu 
\left|
\mkern7mu
\text{The Dehn filling } 
Y\mkern-1.2mu(\mu)
\text{ admits a co-oriented taut foliation}\mkern5mu
\right.
\right\},
       \\
\mathcal{F}(Y)
&:= \left\{ \mkern3.5mu \alpha(F \cap \partial Y)\mkern1.5mu\left|
\mkern15mu F\text{ is a co-oriented taut foliation on }Y\!\text{ transverse to }
\partial Y
\mkern6mu
\right.
\right\}.
\end{align*}
\end{definition}
All linear foliations, even irrational ones, are taut, but rational linear foliations
are product foliations, hence extend to co-oriented
taut foliations on Dehn fillings of matching slope,
implying $\mathcal{F}^L(Y) \mkern-2mu\subset\mkern-2mu \mathcal{F}^D(Y) 
\mkern-2mu\subset\mkern-2mu \mathcal{F}(Y)$.
In fact,
the work of Jankins and Neumann
tells us that $\mathcal{F}^L \mkern-3mu=\mkern-2mu \mathcal{F}^D$
for manifolds Seifert fibered over the disk, and that the analogous result holds
for manifolds Seifert fibered over a punctured $S^2$.
Since the same also holds for manifolds Seifert fibered over a punctured
${\mathbb{R}}{\mathbb{P}}^2$ \cite{BoyerClay},
and since
Corollary~\ref{cor: taut tree manifold foliations restrict}
tells us that
taut foliations on homology sphere graph manifolds
isotop to restrict to taut foliations
transverse to boundaries on Seifert fibered JSJ components,
we additionally have
$\mathcal{F}^L \mkern-3mu=\mkern-2mu\mathcal{F}^D$
for any graph manifold with torus boundary and $b_1 = 1$.

In this latter case, it is natural to ask whether
$\mathcal{F}(Y)$ admits a description analogous to
the Dehn filling characterization for $\mathcal{F}^L\mkern-1.2mu(Y)$.
That is, can $\mathcal{F}(Y)$ be characterized in terms of
taut foliations on some closed union of $Y$ with some other manifold?
Boyer and Clay answer this question affirmatively
\cite{BoyerClay}, as we shall see.

Let $\bar{N}$ denote the regular fiber complement
\begin{equation}
\bar{N} := M_{S^2}(0, {\textstyle{-\frac{1}{2}, \frac{1}{2}}}) \setminus (S^1 \!\times\! D^2_0),
\end{equation}
which Boyer and Clay call the ``twisted $I$-bundle over the Klein bottle,''
or $N_2$.  The manifold $\bar{N}$ can play a role analogous to
that of the solid torus for Dehn fillings.
\begin{definition}
Suppose $Y$ is an oriented three-manifold with toroidal boundary component
$\partial_i Y$.
We call any union $Y \cup_{\varphi}\mkern-3mu \bar{N}$ with gluing map
$\varphi :  \partial \bar{N}  \to  -\partial_i Y$
an $\bar{N}${\em{-filling}} of $Y\!$ along $\mu \in {\mathbb{P}}(H_1(\partial_i Y))$,
where $\mu := \varphi_*^{{\mathbb{P}}}(l)$ is the image of the rational longitude $l$
of $\bar{N}$. If $Y$ has (single) torus boundary, we denote an
$\bar{N}$-filling of $Y$ along $\mu \in {\mathbb{P}}(H_1(\partial Y))$ by 
$Y^{\mkern-1mu\bar{N}}\mkern-2mu(\mu)$.

More generally,
if $\partial Y = \coprod_{i=1}^n \partial_i Y$ is a disjoint union of tori,
then given a slope
$\mu_* \mkern-2.5mu:\mkern-.4mu=\mkern-2mu 
(\mu_1, \ldots, \mu_n) \mkern-2mu\in\mkern-2mu \prod_{i=1}^n {\mathbb{P}}(H_1(\partial_i Y))$ and subset $\bar{J} \subset \{1, \dots, n\}$,
we denote by 
$Y^{\bar{N}}\mkern-2mu(\bar{J}; \mu_*)$ any manifold
resulting from 
$\bar{N}$-filling
$Y\!$ along
$\mu_{\bar{\jmath}}$ in $\partial_{\bar{\jmath}} Y\mkern-2.5mu$ 
for each $\bar{\jmath} \in \bar{J}$.
\end{definition}

We then have the following result for
$\bar{N}$-fillings.

\begin{prop}
\label{prop: boyer and clay's taut foliation slope vs N-filling}
Suppose $Y\!$ is a prime compact oriented manifold
with boundary a disjoint union 
$\coprod_{i=1}^n \partial_i Y\!$ 
of tori, with some given slope
$\mu_* \mkern-2.5mu:\mkern-.4mu=\mkern-2mu 
(\mu_1, \ldots, \mu_n) \mkern-2mu\in\mkern-2mu \prod_{i=1}^n {\mathbb{P}}(H_1(\partial_i Y))$.
Moreover, suppose either that $b_1(Y(\mu_*))>0$ 
for the Dehn filling $Y(\mu_*)$,
or that $Y\!$ is a graph manifold,
and there is some (possibly empty)
$\bar{J} \subset \{1, \ldots,n\}$,
and $\bar{N}$-filling
$Y^{\bar{N}}\mkern-2mu(\bar{J}; \mu_*)$ of $Y\!$ along
$\mu_{\bar{\jmath}}$ in $\partial_{\bar{\jmath}} Y\mkern-2.5mu$ 
for each $\bar{\jmath} \in \bar{J}$,
such that $Y^{\bar{N}}\mkern-2mu(\bar{J}; \mu_*)$ admits
a co-oriented taut foliation $F$ transverse to the boundary,
with $\alpha(F \cap \partial_j Y) \mkern-1.5mu=\mkern-1.5mu \mu_j$
for each $j \mkern-2mu\in\mkern-2mu 
J\mkern-2mu:\mkern-.3mu=\mkern-2mu \{1, \ldots, n\} \mkern-2mu\setminus\mkern-2mu \bar{J}$.

Then, for {\em{every}} 
$\bar{J} \mkern-2mu\subset\mkern-2mu \{1, \ldots,n\}$,
{\em{every}} $\bar{N}$-filling
$Y^{\bar{N}}\mkern-2mu(\bar{J}; \mu_*)$
(including $Y\mkern-2mu:\mkern-.3mu=
\mkern-2mu Y^{\bar{N}}\mkern-2mu(\emptyset; \mu_*)$)
admits a co-oriented taut foliation $F$ transverse to the boundary,
with $\alpha(F \cap \partial_j Y) \mkern-1.5mu=\mkern-1.5mu \mu_j$
for each $j \mkern-2mu\in\mkern-2mu 
J\mkern-2mu:\mkern-.3mu=\mkern-2mu \{1, \ldots, n\} \mkern-2mu\setminus\mkern-2mu \bar{J}$.
\end{prop}

\begin{proof}
Part (2) of Gabai's main theorem in \cite{gabai} tells us that
any prime oriented three-manifold with $b_1 \mkern-2mu>\mkern-2mu 0$
and boundary a (possibly empty) union of tori
admits a co-oriented taut foliation transverse to the boundary.
Thus, if $b_1(Y(\mu_*)) \mkern-2mu>\mkern-2mu 0$,
then any $\bar{N}$-filling 
$Y^{\bar{N}}\mkern-2mu(\mu_*)
\mkern-1.5mu:=\mkern-1.5mu
Y^{\bar{N}}\mkern-2mu(\{1, \ldots, n\};\mu_*)$ has $b_1>0$,
hence admits a co-oriented taut foliation $F$.
Since each $\partial_i  Y$ is an incompressible separating torus
in this $\bar{N}$-filling,
Proposition~\ref{prop: restrict to separated JSJ components}
allows us to isotop these separating tori so that they are everywhere
transverse to $F$.
Restricting $F$ to
any sub-$\bar{N}$-filling 
$Y^{\bar{N}}\mkern-2mu(\bar{J};\mu_*)
\subset Y^{\bar{N}}\mkern-2mu(\mu_*)$
then gives the desired taut foliation on 
$Y^{\bar{N}}\mkern-2mu(\bar{J};\mu_*)$.

If instead, $Y$ is a graph manifold with
$b_1(Y(\mu_*)) = 0$,
and we are given
$\bar{J} \mkern-2mu\subset\mkern-2mu \{1, \ldots, n\}$ and
a co-oriented taut foliation $F$ on some $\bar{N}$-filling
$Y^{\bar{N}}\mkern-2mu(\bar{J}; \mu_*)$,
with $F$ transverse to the boundary
and with $\alpha(F \cap \partial_j Y) \mkern-1.5mu=\mkern-1.5mu \mu_j$
for each $j \mkern-2mu\in\mkern-2mu 
J\mkern-2mu:\mkern-.3mu=\mkern-2mu \{1, \ldots, n\} \mkern-2mu\setminus\mkern-2mu \bar{J}$,
then Proposition~\ref{prop: restrict to separated JSJ components}
again allows us to isotop each 
separating torus $\partial_{\bar{\jmath}} Y$ so that $F$ restricts
to a co-oriented taut foliation on 
$Y$, transverse to $\partial Y$,
with $\alpha(F \cap \partial_i Y) \mkern-1.5mu=\mkern-1.5mu \mu_i$
for each $i \mkern-2mu\in\mkern-2mu \{1, \ldots, n\}$.

We then apply the 
the foliation gluing theorem of Boyer and Clay 
\cite[Theorem 9.5.2]{BoyerClay}.
That is, for each $i \in \{1, \ldots, n\}$,
Theorem~\ref{thm: J Seifert fibered foliation classification}
computes that $\mathcal{F}(\bar{N}_i) = \{l_i\}$, with $l_i$ (of slope $0$)
the rational longitude of 
of the $i^{\mathrm{th}}$ copy $\bar{N}_i$ of $\bar{N}$.
Thus,
for any gluing maps $\varphi_i : \bar{N}_i \to -\partial_i Y$
sending $l_i \mapsto \mu_i$ in homology,
Boyer's and Clay's gluing theorem tells us
that there exist co-oriented taut foliations
$F'$ on 
$Y$ and
$F_i$ on $\bar{N}_i$,
transverse to respective boundaries, with
$\alpha(F' \cap \partial_i Y) \mkern-1.5mu=\mkern-1.5mu \mu_i
\mkern-1.5mu=\mkern-1.5mu
\varphi_{i*}^{{\mathbb{P}}}(\alpha(F_j \cap \partial \bar{N}_j))$
for each $i \in \{1, \ldots, n\}$,
such that
the $F_i$ and $F'$ glue together to form a co-oriented taut foliation
on the $\bar{N}$-filling $Y^{\bar{N}}\mkern-2mu(\mu_*)$
specified by the $\varphi_i$.
After isotoping the $\partial_i Y$ to be transverse to this foliation,
we can restrict this foliation to any
any sub-$\bar{N}$-filling 
$Y^{\bar{N}}\mkern-2mu(\bar{J}';\mu_*)
\subset Y^{\bar{N}}\mkern-2mu(\mu_*)$.
\end{proof}

{\noindent{In particular, for a graph manifold $Y$ with
torus boundary, we have $\mu \in \mathcal{F}(Y)$ if and only if
an $\bar{N}$-filling $Y^{\mkern-1mu\bar{N}}\mkern-2mu(\mu)$
admits a co-oriented taut foliation.

\section{L-space intervals}
\label{s: L-space intervals}
An L-space is a
closed oriented three manifold whose Heegaard Floer homology is trivial,
in the sense that for each $\mathrm{Spin}^c$ structure, the hat Heegaard Floer homology
looks like the singular homology of a point.
The reader unfamiliar with L-spaces could consult
\cite{ozszintro, ozszlectures}
for an introduction to Heegaard Floer homology,
or \cite{lslope} for a treatment of L-space Dehn fillings.
For present purposes, we shall only need
the classification of Seifert fibered L-spaces,
  some formal properties of sets of L-space Dehn-filling slopes,
and some basic gluing results,
all of which we catalog below.

\subsection{L-space Dehn fillings and $\bar{N}$-fillings}
\begin{definition}
If $Y$ is a compact oriented three-manifold with torus boundary, then
we define the {\em{L-space interval}} of $Y$ to be
\begin{equation}
\mathcal{L}(Y) := \left\{ \mu \in {\mathbb{P}}(H_1(\partial Y))|\, 
Y(\mu)\text{ is an L-space}
\right\}.
\end{equation}
We shall write $\mathcal{L}^{\circ}(Y)$ for the interior
of $\mathcal{L}(Y)$ in ${\mathbb{P}}(H_1(\partial Y))$.
\end{definition}

Thus, $\mathcal{L}(Y)$ is analogous to,
and often complementary to,
$\mathcal{F}^D(Y)$, especially when
$Y$ has no reducible non-L-space Dehn-fillings.
Moreover, since the set of slopes of co-oriented taut 
foliations meeting a generalized rotation number condition is often the closure of 
the set of product foliation slopes
meeting that condition
\cite[Lemma 3.31]{ziggurat}, it is natural to ask if
$\mathcal{F}(Y)$ bears any relation to the complement of $\mathcal{L}^{\circ}(Y)$.
In fact, we have the following result.

\begin{prop}[\cite{lslope}]
\label{prop: L-space N-fillings}
If $Y\mkern-2.5mu$ is a compact oriented three manifold with torus boundary,
then an $\bar{N}$-filling $Y^{\mkern-1mu\bar{N}}\mkern-2mu(\mu)$ is an L-space if and only if
$\mu \in \mathcal{L}^{\circ}\mkern-2mu(Y)$.
\end{prop}
\begin{proof}
In \cite[Proposition 7.9]{lslope},
J. Rasmussen and the author prove the above result with $\bar{N}$
replaced by any member of a more general class of manifolds dubbed
{\em{generalized solid tori}}.  Since $\bar{N}$ is a generalized solid torus
as defined in \cite{lslope},
the result follows.
\end{proof}

Thus $\mathcal{L}^{\circ}\mkern-1mu(Y)$ can be regarded as the
L-space $\bar{N}$-filling interval of $Y$.

\subsection{L-space gluing}

Our primary tool for characterizing when a union of
three-manifolds along a torus boundary gives an L-space is the following joint result of 
J. Rasmussen and the author \cite{lslope}.
Hanselman and Watson have proven a similar result in \cite{HanWat}.
\begin{prop}[\cite{lslope}]
\label{prop: L-space gluing thm from lslope}
If each of $Y_1$ and $Y_2$ is a compact oriented three-manifold
with torus boundary,
then for any gluing map $\varphi: \partial Y_1 \to -\partial Y_2$
with $\varphi_*^{{\mathbb{P}}}(\mathcal{L}^{\circ}(Y_1)) \cap \mathcal{L}^{\circ}(Y_2) \neq \emptyset$,
the union $Y_1 \cup_{\varphi} \!Y_2$ is an L-space if and only if
${\mathbb{P}}(H_1(\partial Y_2))
=\varphi_*^{{\mathbb{P}}}(\mathcal{L}^{\circ}(Y_1)) \cup \mathcal{L}^{\circ}(Y_2)$
if both $Y_i$ are boundary incompressible, and if and only if
${\mathbb{P}}(H_1(\partial Y_2))
=\varphi_*^{{\mathbb{P}}}(\mathcal{L}(Y_1)) \cup \mathcal{L}(Y_2)$ otherwise.
\end{prop}

\begin{proof}
Recall that a compact three-manifold with torus boundary is boundary incompressible
if and only if it is not a connected sum of a solid torus with a
(possibly empty) closed three-manifold.
The above proposition replicates Theorem 6.2 from \cite{lslope},
except with the hypothesis of boundary incompressibility of each $Y_i$
replacing an {\em{a priori}} more technical condition that
certain subsets $\mathcal{D}^{\tau}(Y_i) \subset H_1(Y_i)$ be nonempty.
Thomas Gillespie has recently shown
\cite{Gillespietorus} these two conditions to be equivalent.
None of our gluing arguments involving graph manifolds 
actually make use of his result,
but our later cabling results,
which in principle require unions with non graph manifolds,
do require Gillespie's result.

\end{proof}

We shall later show
that in the case of
non-solid-torus graph manifolds $Y_i$ with torus boundary,
various foliation results allow us to
drop some of the above hypotheses, so that one obtains
an L-space if and only if
${\mathbb{P}}(H_1(\partial Y_2))
=\varphi_*^{{\mathbb{P}}}(\mathcal{L}^{\circ}(Y_1)) \cup \mathcal{L}^{\circ}(Y_2)$.

In rather the opposite direction,
if $Y$ is Seifert fibered over a punctured $S^2$ or ${\mathbb{R}}{\mathbb{P}}^2$,
then a Dehn filling $Y'$ of $Y$
fails to be a graph manifold if and only if $Y'$ fails to be prime,
if and only if $Y' \neq S^1 \times S^2$
and the Dehn filling, in some $\partial_i Y$,
was along the fiber lift $\tilde{f}_i \in H_1(\partial_i Y)$
of slope $\pi_i(\tilde{f}_i) = \infty$
(see the Remark in 
Section~\ref{ss: Conventions for Seifert fibered spaces}).
In this case, $Y'$ is neither a graph manifold
nor a habitat for taut foliations,
but since it has {\em{compressible boundary}},
its L-space gluing properties simplify, due to the following result.

\begin{prop}
\label{prop: l-space gluing for compressible boundary}
Suppose that each of $Y_1$ and $Y_2$ is a compact oriented three-manifold
with torus boundary, and that $Y_1$ has compressible boundary.
Then the union $Y_1 \cup_{\varphi} \!Y_2$ is an L-space if and only if
$X_1, \ldots, X_N$ are all L-spaces and 
$\varphi_*^{{\mathbb{P}}}(l_1) \in \mathcal{L}(Y_2)$,
where $l_1$ is the rational longitude of $Y_1$, and
where $Y_1$ decomposes as 
$Y_1 =  (S^1\mkern-1.5mu \times\mkern-1.5mu D^2) \# (X_1 \# \cdots \# X_N )$.
\end{prop}
\begin{proof}
A union along toroidal boundaries with a solid torus is just a Dehn filling, so we have
$Y_1 \cup_{\varphi} \!Y_2 = Y_2(\varphi^{{\mathbb{P}}}_*(l_1)) \#
X_1 \# \cdots \# X_N$, and a connected sum of closed manifolds is an L-space
if and only if each summand is an L-space.
\end{proof}

The above result explains why not every graph manifold $Y$ with torus boundary
satisfies
$\mathcal{F}^D(Y) \amalg \mathcal{L}(Y) = {\mathbb{P}}(H_1(\partial Y))$.
That is, no reducible manifold
(besides $S^1 \times D^2$) admits a co-oriented taut foliation, 
but there exist graph manifolds with reducible Dehn fillings which
are not $S^1 \times D^2$ or an L-space.

\subsection{Floer simple manifolds and L-space intervals}
It is not known, in general, 
what forms the sets
$\mathcal{F}(Y)$ or $\mathcal{F}^D(Y)$ can take for an arbitrary
compact oriented three-manifold $Y$ with torus boundary, but the situation
for L-spaces is better understood.
As shown in \cite{lslope} by J. Rasmussen and the author,
$\mathcal{L}(Y)$ can only be empty, the set of a single point,
a closed interval, or the complement of the 
rational longitude in ${\mathbb{P}}(H_1(\partial Y))$.
For historical reasons, we call $Y$ {\em{Floer simple}}
in the latter two cases.  Equivalently, we could define 
Floer simple manifolds as follows.
\begin{definition}
A compact oriented three-manifold $Y$ with torus boundary is
{\em{Floer simple}} if $\mathcal{L}^{\circ}(Y) \neq \emptyset$.
\end{definition}

In particular, if $Y$ is Floer simple, then its space
$\mathcal{L}(Y)$ of L-space Dehn filling slopes
can be specified entirely in terms of the left-hand
and right-hand endpoints of $\mathcal{L}(Y)$, in a sense we can make precise,
prefaced with the introduction of an
abbreviative notation
for the closed interval with infinite endpoint.

\begin{definition}
For $y \in {\mathbb{Q}}$,
we shall write $[-\infty, y]$, $[y, +\infty]$,
$\left[-\infty, y\right>$, and $\left< y, +\infty\right]$
for the following intervals in $({\mathbb{Q}} \cup \{\infty\}) \subset ({\mathbb{R}} \cup \{\infty\})$:
\begin{align*}
&& [-\infty, y] &:= \{\infty\} \cup \left<-\infty, y\right],
&&& \left[-\infty, y\right> &:= \{\infty\} \cup \left<-\infty, y\right>,&&
       \\
&&[y, +\infty] &:=  \left[y, +\infty\right> \cup \{\infty\},
&&&\left<y, +\infty\right] &:=  \left<y, +\infty\right> \cup \{\infty\}.&&
\end{align*}
\end{definition}

\begin{definition}
If $y_-, y_+ \!\in {\mathbb{Q}} \cup \{\infty\}$, then we define the
{\em{L-space interval from $y_-$ to $y_+$}}, denoted
$[[y_-, y_+]] \subset {\mathbb{Q}} \cup \{\infty\}$, as follows:
\begin{equation}
[[y_-, y_+]]
:=
\begin{cases}
\left< -\infty, +\infty \right>
&
\infty \mkern-1.5mu=\mkern-1.5mu y_- \mkern8.2mu,\mkern8.3mu 
y_+ \mkern-1.5mu=\mkern-1.5mu \infty
    \\
\left<y_-, +\infty\right] \cup \left[-\infty, y_+ \right>
& \,{\mathbb{Q}} \ni y_- = y_+ \in {\mathbb{Q}}
    \\
[y_-, +\infty] \cup [-\infty, y_+ ]
& \,{\mathbb{Q}} \ni y_- > y_+ \in {\mathbb{Q}}
    \\
[y_-, +\infty] \cap [-\infty, y_+ ]
& \,{\mathbb{Q}} \ni y_- < y_+ \in {\mathbb{Q}}
    \\
[-\infty, y_+ ]
&\infty \mkern-1.5mu=\mkern-1.5mu y_- \mkern8.2mu,\mkern8.3mu y_+ \in {\mathbb{Q}}
    \\
[y_-, +\infty]
& \,{\mathbb{Q}} \ni y_- \mkern8.2mu,\mkern8.3mu y_+ \mkern-1.5mu=\mkern-1.5mu \infty
\end{cases}.
\end{equation}
In other words, $[[y_-, y_+]]$ is the unique interval with left-hand endpoint
$y_-$ and right-hand endpoint $y_+$ which is closed if $y_- \neq y_+$ and open otherwise. 
\end{definition}
{\noindent{\bf{Remark:}}
In practice, we extend the above definition to allow $y_- = -\infty$
or $y_+ = +\infty$, which we treat as identical to the respective cases of
$y_- = \infty$ or $y_+ = \infty$.}

\begin{prop}
\label{prop: [[]] for any Floer simple}
Suppose, for some compact oriented three-manifold $Y$ with torus boundary,
that we are given an identification 
${\mathbb{P}}(H_1(\partial Y)) \cong {\mathbb{Q}} \cup \{\infty\}$.
If $Y$ is Floer simple, then there are unique 
$y_-, y_+ \!\in {\mathbb{Q}} \cup \{\infty\}$ such that 
$\mathcal{L}(Y) = [[y_-, y_+]] \subset {\mathbb{Q}} \cup \{\infty\}$.
Conversely, if there are $y_-, y_+ \!\in {\mathbb{Q}} \cup \{\infty\}$
for which $\mathcal{L}(Y) = [[y_-, y_+]]$, then
$Y$ is Floer simple.
\end{prop}
The above follows from the aforementioned result, 
proven in \cite{lslope}, 
that if $\mathcal{L}(Y)$ contains more than one point, then
$\mathcal{L}(Y)$ is either a closed interval or the complement of
a point in ${\mathbb{P}}(H_1(\partial Y))$.
The following computation of L-space intervals for Seifert fibered spaces
demonstrates one use of this
``$[[\cdot\mkern.7mu,\mkern-.8mu \cdot]]$'' notation.

\begin{prop}
\label{prop: l space interval for seifert fibered spaces}
If $Y$ is a regular fiber complement in a
Seifert fibered rational homology sphere,
then $\mathcal{L}^{\circ}(Y) \amalg \mathcal{F}(Y) = {\mathbb{P}}(H_1(\partial Y))$.
If $Y$ has non-orientable base,
we have $\mathcal{L}(Y) = \left<-\infty, +\infty\right>$
and $\mathcal{F}^D\mkern-2mu(Y) = \emptyset$ (unless $Y$ is the twisted
$S^2$-bundle over the M{\"o}bius strip,
in which case $\mathcal{F}^D\mkern-2mu(Y) \mkern-2mu=\mkern-2mu \{\infty\}$)
If $Y$ is a regular fiber complement in
$M_{S^2}(y_*)$, for some $y_* = (y_0, \ldots, y_n) \in {\mathbb{Q}}^{n+1}$,
then 
${\mathbb{P}}(H_1(\partial Y)) \setminus \mathcal{F}^D(Y) 
= \mathcal{L}(Y) = [[y_-, y_+]]$, where
\begin{equation}
\label{eq: SFS L-space endpoints}
y_-
:=
\max_{k > 0}
-\mfrac{1}{k}\mkern-1.8mu\left(\mkern2mu
1
+ {{\textstyle{\sum\limits_{i=0}^n  \left\lfloor y_i k \right\rfloor}}} \right),
\;\;\;\;\;
y_+ 
:=
\min_{k > 0} 
-\mfrac{1}{k}\mkern-1.8mu\left(\mkern-3mu
-1
+ {{\textstyle{\sum\limits_{i=0}^n   \left\lceil y_i k \right\rceil}}} \right),
\end{equation}
unless $Y\mkern-4.5mu$ is a solid torus, in which case
$y_- \mkern-3.2mu:\mkern-.3mu=\mkern-1mu y_+ \mkern-3.2mu:\mkern-.3mu=\mkern-1mu
-\mkern-3.5mu\sum_{i=0}^n y_i$ is the rational longitude of $Y\mkern-5mu$.
\end{prop}
\begin{proof}
For the case of $Y$ with non-orientable base,
see the work of Boyer, Gordon, and Watson \cite{BGW} and 
Boyer and Clay \cite{BoyerClay}.  For $Y$ with orientable base,
the foliations result is due to Jankins, Neumann, and Naimi
\cite{JankinsNeumann, Naimi},
and the L-space result is originally due to the combined work of
Jankins, Neumann, and Naimi \cite{JankinsNeumann, Naimi},
Eliashberg and Thurston \cite{EliashbergThurston},
Ozsv{\'a}th and Szab{\'o} \cite{OSGen},
Lisca and Mati{\'c} \cite{LiscaMatic}, and
Lisca and Stipsicz \cite{LSIII}.
Alternatively,
J. Rasmussen and the author
offer a recent stand-alone proof of the L-space result \cite{lslope}.
\end{proof}

\section{L-space Intervals and Foliation Slopes for Graph Manifolds}
\label{s: main results section}

This is the section in which we prove most of our main results.
We begin, however, by 
introducing the notion of {\em{L/NTF-equivalence}},
the presence of which makes gluing easier.
We further pause in
Section \ref{subsection: conventions for graph manifolds},
to establish some conventions
for graph manifolds with torus boundary and $b_1 = 1$.

\subsection{L/NTF-equivalence and Gluing}

For a pair of manifolds spliced together along torus boundaries,
we can often prove stronger gluing results about the existence of co-oriented
taut foliations or non-trivial Heegaard Floer homology 
if we are able to use gluing theorems from both areas of mathematics.
In general, however, this strategy only works if we know that
each manifold behaves in a suitably complementary manner with
respect to co-oriented taut foliations and L-space Dehn fillings,
a notion which we now make precise.

\begin{definition}
If $Y\!$ is a prime compact oriented three-manifold with torus boundary,
then we call $Y$ {\em{L/NTF-equivalent}} if 
$\mathcal{F}(Y) \amalg \mathcal{L}^{\circ}(Y) = {\mathbb{P}}(H_1(\partial Y))$.
\end{definition}
In certain circumstances, one can characterize L/NTF-equivalence
in terms of $\bar{N}$-fillings.

\begin{prop}
\label{prop: lntf in terms of Nfilling}
Suppose $Y\!$ is a prime compact oriented three-manifold with torus boundary.
If $b_1(Y) >1$ or $Y$ is a graph manifold,
then $Y\!$ is L/NTF-equivalent if and only if
each $\bar{N}$-filling of $Y\!$ is an L-space
precisely when it fails to admit a 
co-oriented taut foliation.
\end{prop}
\begin{proof}
This follows immediately from
Propositions~\ref{prop: boyer and clay's taut foliation slope vs N-filling}
and
\ref{prop: L-space N-fillings}.
\end{proof}

There are some classes of manifold which we already know to be
L/NTF-equivalent.

\begin{prop}
\label{prop: proof of lntf equivalence for b1>1 and seifert + N}
Suppose $Y$ is a prime compact oriented three-manifold with torus boundary.
If $b_1(Y) > 1$, or if $Y$ is the union of a Seifert fibered space with
zero or more copies of $\bar{N}$,
then $Y$ is L/NTF-equivalent.
\end{prop}

\begin{proof}
If $b_1(Y) \mkern-2mu>\mkern-2mu 1$, 
then since no Dehn filling of $Y$ is a rational homology sphere,
we have $\mathcal{L}(Y) \mkern-2mu=\mkern-2mu \emptyset$,
implying $\mathcal{L}^{\circ}\mkern-2mu(Y) \mkern-2mu=\mkern-2mu \emptyset$.
Correspondingly,
Proposition~\ref{prop: boyer and clay's taut foliation slope vs N-filling}
implies $\mathcal{F}(Y) \mkern-2mu=\mkern-2mu {\mathbb{P}}(H_1(\partial Y))$.

Suppose $Y$ has $b_1(Y) = 1$ and 
is the union of a Seifert fibered space with zero or more copies
of $\bar{N}$.
Then, using \cite[Proposition 7.9]{lslope}
of J. Rasmussen and the author to replace Boyer's and Clay's
``$N_t$'' manifolds with $\bar{N}$, we invoke
the ``slope detection'' theorem
of Boyer and Clay
\cite[Theorem~8.1]{BoyerClay} 
to deduce L/NTF-equivalence for $Y$.

Alternatively,
one could prove the same result
by inductively performing $\bar{N}$-fillings in
regular fiber complements,
starting with
Proposition~\ref{prop: l space interval for seifert fibered spaces}
for the Seifert fibered base case,
and using the gluing result in
Proposition 
\ref{prop: gluing prop for LNTF equivalent manifolds}
below, together with
(\ref{eq: closed interval to ceiling floor}) and
(\ref{eq: open interval to ceiling floor}),
to evolve
(\ref{eq: SFS L-space endpoints})
to match
(\ref{eq: defs of y- and y+}).
Similar inductive arguments appear in the proof of
Theorem~\ref{thm: l space interval for graph manifolds}.
\end{proof}
{\noindent{\bf{Remark.}} We later prove
L/NTF-equivalence for
all graph manifolds with torus boundary.}
\smallskip

We are now ready to state our main gluing result.
\begin{prop}
\label{prop: gluing prop for LNTF equivalent manifolds}
Suppose $Y_1$ and $Y_2$ are non-solid-torus
L/NTF-equivalent graph manifolds with torus boundary.
Then for any union $Y_1 \cup_{\varphi}\! Y_2$,
$\varphi : \partial Y_1 \to -\partial Y_2$,
the following are equivalent:
\begin{enumerate}
\item[$(i)\phantom{ii}$]
$Y_1 \cup_{\varphi} \!Y_2$ is an L-space.

\item[$(ii)\phantom{i}$]
$Y_1 \cup_{\varphi} \!Y_2$ does not admit a co-oriented taut foliation.

\item[$(iii)$]
$\varphi_*^{{\mathbb{P}}}(\mathcal{L}^{\circ}(Y_1)) \cup \mathcal{L}^{\circ}(Y_2)
={\mathbb{P}}(H_1(\partial Y_2))$.
\end{enumerate}
\end{prop}
\begin{proof}
Suppose $(iii)$ holds, so that
Proposition \ref{prop: L-space gluing thm from lslope}
implies 
$Y_1 \cup_{\varphi}\! Y_2$ is an L-space.
Ozsv{\'a}th and Szab{\'o} have shown \cite{OSGen}
that an L-space does not admit $C^2$ co-oriented taut foliations,
and this result has been improved to $C^0$ co-oriented taut foliations
by Bowden \cite{Bowden} 
and independently by Kazez and Roberts \cite{KazezRobertsCzero}.

Suppose $(iii)$ fails to hold.
If $b_1(Y_1 \mkern1.5mu\cup_{\varphi}\mkern-1.5mu Y_2) \mkern-1mu>\mkern-1mu 0$, then
Gabai \cite{gabai} tells us there is a co-oriented taut foliation on
$Y_1 \cup_{\varphi}\mkern-1.5mu Y_2$, 
and we also know that $Y_1 \cup_{\varphi}\mkern-2mu Y_2$ is not an L-space.
Suppose instead that $b_1(Y_1 \mkern1.5mu\cup_{\varphi}\mkern-1.5mu Y_2) = 0$.
Then L/NTF-equivalence implies
$\varphi_*^{{\mathbb{P}}}(\mathcal{F}(Y_1)) \cap \mathcal{F}(Y_2) \neq \emptyset$,
and so there are
co-oriented taut foliations $F_i$ on $Y_i$ transverse to
$\partial Y_i$
such that $\varphi^{{\mathbb{P}}}_*(\alpha(F_1\cap \partial Y_1)) = \alpha(F_2 \cap \partial Y_2)$.
By Corollary 
\ref{cor: taut tree manifold foliations restrict},
we can isotop the incompressible tori in $Y_1$ and $Y_2$ so that
each $F_i$ restricts to boundary-transverse 
co-oriented taut foliations on each of the JSJ components
of each $Y_i$.
This means that the JSJ components of 
$Y_1 \cup_{\varphi}\mkern-2mu Y_2$
each admit boundary-transverse co-oriented taut foliations
which restrict to boundary foliations of matching slopes with
respect to boundary-gluing maps.
We can therefore invoke
the foliation gluing theorem of Boyer and Clay
\cite[Theorem 9.5.2]{BoyerClay}
to assert the existence of a co-oriented taut foliation
on all of $Y_1 \cup_{\varphi}\mkern-2mu Y_2$.
Again, this co-oriented taut foliation implies that 
$Y_1 \cup_{\varphi}\! Y_2$ is not an L-space.
\end{proof}

\subsection{Graph manifold conventions}
\label{subsection: conventions for graph manifolds}

Any closed graph manifold can be regarded as a Dehn filling of a 
graph manifold $Y$ with torus boundary.
If $b_1(Y) > 1$, then $\mathcal{L}(Y) = \mathcal{L}^{\circ}(Y) = \emptyset$
and $\mathcal{F}(Y) = {\mathbb{P}}(H_1(\partial Y))$
(see 
Proposition~\ref{prop: proof of lntf equivalence for b1>1 and seifert + N}),
which is not very interesting.

If $b_1(Y) = 1$, then
we call $Y$ a {\em{tree manifold}}, since 
$Y$ admits a rational homology sphere Dehn filling,
corresponding to a tree graph.
Rooting the tree graph for $Y$ at the Seifert fibered
piece containing $\partial Y$ provides a recursive construction for $Y$,
\begin{equation}
\label{eq: recursive tree construction}
Y = \hat{M} \cup 
\left(
{{\textstyle{\coprod\limits_{i=1}^{n_{{\textsc{d}}}}(S^1 \times D^2_i)}}} 
\;\amalg\; 
{{\textstyle{\coprod\limits_{i=1}^{n_{{\textsc{g}}}} Y_i}}}\right),
\end{equation}
where each of $Y_1, \ldots, Y_{n_{{\textsc{g}}}}$ is
a non-solid-torus tree manifold with torus boundary.
Since $b_1(Y) = 1$,
$\hat{M}$ is the trivial circle fibration over
an $n \mkern-1.4mu+\mkern-2mu 1$-punctured $S^2\mkern-2mu$ or the twisted circle fibration
over an $n \mkern-2mu+\mkern-2mu 1$-punctured ${\mathbb{R}}{\mathbb{P}}^2\!$,
with boundary components $\partial^{{\textsc{d}}}_i\mkern-2mu \hat{M} \mkern-1.5mu:= \partial_i \hat{M}$
for $i \mkern-1.8mu\in\mkern-1.8mu \{1, \ldots, n_{{\textsc{d}}}\}$, 
$\partial^{{\textsc{g}}}_i \mkern-2mu\hat{M} 
\mkern-1.5mu:= \mkern-.8mu\partial_{n_{{\textsc{d}}} + i} \hat{M}$
for $i \mkern-1.8mu\in\mkern-1.8mu \{1, \ldots, n_{{\textsc{g}}}\}$, and 
$\partial Y \mkern-2.2mu:=\mkern-1mu 
\partial_{n+1} \hat{M} \mkern-2mu=:\mkern-2mu \partial^{{\textsc{g}}}_{n_{{\textsc{g}}}+1}\hat{M} 
\mkern-2mu=:\mkern-2mu
\partial^{{\textsc{d}}}_{n_{{\textsc{d}}}+1}\hat{M}$,
with $n \mkern-1mu:\mkern-1mu= n_{{\textsc{d}}} \mkern-.5mu+\mkern1mu n_{{\textsc{g}}}$.
We shall sometimes call $\hat{M}$ the ``foundation'' for $Y$.

Since edges in the graph for $Y$ correspond to gluings along incompresible tori,
each gluing map
$\varphi_i: \partial Y_i \to -\partial^{{\textsc{g}}}_i\mkern-2mu \hat{M}$, 
for $i \in \{1, \ldots, n_{{\textsc{g}}}\}$,
labels one of the $n_{{\textsc{g}}}$ edges descending from the root of the graph for $Y$.
We call $Y_1, \ldots, Y_{n_{{\textsc{g}}}}$ the {\it{daughter subtrees}} of $Y$.
Each $Y_i$ is a tree manifold with torus boundary and $b_1(Y_i) = 1$,
with tree rooted at the Seifert
fibered piece containing $\partial Y_i$, giving rise to a
recursive description for $Y_i$ analogous to that for $Y$ in
(\ref{eq: recursive tree construction}).
For any $i \in \{1, \ldots, n_{{\textsc{g}}}\}$
for which $Y_i$ is Floer simple,
{\em{i.e.}},
for which $\mathcal{L}^{\circ}(Y_i) \neq \emptyset$,
then we invoke
Proposition~\ref{prop: [[]] for any Floer simple}
to write
\begin{equation}
[[y_{i-}^{\textsc{g}}, y_{i+}^{\textsc{g}}]]
:= \varphi_{i*}^{{\mathbb{P}}}(\mathcal{L}(Y_i)).
\end{equation}
If instead,
$\mathcal{L}(Y_i) \neq \emptyset$ for some non-Floer-simple $Y_i$, then we write
\begin{equation}
\{y_{i-}^{\textsc{g}}\} := \{y_{i+}^{\textsc{g}}\}
:= \varphi_{i*}^{{\mathbb{P}}}(\mathcal{L}(Y_i)).
\end{equation}

Since the $n_{\textsc{d}}$ solid tori glued to $\hat{M}$
create the exceptional fibers of the Seifert fibered ``root''
$\hat{M} \cup \coprod_{i=1}^{n_{\textsc{d}}} (S^1 \times D_i^2)$ of our tree,
we record the Seifert data of this Seifert fibered space
by labeling the root vertex with the Dehn filling slopes
$y_1^{\textsc{d}}, \ldots, y_{n_{{\textsc{d}}}}^{\textsc{d}} \in {\mathbb{Q}}$.
More explicitly, to each solid torus $S^1 \times D^2_i$ in
(\ref{eq: recursive tree construction}),
we associate the gluing map
$\varphi^{\textsc{d}}_i \mkern-2.5mu:\mkern-1.5mu 
\partial (S^1 \mkern-1mu\times\mkern-1mu D^2_i) \mkern-1.5mu\to\mkern-1.5mu
-\partial^{{\textsc{d}}}_i\mkern-2mu \hat{M}$,
and set $y_i^{\textsc{d}} \mkern-2mu:=\mkern-2mu \varphi_{i*}^{\textsc{d} {\mathbb{P}}}(l_i)$,
for $l_i$ the rational longitude of $S^1 \mkern-1mu\times\mkern-1mu D^2_i$.
As usual, we demand that each $y_i^{\textsc{d}} \neq \infty$.
We stray slightly from our earlier convention by
allowing $y_i^{\textsc{d}} \in {\mathbb{Z}}$,
but this allows us to fix
$e_0 := y_0^{\textsc{d}} := 0$ and then forget the $0^{\mathrm{th}}$
fiber complement altogether, without loss of generality.


\subsection{Statement of Main Results}
\label{ss: Statement of Main Results}
We first show that all graph manifolds with torus boundary
are L/NTF-equivalent,
making our main gluing tool,
Proposition
\ref{prop: gluing prop for LNTF equivalent manifolds},
applicable for all such non-solid-torus graph manifolds.
We then can make the inductive gluing arguments
necessary to calculate
L-space intervals for graph manifolds
with torus boundary.

\begin{theorem}
\label{thm: lntf equivalence}
Every graph manifold $Y$ with torus boundary is L/NTF equivalent,
{\em{i.e.}}, satisfies
$\mathcal{F}(Y) \amalg \mathcal{L}^{\circ}(Y) = {\mathbb{P}}(H_1(\partial Y))$.
Moreover, if we let $\mathcal{R}(Y)\!$ denote the set of slopes of 
reducible (and not $S^1 \times D^2$)
Dehn fillings of $Y\!$, then
$\mathcal{F}^D(Y) \amalg (\mathcal{L}(Y) \cup \mathcal{R}(Y)) = {\mathbb{P}}(H_1(\partial Y))$.
\end{theorem}
The above also implies that the following calculation of $\mathcal{L}(Y)$
for graph manifolds $Y\!$ with torus boundary completely determines
both $\mathcal{F}(Y)$ and $\mathcal{F}^D(Y)$.

\begin{theorem}
\label{thm: l space interval for graph manifolds}
Suppose $Y$ is a graph manifold with torus boundary and nonempty $\mathcal{L}(Y)$.
If the Seifert fibered component of $Y$ containing $\partial Y$
has non-orientable base, then 
$\mathcal{L}(Y) = \left<-\infty, +\infty\right>$.
Otherwise, we have 
\begin{equation*}
\mathcal{L}(Y) =
\begin{cases}
[[y_-, y_+]]
&\;\;\;
Y\,\text{Floer simple},
 \\
\{y_-\} = \{ y_+\}
&\;\;\;
Y\,\text{not Floer simple},
\end{cases}
\end{equation*}
where,
for $n_{{\textsc{d}}}, n_{{\textsc{g}}}, y_i^{\textsc{d}}, y_{i-}^{\textsc{g}}\mkern-.2mu,\mkern-1mu$
and $y_{i+}^{\textsc{g}}\mkern-3mu$
as defined in
Section~\ref{subsection: conventions for graph manifolds},
we define $y_-, y_+ \in {\mathbb{Q}} \cup \{\infty\}$ as
\begin{align*}
y_-
&:=\,
\max_{k > 0}
-\mfrac{1}{k}\mkern-2.5mu\left( 1
+ {{\textstyle{\sum\limits_{i=1}^{n_{{\textsc{d}}}}  \left\lfloor y^{\textsc{d}}_i k \right\rfloor}}}
+ {{\textstyle{\sum\limits_{i=1}^{n_{{\textsc{g}}}} \left(  \left\lceil y^{\textsc{g}}_{i+}k 
\right\rceil - 1 \right)}}}
\right),
\\ \nonumber
y_+ 
&:=\,
\min_{k > 0}
-\mfrac{1}{k}\mkern-2.5mu\left(\mkern-2mu -1
+ {{\textstyle{\sum\limits_{i=1}^{n_{{\textsc{d}}}} \left\lceil y^{\textsc{d}}_i k \right\rceil}}}
+ {{\textstyle{\sum\limits_{i=1}^{n_{{\textsc{g}}}} \left(  \left\lfloor y^{\textsc{g}}_{i-} k 
\right\rfloor + 1 \right)}}}
\right),
\end{align*}
unless $Y\!$ is a solid torus, in which case
$y_- \!:=\mkern-2mu y_+ \!:=\mkern-2mu -\sum_{i=1}^{n_{{\textsc{d}}}} y_i^{\textsc{d}}$.
\end{theorem}
{\noindent{\bf{Remark.}} The above formulae for $y_{\mp}$ are
finitely computable.
In particular,
the maximum (respectively, minimum) is realized for 
$k \le s_{\pm}$, where $s_{\pm}$
is the least common positive multiple of the
denominators of $y_1^{\textsc{d}}, \ldots, y_{n_{{\textsc{d}}}}^{\textsc{d}}$ and
$y_{1\pm}^{\textsc{g}}, \ldots, y_{n_{{\textsc{g}}}\pm}^{\textsc{g}}$, with $\mp$ and $\pm$ cases
taken respectively from top to bottom.
If computation is not the goal, then one can avoid
treating the solid torus case separately by
replacing ``$\max$'' with ``$\sup$'' and ``$\min$'' with ``$\inf$.''}
\smallskip

Unlike the case of oriented Seifert fibered spaces over
the M{\"o}bius strip or disk,
not all graph manifolds with torus boundary are Floer simple.
We therefore need a companion result to characterize precisely 
when $\mathcal{L}^{\circ}$ or $\mathcal{L}$ is nonempty.
\begin{prop}
\label{prop: characterization of Floer simple graph manifolds}
Suppose $Y\!$ is a graph manfold with torus boundary.  If 
$b_1(Y) \mkern-1.5mu>\mkern-1.5mu 1$,
then $\mathcal{L}(Y) \mkern-1.5mu=\mkern-.5mu \emptyset$.
Suppose $b_1(Y)\mkern-1.5mu=\mkern-1.5mu1$, so that 
$Y\!$ admits the recursive description in
Section~\ref{subsection: conventions for graph manifolds}.

If the JSJ component containing $\mkern.2mu\partial Y\mkern-3.5mu$ has
non-orientable base, 
then $\mathcal{L}(Y) \mkern-1.5mu\neq\mkern-.5mu \emptyset$ if and only if
$\mkern1.5muY\!$ is Floer simple, if and only if the following holds:
\begin{enumerate}
\item[$(\textsc{fs0}\mkern-.5mu)$]
All daughter subtrees
$Y_1, \ldots, Y_{n_{{\textsc{g}}}}\mkern-3mu$ are Floer simple;
\\
$\infty \in ([[y_{i-}^{\textsc{g}}, y_{i+}^{\textsc{g}}]])^{\circ}$
for all $i \in \{1, \ldots, n_{{\textsc{g}}}\}$.
\end{enumerate}

If the JSJ component containing $\mkern.8mu\partial Y\mkern-2mu$
has orientable base, then $Y\mkern-4mu$ is Floer simple if and only if
the daughter subtrees
$Y_1, \ldots, Y_{\mkern-1.5mu n_{{\textsc{g}}}}\mkern-3mu$ are each Floer simple
and one of the following holds:
\begin{enumerate}
\item[$(\textsc{fs1}\mkern-.5mu)$]
$[[y_{j-}^{\textsc{g}}, y_{j+}^{\textsc{g}}]]
\mkern-2mu=\mkern-2mu
\left<-\infty, +\infty\right>$
for some $j \mkern-2mu\in\mkern-2mu \{1, \ldots, n_{{\textsc{g}}}\}$;
   \\
$\infty \mkern-2mu\in\mkern-2mu ([[y_{i-}^{\textsc{g}}, y_{i+}^{\textsc{g}}]])^{\circ}$
for all $\mkern1mu i \mkern-2mu\in\mkern-2mu \{1, \ldots, n_{{\textsc{g}}}\} \setminus \{j\}$.

\item[$(\textsc{fs2}\mkern-.5mu)$]
$[[y_{j-}^{\textsc{g}}, y_{j+}^{\textsc{g}}]]
\mkern-2mu=\mkern-2mu
[y_{j-}^{\textsc{g}}, y_{j+}^{\textsc{g}}]$ for some 
$j \mkern-2mu\in\mkern-2mu \{1, \ldots, n_{{\textsc{g}}}\}$;
  \\
$\infty \mkern-2mu\in\mkern-2mu ([[y_{i-}^{\textsc{g}}, y_{i+}^{\textsc{g}}]])^{\circ}$
for all $\mkern1mu i \mkern-2mu\in\mkern-2mu \{1, \ldots, n_{{\textsc{g}}}\} \setminus \{j\}$; 
  \\
$y_- \mkern-2mu<\mkern-1mu y_+$.

\item[$(\textsc{fs3}\mkern-.5mu)$]
At least one of 
$\{i :
[\mkern-.4mu[y_{i-}^{\textsc{g}}, y_{i+}^{\textsc{g}}]\mkern-.4mu]
\mkern-2.5mu=\mkern-2.5mu [-\infty, y_{i+}^{\textsc{g}}]\}$ and
$\{i : 
[\mkern-.4mu[y_{i-}^{\textsc{g}}, y_{i+}^{\textsc{g}}]\mkern-.4mu]
\mkern-2.5mu=\mkern-2.5mu [y_{i-}^{\textsc{g}} ,+\infty]\}$
is the empty set;
\\
$\infty \in 
[[y_{i-}^{\textsc{g}}, y_{i+}^{\textsc{g}}]]$
for all $\mkern1mu i \mkern-2mu\in\mkern-2mu 
\{1, \ldots, n_{{\textsc{g}}}\}$.

\end{enumerate}

If the JSJ component containing $\mkern.2mu\partial Y\mkern-3.5mu$
has orientable base, 
then $\mathcal{L}(Y) \neq \emptyset$
with $Y\mkern-2mu$ {\em{not}} Floer simple
if and only if
one of the following holds:
\begin{enumerate}
\item[$(\textsc{nfs1}\mkern-.5mu)$]
$n_{{\textsc{g}}} \mkern-3mu=\mkern-1mu 1$,
$\left|\{i: y_i^{\textsc{d}} \in {\mathbb{Q}} \setminus {\mathbb{Z}}\}\right| \le 1$;
\\
$\varphi_{1*}^{{\mathbb{P}}}(\mathcal{L}(Y_1)) = \{y^{\textsc{g}}_1 \}$
for some $y^{\textsc{g}}_1 \in {\mathbb{Q}}$;
  \\
$\sum_{i=1}^{n_{{\textsc{d}}}} \mkern-2mu y^{\textsc{d}}_i \mkern1mu+\mkern1mu y^{\textsc{g}}_1 
\in {\mathbb{Z}}$ or all $\mkern1.5mu y_1^{\textsc{d}},\ldots, y_{n_{{\textsc{d}}}}^{\textsc{d}} 
\mkern-2.5mu\in\mkern-.5mu{\mathbb{Z}}$.

\item[$(\textsc{nfs2}\mkern-.5mu)$]
All daughter subtrees
$Y_1, \ldots, Y_{n_{{\textsc{g}}}}$ are Floer simple;
\\
$[[y_{j-}^{\textsc{g}}, y_{j+}^{\textsc{g}}]]
=[y_{j-}^{\textsc{g}}, y_{j+}^{\textsc{g}}]$ for some $j \in \{1, \ldots, n_{{\textsc{g}}}\}$;
\\
$\infty \in ([[y_{i-}^{\textsc{g}}, y_{i+}^{\textsc{g}}]])^{\circ}$
for all $i \in \{1, \ldots, n_{{\textsc{g}}}\} \setminus \{j\}$;
\\
$y_- = y_+$.

\item[$(\textsc{nfs3}\mkern-.5mu)$]
All daughter subtrees
$Y_1, \ldots, Y_{n_{{\textsc{g}}}}$ are Floer simple;
  \\
$\{i :
[\mkern-.4mu[y_{i-}^{\textsc{g}}, y_{i+}^{\textsc{g}}]\mkern-.4mu]
\mkern-2.5mu=\mkern-2.5mu [-\infty, y_{i+}^{\textsc{g}}]\} \neq \emptyset$
and
$\{i : 
[\mkern-.4mu[y_{i-}^{\textsc{g}}, y_{i+}^{\textsc{g}}]\mkern-.4mu]
\mkern-2.5mu=\mkern-2.5mu [y_{i-}^{\textsc{g}} ,+\infty]\} \neq \emptyset$;
  \\
$\infty \in 
[[y_{i-}^{\textsc{g}}, y_{i+}^{\textsc{g}}]]$
for all $i \mkern-2mu\in\mkern-2mu 
\{1, \ldots, n_{{\textsc{g}}}\}$.

\item[$(\textsc{nfs4}\mkern-.5mu)$]
$\varphi^{{\mathbb{P}}}_{j*}(\mathcal{L}(Y_j)) 
\mkern-2.5mu=\mkern-2.5mu \{\infty\}$
for some 
$j \mkern-2.5mu\in\mkern-2.5mu \{1, \ldots, n_{{\textsc{g}}}\}$;
  \\
$\infty \mkern-2.5mu\in\mkern-2.5mu \varphi^{{\mathbb{P}}}_{i*}(\mathcal{L}(Y_i))$
for all 
$i \mkern-2.5mu\in\mkern-2.5mu \{1, \ldots, n_{{\textsc{g}}}\}$.

\end{enumerate}
\end{prop}
{\noindent{Note that all eight $(\textsc{fs})$ and $(\textsc{nfs})$
conditions are mutually exclusive.
Note also that the isolated L-space fillings described in
$(\textsc{nfs3}\mkern-.5mu)$ and $(\textsc{nfs4}\mkern-.5mu)$
are not graph manifolds.}}
\smallskip

We now proceed to prove our main results,
starting with that of
L/NTF-equivalence.

\subsection{Proof of Theorem~\ref{thm: lntf equivalence}}
\label{ss: Proof of Lntf equivalence}
Proposition~\ref{prop: proof of lntf equivalence for b1>1 and seifert + N}
gives the desired result for $b_1(Y) >1$.

We therefore restrict attention to
graph manifolds $Y$ with $b_1(Y) = 1$ and torus boundary $\partial Y$,
so that $Y$ admits the recursive description in
(\ref{eq: recursive tree construction}),
with tree graph rooted at the Seifert fibered piece containing $\partial Y$.
Inductively assume that any such $Y\!$ with tree height $\le k\mkern-1.5mu-\mkern-1.5mu1$
is L/NTF-equivalent and satisfies
$\mathcal{F}(Y) \amalg (\mathcal{L}(Y) \cup \mathcal{R}(Y)) = {\mathbb{P}}(H_1(\partial Y))$,
noting that
Proposition~\ref{prop: l space interval for seifert fibered spaces}
covers the case of trees of height zero.
Fix an arbitrary tree manifold $Y$ with $b_1(Y)=1$,
torus boundary, and tree height $k>0$,
and parameterize its data as in
Section~\ref{subsection: conventions for graph manifolds}.

To each $j \mkern-2mu\in\mkern-2mu \{1, \ldots, n_{{\textsc{g}}}\}$,
$y^j_* \mkern-2mu:=\mkern-2mu 
(y_{j+1}^j, \ldots, y_{n_{{\textsc{g}}}+1}^j) 
\mkern-2mu\in\mkern-2mu ({\mathbb{Q}} \cup \{\infty\})^{n_{{\textsc{g}}}+1-j}$,
and $\theta \mkern-2mu\in\mkern-2mu \{0,1\}$,
we associate a manifold $Y^j_{\theta}[y^j_*]$, constructed as follows.
Starting with $\hat{M}$, first perform 
Dehn fillings of slopes $y_1^{\textsc{d}}, \ldots, y_{n_{{\textsc{d}}}}^{\textsc{d}}$
along the respective boundary components 
$\partial^{{\textsc{d}}}_1 \mkern-1.8mu\hat{M}, \ldots, \partial^{{\textsc{d}}}_{n_{{\textsc{d}}}}\mkern-1.8mu \hat{M}$
in $\hat{M}$.
Next, attach the graph manifolds
$Y_1, \ldots, Y_{j-1}$, via the 
respective gluing maps $\varphi_1, \ldots, \varphi_{j-1}$,
to the resulting manifold.
Leaving the boundary component 
$\partial^{{\textsc{g}}}_{\mkern-1mu j}\mkern-2mu \hat{M} \mkern-1.8mu=:\mkern-1.8mu 
\partial Y^{\mkern-.5mu j}[y_*^j]$ unfilled,
lastly perform 
$\bar{N}$-fillings of slopes
$y_{j+1}^j, \ldots, y_{n_{{\textsc{g}}}+1}^j$ along the respective boundary components
$\partial^{{\textsc{g}}}_{j+1} \mkern-1.4mu\hat{M}, \ldots, 
\partial^{{\textsc{g}}}_{n_{{\textsc{g}}}+1}\mkern-1.4mu \hat{M}$
of the resulting manifold, and call the result $Y^j_0[y_*^j]$.
To form $Y^j_1[y_*^j]$, replace the $\bar{N}$-filling
of slope $y^j_{n_{{\textsc{g}}}+1}$ in 
$\partial^{{\textsc{g}}}_{{n_{{\textsc{g}}}+1}}\mkern-.7mu\hat{M} \mkern-2mu\subset\mkern-2mu Y^j_0[y_*^j]$
with the Dehn filling of slope $y^j_{n_{{\textsc{g}}}+1}$.
In addition, set $y^{n_{{\textsc{g}}}+1}_* \mkern-2.5mu:= \emptyset$ and 
$Y^{n_{{\textsc{g}}}+1}_0\mkern-1mu[\emptyset]\mkern-1mu:=\mkern-1mu Y$.

For positive $j \mkern-1.7mu\le\mkern-1.7mu n_{{\textsc{g}}}$, 
inductively assume
that for any $y_*^j \in ({\mathbb{Q}} \cup \infty)^{n_{{\textsc{g}}}+1-j}$ and $\theta \in \{0,1\}$,
any manifold of the form
$Y_{\theta}^j[y^j_*]\mkern.5mu$ is L/NTF-equivalent
if it is prime,
noting that
Proposition~\ref{prop: proof of lntf equivalence for b1>1 and seifert + N}
covers the base case of $j=1$.
For any $y \in {\mathbb{Q}} \cup \{\infty\}$,
$y^{j+1}_* := (y_{j+2}^j, \ldots, y_{n_{{\textsc{g}}}+1}^j) \in ({\mathbb{Q}} \cup \{\infty\})^{n_{{\textsc{g}}}+1-j}$,
$\theta \in \{0,1\}$, and
manifold of the form
$Y_{\theta}^{j+1}[y^{j+1}_*]$,
we can make matching choices of $\bar{N}$-filling gluing maps to obtain
\begin{equation}
Y_{\theta}^{j+1}[y^{j+1}_*]^{\bar{N}}\mkern-2mu(y)
\mkern1.5mu=\mkern1.5mu 
Y_j \cup Y_{\theta}^j[(y, y^{j+1}_*)].
\end{equation}
Suppose $Y_{\theta}^{j+1}[y^{j+1}_*]^{\bar{N}}\mkern-2mu(y)$
is prime, implying
$Y_{\theta}^j[(y, y^{j+1}_*)]$
is prime and hence L/NTF-equivalent
by inductive assumption.
Since $Y_{\theta}^j[(y, y^{j+1}_*)]$
and $Y_j$ are L/NTF-equivalent non-solid-torus graph manifolds
with torus boundary,
Proposition~\ref{prop: gluing prop for LNTF equivalent manifolds}
makes $Y_{\theta}^{j+1}[y^{j+1}_*]^{\bar{N}}\mkern-2mu(y)$ an L-space if and only if
it fails to admit a co-oriented taut foliation.
Since this holds for arbitrary $y \in {\mathbb{Q}} \cup \{\infty\}$,
Proposition~\ref{prop: lntf in terms of Nfilling}
tells us $Y_{\theta}^{j+1}[y^{j+1}_*]$ is L/NTF-equivalent.

Completing our induction on $j$, we conclude that
$Y := Y_0^{n_{{\textsc{g}}}+1}[\emptyset]$ is L/NTF-equivalent,
and that $Y_1^{n_{{\textsc{g}}}}[y]$ is L/NTF-equivalent
for any $y \in {\mathbb{Q}} \cup \{\infty\}$ for which $Y_1^{n_{{\textsc{g}}}}[y]$ is prime.
Thus, for any prime Dehn filling $Y(y)$,
Proposition~\ref{prop: gluing prop for LNTF equivalent manifolds}
tells us that the union
\begin{equation}
Y(y) = Y_{n_{{\textsc{g}}}} \cup_{\varphi_{n_{{\textsc{g}}}}} Y_1^{n_{{\textsc{g}}}}[y]
\end{equation}
is an L-space if and only if it fails to admit a co-oriented taut foliation,
and so
\begin{equation}
{\mathbb{P}}(H_1(\partial Y)) \setminus (\mathcal{L}(Y)\cup \mathcal{R}(Y))
=
\mathcal{F}(Y) \setminus \mathcal{R}(Y)
=
\mathcal{F}(Y).
\end{equation}
Inducting on tree height $k$ then
completes the proof. \qed
\smallskip

We next prove 
Theorem~\ref{thm: l space interval for graph manifolds}
and Proposition~\ref{prop: characterization of Floer simple graph manifolds}}
in tandem over the course of 
Sections 
\ref{ss: inductive set-up for main result} -- 
\ref{ss: orientable base, cases involving solid torus Y hat}.
The inductive program laid out in
Section
\ref{ss: inductive set-up for main result}
spans all three of the subsequent subsections.

\subsection{Inductive Set-up for Proof of 
Theorem~\ref{thm: l space interval for graph manifolds}
and Proposition~\ref{prop: characterization of Floer simple graph manifolds}}
\label{ss: inductive set-up for main result}
Both results 
hold automatically when $b_1(Y) \!>\! 1$.
This leaves the case of $b_1(Y) \!=\! 1$, so that 
$Y$ admits the recursive description in
(\ref{eq: recursive tree construction}),
with tree rooted at the Seifert fibered piece containing $\partial Y$.

Inductively assume that both
Theorem~\ref{thm: l space interval for graph manifolds}
and Proposition~\ref{prop: characterization of Floer simple graph manifolds}
hold for all tree manifolds with torus boundary, $b_1 = 1$, and
tree height $\le k-1$,
noting that Proposition
\ref{prop: l space interval for seifert fibered spaces}
covers the height zero case.
In addition, inductively assume that
Theorem~\ref{thm: l space interval for graph manifolds}
and Proposition~\ref{prop: characterization of Floer simple graph manifolds}}
hold for any tree manifold with torus boundary, $b_1 \mkern-2.5mu=\mkern-2.5mu 1$,
tree height $\le\mkern-2mu k$, and 
$\le \mkern-2mu n_{{\textsc{g}}} \mkern-2.5mu-\mkern-2mu 1$
daughter subtrees,
noting that
Proposition~\ref{prop: l space interval for seifert fibered spaces}
also covers the case of zero daughter subtrees.

For the remainder of the proof, we fix an arbitrary 
height $k$ tree manifold $Y\mkern-3mu$ with torus boundary and $b_1(Y)=1$,
as described in 
Section~\ref{subsection: conventions for graph manifolds}.
Thus, $Y$ has $n_{{\textsc{g}}}$
daughter subtrees $Y_1, \ldots, Y_{n_{{\textsc{g}}}}$,
attached via respective gluing maps
$\varphi_1, \ldots, \varphi_{n_{{\textsc{g}}}}$,
and the Seifert fibered piece containing $\partial Y\mkern-3.2mu$, 
at which we root the tree for $Y$,
is the Dehn filling of slope
$y^{\textsc{d}}_*=(y_1^{\textsc{d}}, \ldots, y_{n_{{\textsc{d}}}}^{\textsc{d}})$
of the ``foundation'' $\mkern-2.5mu\hat{M}$ of $Y\mkern-4.5mu,\mkern1mu$
where $\hat{M}$ is
either the trivial $S^1$-fibration 
over an 
$n\mkern-2.8mu:\mkern-.5mu=\mkern-1.8mu n_{{\textsc{d}}}\mkern-1.9mu+\mkern.5mu n_{{\textsc{g}}}\mkern-2.2mu
+\mkern-1.5mu 1$-punctured $S^2\mkern-1.8mu$ or the 
twisted $S^1$-fibration
over an 
$n$-punctured ${\mathbb{R}}{\mathbb{P}}^2\mkern-2.2mu$.
For any $y  \in {\mathbb{Q}} \cup \{\infty\}$,
let $\hat{Y}[y]$ denote the complement of 
$Y_{n_{{\textsc{g}}}} \mkern-5mu\setminus\mkern-2mu \partial Y_{n_{{\textsc{g}}}}$
in the Dehn filling $Y\mkern-1.5mu(y)$, so that we regard $Y\mkern-1.5mu(y)$ as the union
\begin{equation}
\label{eq: union for induction}
Y\mkern-1.9mu(y) = \hat{Y}[y] \mkern1.5mu \cup_{\varphi_{n_{{\textsc{g}}}}} 
\mkern-3.9muY_{n_{{\textsc{g}}}}.
\end{equation}
For any $y \in {\mathbb{Q}}$, our inductive assumptions make
Theorem~\ref{thm: l space interval for graph manifolds}
and Proposition~\ref{prop: characterization of Floer simple graph manifolds}
hold for $\hat{Y}[y]$ and $Y_{n_{{\textsc{g}}}}$, since
$Y_{n_{{\textsc{g}}}}$ has tree height $\le k-1$,
and since for $y \neq \infty$,
$\hat{Y}[y]$ is a $b_1 = 1$ tree manifold with
torus boundary,
$n_{{\textsc{g}}}-1 $ daughter subtrees, and tree height $\le k$.

\subsection{Non-orientable Base}
\label{ss: non-orientable base case}
Consider the case in which
$\hat{M}$ is $S^1$-fibered over a punctured ${\mathbb{R}}{\mathbb{P}}^2$.
First note that since the regular fiber class is torsion, its primitive lift 
$\tilde{f}^{\textsc{g}}_{n_{{\textsc{g}}}} \in H_1(\partial Y)$,
of slope $\infty$, is the rational longitude,
which means that $\infty \notin \mathcal{L}(Y)$.

Suppose there is some $y \mkern-2.5mu\neq\mkern-2.8mu \infty$ for which
$\mathcal{L}(\hat{Y}[y]) \mkern-2.8mu\neq\mkern-2mu \emptyset$.
Then, by inductive assumption, 
Proposition~\ref{prop: characterization of Floer simple graph manifolds}
tells us that
the daughter subtrees $Y_1, \ldots, Y_{n_{{\textsc{g}}}-1}\mkern-1.2mu$ are Floer simple,
with
$\infty \mkern-2mu\neq\mkern-2mu y^{\textsc{g}}_{i-} 
\mkern-2mu\ge\mkern-2mu y^{\textsc{g}}_{i+} 
\mkern-2mu\neq\mkern-2mu \infty$
for all $i \mkern-2mu\in\mkern-2mu \{1, \ldots, n_{{\textsc{g}}}\mkern-2mu-\mkern-.8mu1\}$,
where
$[[y^{\textsc{g}}_{i-}, y^{\textsc{g}}_{i+}]]:=
\varphi_i(\mathcal{L}(Y_i))$.
This conversely implies that $\hat{Y}[y]$ is Floer simple
for all $y \in \left<-\infty, +\infty\right>$.
Now, for each $y \in \left<-\infty, +\infty\right>$,
$\hat{Y}[y]$ is a non-solid-torus graph manifold, and so
Proposition~\ref{prop: gluing prop for LNTF equivalent manifolds}
tells us that 
the union $Y(y) = \hat{Y}[y] \cup Y_{n_{{\textsc{g}}}}$
in (\ref{eq: union for induction}) is an L-space if and only if
\begin{equation}
\label{eq: l-space condition for rp2 in graph manifold proof}
\mathcal{L}^{\circ}(\hat{Y}[y]) \cup 
\varphi^{{\mathbb{P}}}_{n_{{\textsc{g}}}*}(\mathcal{L}^{\circ}(Y_{n_{{\textsc{g}}}}))
= {\mathbb{P}}(H_1(\partial^{\textsc{g}}_{n_{{\textsc{g}}}}\mkern-2mu \hat{M})).
\end{equation}
Since, by inductive assumption, 
Theorem~\ref{thm: l space interval for graph manifolds}
implies
$\mathcal{L}^{\circ}(\hat{Y}[y]) \!=\! \left<-\infty, +\infty\right>$
for all $y \!\in\! \left<-\infty, +\infty\right>$,
(\ref{eq: l-space condition for rp2 in graph manifold proof})
holds if and only if
$Y_{n_{{\textsc{g}}}}$ is Floer simple and has
$\infty \mkern-2mu\neq\mkern-2mu y^{\textsc{g}}_{{n_{{\textsc{g}}}}-} 
\mkern-2mu\ge\mkern-2mu y^{\textsc{g}}_{{n_{{\textsc{g}}}}+} 
\mkern-2mu\neq\mkern-2mu \infty$,
in which case we consequently have $\mathcal{L}(Y) = \left<-\infty, +\infty\right>$.

Conversely, suppose that $\mathcal{L}(\hat{Y}[y]) = \emptyset$
for all $y \in {\mathbb{Q}}$.
Then for all $y \in {\mathbb{Q}}$, 
Proposition~\ref{prop: gluing prop for LNTF equivalent manifolds}
implies that $Y(y) = \hat{Y}[y] \cup Y_{n_{{\textsc{g}}}}$
is not an L-space,
and so $\mathcal{L}(Y) = \emptyset$.
Moreover, since there is $y \in {\mathbb{Q}}$
with $\mathcal{L}(\hat{Y}[y]) = \emptyset$,
our inductive assumption for
Proposition~\ref{prop: characterization of Floer simple graph manifolds}
tells us that either there is some $i \in \{1, \ldots, n_{{\textsc{g}}}\mkern-2mu-\mkern-2mu1\}$
for which $Y_i$ is not Floer simple,
or there is some Floer simple $Y_i$ failing to satisfy
$\infty \mkern-2mu\neq\mkern-2mu y^{\textsc{g}}_{i-} 
\mkern-2mu\ge\mkern-2mu y^{\textsc{g}}_{i+} 
\mkern-2mu\neq\mkern-2mu \infty$.

Thus, in either case, both
Theorem~\ref{thm: l space interval for graph manifolds}
and
Proposition~\ref{prop: characterization of Floer simple graph manifolds}
hold for $Y$.

\subsection{Orientable Base: Cases in which
$\boldsymbol{\hat{Y}[y]}$ is never a solid torus}
\label{ss: orientable base, no solid tori}
From now on, we assume that
the JSJ component of $Y$ containing $\partial Y$ has orientable base.

In this subsection of the proof, we consider the case
in which $\hat{Y}[y]$ is not a solid torus for any $y \in {\mathbb{Q}} \cup \{\infty\}$.
More precisely, we consider the case of a fixed
tree manifold $Y$ with torus boundary and
$b_1(Y) \mkern-2mu=\mkern-2mu 1$,
parameterized as in 
Section
\ref{subsection: conventions for graph manifolds},
with tree height $k\mkern-2mu>\mkern-2mu0$
and $n_{{\textsc{g}}}\mkern-3mu>\mkern-2mu0$ daughter subtrees,
where we demand that
if $n_{{\textsc{g}}}\mkern-3mu-\mkern-2mu1 \mkern-2mu=\mkern-2mu 0$,
then $y_i^{\textsc{d}} \notin {\mathbb{Z}}$ for at least two distinct values of 
$i \in \{1, \ldots, n_{{\textsc{d}}}\}$.

We begin by fixing some notation.
For all $k\in {\mathbb{Z}}_{>0}$,
define ${{{\hat{y}^{\mkern1.7mu0}_{\mkern-.3mu-}\mkern-2mu(k)}}}, \mkern1mu{{{\hat{y}^{\mkern1.7mu0}_{\mkern-.3mu+}\mkern-2mu(k)}}} \in {\mathbb{Q}} \cup \{\infty\}$ by
\begin{align}
{{{\hat{y}^{\mkern1.7mu0}_{\mkern-.3mu-}\mkern-2mu(k)}}}
&:=
-\mfrac{1}{k}{{\textstyle{\left(
\mkern12mu1 + \sum\limits_{i=1}^{n_{{\textsc{d}}}} \lfloor y_i^{\textsc{d}}k \rfloor
+
\sum\limits_{i=1}^{n_{{\textsc{g}}}-1} (\lceil y_{i+}^{\textsc{g}}k \rceil -1)
\right)}}},
       \\ \nonumber
{{{\hat{y}^{\mkern1.7mu0}_{\mkern-.3mu+}\mkern-2mu(k)}}}
&:=
-\mfrac{1}{k}{{\textstyle{\left(
\mkern-2mu-1 + \sum\limits_{i=1}^{n_{{\textsc{d}}}} \lceil y_i^{\textsc{d}}k \rceil
+
\sum\limits_{i=1}^{n_{{\textsc{g}}}-1} (\lfloor y_{i-}^{\textsc{g}}k \rfloor +1)
\right)}}}.
\end{align}
The endpoints $y_-, y_+ \in {\mathbb{Q}} \cup \{\infty\}$ defined in
Theorem~\ref{thm: l space interval for graph manifolds} are then given by
\begin{equation}
\label{eq: defs of y- y+ in no solid torus case}
y_-\mkern-2mu
:= \max_{k>0} \left( \mkern1mu {{{\hat{y}^{\mkern1.7mu0}_{\mkern-.3mu-}\mkern-2mu(k)}}} - {{\textstyle{\frac{1}{k}}}}
(\lceil y_{n_{\textsc{g}}\mkern-.7mu+}^{\textsc{g}}k \rceil - 1)\right),
\;\;\;
y_+\mkern-2mu
:= \min_{k>0} \left( \mkern1mu {{{\hat{y}^{\mkern1.7mu0}_{\mkern-.3mu+}\mkern-2mu(k)}}} - {{\textstyle{\frac{1}{k}}}}
(\lfloor y_{n_{\textsc{g}}\mkern-.7mu-}^{\textsc{g}}k \rfloor + 1)\right).
\end{equation}

Moreover, if we define 
the functions 
$\hat{y}_-, \hat{y}_+ \in {\mathbb{Q}} \cup \{\infty\}$
of $y \in {\mathbb{Q}} \cup \{\infty\}$ by
\begin{equation}
\label{eq: defs of hat y- hat y+ in no solid torus case}
\hat{y}_- 
:= \max_{k>0} \left(-{{\textstyle{\frac{1}{k}}}}\lfloor yk \rfloor 
+ {{{\hat{y}^{\mkern1.7mu0}_{\mkern-.3mu-}\mkern-2mu(k)}}} \right),
\;\;\;\;
\hat{y}_+ 
:= \min_{k>0} \left(-{{\textstyle{\frac{1}{k}}}}\lceil yk \rceil 
+ {{{\hat{y}^{\mkern1.7mu0}_{\mkern-.3mu+}\mkern-2mu(k)}}} \right),
\end{equation}
then by inductive assumption, we have
\begin{equation}
\mathcal{L}(\hat{Y}[y])
=
\begin{cases}
\{\hat{y}_-\}
= \{\hat{y}_+\}
   &
\hat{Y}[y]\;\text{not Floer simple}
    \\
[[\hat{y}_-, \hat{y}_+]]
   &
\hat{Y}[y]\;\text{Floer simple}
\end{cases}
\end{equation}
for all $y \in {\mathbb{Q}}$ for which $\mathcal{L}(\hat{Y}[y]) \neq \emptyset$.
Since $\hat{Y}[y]$ and $Y_{n_{{\textsc{g}}}}$ are each non-solid-torus
graph manifolds for all $y \in {\mathbb{Q}}$,
Proposition~\ref{prop: gluing prop for LNTF equivalent manifolds}
then implies, for each $y \in {\mathbb{Q}}$, that
\begin{equation}
\label{eq: L-space for y rational when union with Yhat}
y \in \mathcal{L}(Y)
\;\;\iff\;\;
\mathcal{L}^{\circ}\mkern-1.5mu(\hat{Y}[y])
\mkern1mu \cup \mkern1mu
\varphi_{\mkern-1.7mu n_{{\textsc{g}}}\mkern-.5mu *}^{{\mathbb{P}}}(\mathcal{L}^{\circ}\mkern-1.5mu(Y_{n_{{\textsc{g}}}}))
=
{\mathbb{Q}} \cup \{\infty\}
\end{equation}
for all $y \in {\mathbb{Q}}$.
Note that $\hat{Y}[\infty]$ is not a graph manifold,
being a non-solid-torus with compressible boundary, hence not prime.

We next prove some
basic rules about the behavior of $y_-$ and $y_+$.

\begin{claim}
\label{claim: non-solid-torus Yhat case, ineqs imply ineqs}
If $\mkern1.5mu y_{n_{{\textsc{g}}}+}^{\textsc{g}}, \hat{y}_-, y_-, y \in {\mathbb{Q}}$, then
\begin{equation}
y_{n_{{\textsc{g}}}+}^{\textsc{g}}\mkern-1.5mu  > \hat{y}_-
\;\;\;\iff\;\;\;
y \mkern1mu \in [\mkern1mu y_-, +\infty].
\end{equation}
If $\mkern1.5mu y_{n_{{\textsc{g}}}-}^{\textsc{g}}, \hat{y}_+, y_+, y \in {\mathbb{Q}}$, then
\begin{equation}
\label{eq: second line of non solid torus ineqs claim}
y_{n_{{\textsc{g}}}-}^{\textsc{g}}\mkern-1.5mu < \hat{y}_+ 
\;\;\;\iff\;\;\;
y \mkern1mu \in [-\infty, \mkern1mu y_+].
\end{equation}
\end{claim}
\begin{proof}[Proof of Claim]
Suppose that $y_{n_{{\textsc{g}}}+}^{\textsc{g}}, \hat{y}_-, y_-, y \in {\mathbb{Q}}$.  Then
\begin{align}
&&     &&
\mkern-50mu
y_{n_{{\textsc{g}}}+}^{\textsc{g}}
&> \hat{y}_- 
\mkern-40mu
&&&&
\nonumber
      \\
\label{eq: claim for nontorus case ineqs 1}
&&\iff &&
\mkern-50mu
y_{n_{{\textsc{g}}}+}^{\textsc{g}}
&> -{{\textstyle{\frac{1}{k}}}}\lfloor yk \rfloor
+ {{{\hat{y}^{\mkern1.7mu0}_{\mkern-.3mu-}\mkern-2mu(k)}}}
\mkern-40mu
&&\forall\; k\in {\mathbb{Z}}_{>0}
&&
       \\
\label{eq: claim for nontorus case ineqs 2}
&&\iff  &&
\mkern-50mu
\lceil y_{n_{{\textsc{g}}}+}^{\textsc{g}}k \rceil \mkern-2mu-\mkern-2mu 1
&\ge -\lfloor yk \rfloor
+ {{{\hat{y}^{\mkern1.7mu0}_{\mkern-.3mu-}\mkern-2mu(k)}}} k
\mkern-40mu
&&\forall\; k\in {\mathbb{Z}}_{>0}
&&
       \\
\label{eq: claim for nontorus case ineqs 3}
&&\iff  &&
\mkern-50mu
yk
&\ge 
{{{\hat{y}^{\mkern1.7mu0}_{\mkern-.3mu-}\mkern-2mu(k)}}} k
-(\lceil y_{n_{{\textsc{g}}}+}^{\textsc{g}}k \rceil - 1)
\mkern-40mu
&&\forall\; k\in {\mathbb{Z}}_{>0}
&&
       \\
\label{eq: claim for nontorus case ineqs 4}
&&\iff  &&
\mkern-50mu
y
&\ge 
y_-,
\mkern-40mu
&&
&&
\end{align}
where
(\ref{eq: defs of y- y+ in no solid torus case})
implies
(\ref{eq: claim for nontorus case ineqs 1}),
(\ref{eq: open interval to ceiling floor})
implies
(\ref{eq: claim for nontorus case ineqs 2}),
(\ref{eq: closed interval to ceiling floor})
implies
(\ref{eq: claim for nontorus case ineqs 3}),
and
(\ref{eq: defs of hat y- hat y+ in no solid torus case})
implies
(\ref{eq: claim for nontorus case ineqs 4}).
The proof of 
(\ref{eq: second line of non solid torus ineqs claim})
is nearly identical, but with signs reversed.
\end{proof}

\begin{claim}
\label{claim: y- and y+ outside long+ and long-}
If $\mkern1.5mu y_1^{\textsc{d}}, \ldots, y_{n_{{\textsc{d}}}}^{\textsc{d}},
y_{1\mkern-.5mu +}^{\textsc{g}}, \ldots, y_{n_{{\textsc{g}}}\mkern-.5mu+}^{\textsc{g}} \in {\mathbb{Q}},$ then
\begin{equation}
\label{eq: claim for y- > y+, y- > something}
y_- > 
{{\textstyle{-\sum\limits_{i=1}^{n_{{\textsc{d}}}}} y_i^{\textsc{d}}}}
-{{\textstyle{\sum\limits_{i=1}^{n_{{\textsc{g}}}}} y_{i+}^{\textsc{g}}}}.
\end{equation}
If $\mkern1mu y_1^{\textsc{d}}, \ldots, y_{n_{{\textsc{d}}}}^{\textsc{d}},
y_{1\mkern-.5mu -}^{\textsc{g}}, \ldots, y_{n_{{\textsc{g}}}\mkern-.5mu-}^{\textsc{g}} \in {\mathbb{Q}},$ then
\begin{equation}
\label{eq: claim for y- > y+, y+ < something}
y_+ <
{{\textstyle{-\sum\limits_{i=1}^{n_{{\textsc{d}}}}} y_i^{\textsc{d}}}}
-{{\textstyle{\sum\limits_{i=1}^{n_{{\textsc{g}}}}} y_{i-}^{\textsc{g}}}}.
\end{equation}
\end{claim}
\begin{proof}[Proof of Claim]
Writing $[\cdot]: {\mathbb{Q}} \to \left[0, 1\right>$ for the map sending 
$q \mapsto [q]:= q - \lfloor q \rfloor$, define
\begin{equation}
y'_-(k) :=
\mfrac{1}{k}\mkern-1.5mu\left( \mkern-2mu -1
\mkern2mu +\mkern1mu{{\textstyle{\sum\limits_{i=1}^{n_{{\textsc{d}}}} [y_i^{\textsc{d}} k] }}}
\mkern3.5mu+\mkern.5mu{{\textstyle{\sum\limits_{i=1}^{n_{{\textsc{g}}}} (1 - [-y_{i+}^{\textsc{g}} k]) }}}
\mkern-.5mu\right)
\end{equation}
for each $k \in {\mathbb{Z}}_{>0}$, so that
\begin{equation}
y_- 
\mkern1mu=\mkern1mu
{{\textstyle{-\mkern-1.5mu\sum\limits_{i=1}^{n_{{\textsc{d}}}}} y_i^{\textsc{d}}}}
-{{\textstyle{\sum\limits_{i=1}^{n_{{\textsc{g}}}}} y_{i+}^{\textsc{g}}}}
\mkern1.5mu+\mkern4mu \max_{k>0} \mkern1.5mu y'_-(k).
\end{equation}
In addition, let
\begin{equation}
s_{\mkern-1mu+} := \min\left\{k \mkern-1mu\in\mkern-1mu {\mathbb{Z}}_{>0} 
\mkern1mu|\mkern3mu 
y_1^{\textsc{d}}k, \ldots, \mkern-1mu y_{n_{{\textsc{d}}}}^{\textsc{d}}\mkern-2mu k,\mkern1mu
y_{1\mkern-.5mu +}^{\textsc{g}}\mkern-.5mu k, 
\ldots,\mkern-1mu y_{n_{{\textsc{g}}}\mkern-.5mu +}^{\textsc{g}}\mkern-.5mu k \in {\mathbb{Z}}\right\}
\end{equation}
denote the least common positive multiple of the denominators of the
$\{y_i^{\textsc{d}}\}$ and $\{y_{i+}^{\textsc{g}}\}$.

Suppose 
(\ref{eq: claim for y- > y+, y- > something})
fails to hold.  Then since $y'_-(k) \le 0$ for all $k \in {\mathbb{Z}}_{>0}$, we have
\begin{align}
0
&\mkern1.5mu\ge\mkern1.5mu
y'_-\mkern-1.2mu(1) \mkern1.5mu+\mkern1.5mu (s_{\mkern-1mu+}\mkern-3mu-\mkern-1.5mu1)
\mkern.3mu y'_-\mkern-1.2mu(s_{\mkern-1mu+}\mkern-3mu-\mkern-1.5mu1)
   \nonumber   \\
&\mkern1.5mu=\mkern1.5mu
\mkern-.8mu-\mkern-.2mu2 \mkern3.5mu+\mkern.5mu
{{\textstyle{\sum\limits_{i=1}^{n_{{\textsc{d}}}} 
([y_i^{\textsc{d}}] + [-y_i^{\textsc{d}}]) }}}
\mkern4.5mu+\mkern1.5mu
{{\textstyle{\sum\limits_{i=1}^{n_{{\textsc{g}}}} 
(1 + (1 - [-y_{i+}^{\textsc{g}}] - [y_{i+}^{\textsc{g}}] ) ) }}}
\mkern-.5mu
   \nonumber   \\
\label{eq: nG + nD - 2 le 0}
&\mkern1.5mu\ge\mkern1.5mu
-2 \mkern1mu+\mkern1mu
\left|\left\{i: y_i^{\textsc{d}} \in {\mathbb{Q}} \setminus {\mathbb{Z}}\right\}\right| 
\mkern1mu+\mkern1mu n_{{\textsc{g}}},
\end{align}
and likewise, we have
\begin{equation}
\label{eq: nG le 1}
0
\mkern3.5mu\ge\mkern2.5mu
s_{\mkern-1mu+}y'_-\mkern-1.2mu(s_{\mkern-1mu+})
\mkern1.5mu=\mkern1.5mu
-1 + \mkern-1mu{{\textstyle{\sum\limits_{i=1}^{n_{{\textsc{g}}}} 1}}}
\mkern1.5mu=\mkern2.5mu
n_{{\textsc{g}}}-1.
\end{equation}
The hypotheses of Section~\ref{ss: orientable base, no solid tori},
however, demand that either
$\left|\left\{i\mkern-1mu:\mkern-1mu y_i^{\textsc{d}} 
\mkern-1.5mu\in\mkern-1.5mu {\mathbb{Q}} \mkern-2mu\setminus\mkern-2mu 
{\mathbb{Z}}\right\}\right| 
\mkern-.8mu\ge\mkern-.8mu 2$
and
$n_{{\textsc{g}}}\mkern-3mu=\mkern-2mu1$,
contradicting
(\ref{eq: nG + nD - 2 le 0}),
or $n_{{\textsc{g}}} \!>\! 1$, contradicting
(\ref{eq: nG le 1}).
Thus
(\ref{eq: claim for y- > y+, y- > something}) holds,
and a similar argument proves
(\ref{eq: claim for y- > y+, y+ < something}).
\end{proof}

For the proof that
Theorem~\ref{thm: l space interval for graph manifolds}
and Proposition~\ref{prop: characterization of Floer simple graph manifolds}
hold for $Y$, we divide our argument into two main cases, depending on 
whether or not $\infty \in \mathcal{L}(Y)$.

\begin{prop}
\label{prop: infty in L(Y) case of main thm and prop, non solid torus Y hat}
$\infty \mkern-1.5mu\in\mkern-1.5mu \mathcal{L}(Y)\mkern-1mu$ if and only if
either 
condition $(\textsc{fs3}\mkern-.5mu)$ from 
Proposition~\ref{prop: characterization of Floer simple graph manifolds} holds,
in which case 
$\infty \mkern-2mu\in\mkern-2mu \mathcal{L}(Y) 
\mkern-2mu=\mkern-2mu [[y_-, y_+]]$,
or condition $(\textsc{nfs3}\mkern-.5mu)$ or $(\textsc{nfs4}\mkern-.5mu)$
from 
Proposition~\ref{prop: characterization of Floer simple graph manifolds} holds,
in which case $\mathcal{L}(Y) = \{y_-\} = \{y_+\} = \{\infty\}$.
\end{prop}
\begin{proof}

Since $\hat{Y}[\infty]$ is not a graph manifold,
our inductive assumptions fail to hold for $\hat{Y}[\infty]$,
but fortunately, $Y(\infty)$ has a simple structure.
The remark in Section
\ref{ss: Conventions for Seifert fibered spaces} implies
\begin{equation}
Y(\infty)
= \left(\Bigconnect\limits_{i=1}^{n_{{\textsc{d}}}} L(s_i^{\textsc{d}}, r_i^{\textsc{d}})\right)
\#
\left(\Bigconnect\limits_{i=1}^{n_{{\textsc{g}}}} 
Y_i\mkern-1mu\left(\mkern.5mu(\varphi_{i*}^{{\mathbb{P}}})^{-1}(\infty) \mkern.5mu\right) \right),
\end{equation}
where we have written $y_i^{\textsc{d}} ={r_i^{\textsc{d}}}/{s_i^{\textsc{d}}}$,
$L(s_i^{\textsc{d}}, r_i^{\textsc{d}})$ denotes the lens space
of slope $s_i^{\textsc{d}}/r_i^{\textsc{d}}$,
and
$Y_i\mkern-1mu\left(\mkern.5mu(\varphi_{i*}^{{\mathbb{P}}})^{-1}(\infty) \mkern.5mu\right)$
is the Dehn filling of the inverse image of the slope $\infty$.  Thus,
\begin{equation}
\label{eq: condition that infty in each Yg}
\infty \in \mathcal{L}(Y)
\;\;\;\iff\;\;\;
\infty \in 
\varphi_{i*}^{{\mathbb{P}}}(\mathcal{L}(Y_i)) 
\;\;\text{for all}\;\; i \in \{1, \ldots, n_{{\textsc{g}}}\}.
\end{equation}

Conditions 
$(\textsc{nfs3}\mkern-.5mu)$ and $(\textsc{fs3}\mkern-.5mu)$
jointly exhaust the cases in which
the right-hand condition of 
(\ref{eq: condition that infty in each Yg})
holds and all daughter subtrees $Y_1, \dots, Y_{n_{{\textsc{g}}}}$ are Floer simple.
Condition $(\textsc{nfs4}\mkern-.5mu)$,
on the other hand, describes all cases in which the
the right-hand condition of 
(\ref{eq: condition that infty in each Yg}) holds
and at least one $Y_j$ is not Floer simple.

If $(\textsc{nfs4}\mkern-.5mu)$ holds, then,
permuting the daughter subtrees without loss of generality so that 
$\varphi^{{\mathbb{P}}}_{n_{{\textsc{g}}}*}(\mathcal{L}(Y_{n_{{\textsc{g}}}})) 
\mkern-2.5mu=\mkern-2.5mu \{\infty\}$, we have 
$\mathcal{L}^{\circ}(Y_{n_{{\textsc{g}}}}) \mkern-2mu=\mkern-2mu \emptyset$,
which, by
(\ref{eq: L-space for y rational when union with Yhat}),
implies $\mathcal{L}(Y) \cap {\mathbb{Q}} \mkern-2mu=\mkern-2mu \emptyset$,
so that
$\mathcal{L}(Y) \mkern-2mu=\mkern-2mu \{\infty\}$.
Moreover, the fact that
$\varphi^{{\mathbb{P}}}_{n_{{\textsc{g}}}*}(\mathcal{L}(Y_{n_{{\textsc{g}}}})) 
\mkern-2.5mu=\mkern-2.5mu \{\infty\}$
implies, by inductive assumption, that
$y_{n_{{\textsc{g}}}-}^{\textsc{g}} 
\mkern-2mu=\mkern-2mu y_{n_{{\textsc{g}}}+}^{\textsc{g}} \mkern-2mu=\mkern-2mu \infty$.
Thus $y_-$ and $y_+$ each have infinite summands,
and so $y_- \mkern-2mu=\mkern-2mu y_+ \mkern-2mu=\mkern-2mu \infty$.

Next suppose that $(\textsc{nfs3}\mkern-.5mu)$ holds, and set
\begin{equation}
I_{-\infty}:= 
\{i :
[\mkern-.4mu[y_{i-}^{\textsc{g}}, y_{i+}^{\textsc{g}}]\mkern-.4mu]
\mkern-2.5mu=\mkern-2.5mu [-\infty, y_{i+}^{\textsc{g}}]\},
\;\;\;\;
I_{+\infty}:=
\{i : 
[\mkern-.4mu[y_{i-}^{\textsc{g}}, y_{i+}^{\textsc{g}}]\mkern-.4mu]
\mkern-2.5mu=\mkern-2.5mu [y_{i-}^{\textsc{g}} ,+\infty]\}.
\end{equation}
We can also define the analogous sets for $\hat{Y}[y]$:
\begin{equation}
I_{-\infty}^{\hat{Y}[y]} := I_{-\infty} \cap 
\{1, \ldots, n_{{\textsc{g}}}\mkern-3.5mu-\mkern-2.5mu1\}, 
\;\;\;\;
I_{+\infty}^{\hat{Y}[y]} := I_{+\infty} \cap 
\{1, \ldots, n_{{\textsc{g}}}\mkern-3.5mu-\mkern-2.5mu1\}.
\end{equation}
Since $I_{-\infty}$ and $I_{+\infty}$ are nonempty for
$(\textsc{nfs3}\mkern-.5mu)$, we know that $y_-$ and $y_+$
each have infinite summands, implying $y_- = y_+ = \infty$.
Assume without loss of
generality that $n_{{\textsc{g}}} \in I_{+\infty}$.
Thus $I_{-\infty}^{\hat{Y}[y]} \neq \emptyset$, and
$I_{+\infty}^{\hat{Y}[y]}$ is either empty or not.
By inductive assumption,
Proposition~\ref{prop: characterization of Floer simple graph manifolds} holds
for $\hat{Y}[y]$ for all $y \in {\mathbb{Q}}$.
Thus, for each $y \in {\mathbb{Q}}$,
either $I_{-\infty}^{\hat{Y}[y]}$ and $I_{+\infty}^{\hat{Y}[y]}$ 
are both nonempty,
making $(\textsc{nfs3}\mkern-.5mu)$ hold for $\hat{Y}[y]$,
so that $\mathcal{L}(\hat{Y}[y]) = \{\infty\}$;
or, 
$I_{-\infty}^{\hat{Y}[y]} \neq \emptyset$
and $I_{+\infty}^{\hat{Y}[y]} = \emptyset$,
making $(\textsc{fs3}\mkern-.5mu)$ hold for $\hat{Y}[y]$,
so that $\mathcal{L}(\hat{Y}[y]) = [\hat{y}_-, +\infty]$.
In both cases, the right-hand side of
(\ref{eq: L-space for y rational when union with Yhat})
fails to hold for all $y \in {\mathbb{Q}}$,
and we are left with $\mathcal{L}(Y) = \{\infty\} = \{y_-\} = \{y_+\}$.

This leaves us with the case in which $(\textsc{fs3}\mkern-.5mu)$ holds,
for which we first consider the subcase $I_{+\infty} = \emptyset$
and $I_{-\infty} \mkern-2mu\neq\mkern-2mu \emptyset$,
implying
$[[y_-, y_+]] \mkern-2mu=\mkern-2mu [y_-, +\infty]$.
Without loss of generality, assume 
$n_{{\textsc{g}}} \mkern-2mu\in\mkern-2mu I_{-\infty}$, 
so that
$[[y_{n_{{\textsc{g}}}-}^{\textsc{g}}, y_{n_{{\textsc{g}}}+}^{\textsc{g}}]]
\mkern-2mu=\mkern-2mu [-\infty, y_{n_{{\textsc{g}}}+}^{\textsc{g}}]$.
Since 
$\mathcal{L}(\hat{Y}[y]) \mkern-2mu=\mkern-2mu [[\hat{y}_-, \hat{y}_+]]$
for all $y \mkern-2mu\in\mkern-2mu {\mathbb{Q}}$ by inductive assumption,
we know that either
\begin{enumerate}
\item[(a)] $I^{\hat{Y}[y]}_{-\infty} = I^{\hat{Y}[y]}_{+\infty} = \emptyset$,
or
\item[(b)]
$I^{\hat{Y}[y]}_{-\infty} \neq \emptyset$,
with $\mathcal{L}(\hat{Y}[y]) = [\hat{y}_-, +\infty]$.
\end{enumerate}
In case (a), the condition
$I^{\hat{Y}[y]}_{-\infty} = I^{\hat{Y}[y]}_{+\infty} = \emptyset$,
together with
(\ref{eq: condition that infty in each Yg}), implies
\begin{equation}
\label{subsubcase of n_g - 1 guys satisfying y_i- ge y_-i+}
\infty \neq y_{i-}^{\textsc{g}} \ge  y_{i+}^{\textsc{g}} \neq \infty
\text{ for all } i \in \{1, \ldots, n_{{\textsc{g}}}-1\}.
\end{equation}
Thus, if we apply
Claim~\ref{claim: y- and y+ outside long+ and long-}
to
(\ref{eq: defs of hat y- hat y+ in no solid torus case}),
respectively substituting $\hat{y}_-$, $\hat{y}_+$,
$n_{\textsc{d}}\mkern-1.5mu + \mkern-1.5mu 1$, $n_{\textsc{g}}-1$,
and $(y, y_1^{\textsc{d}}, \ldots, y_{n_{\textsc{d}}}^{\textsc{d}})$
for $y_-$, $y_+$, $n_{\textsc{d}}$, $n_{\textsc{g}}$
and $(y_1^{\textsc{d}}, \ldots, y_{n_{\textsc{d}}}^{\textsc{d}})$
in the statement of the claim,
then we obtain
\begin{equation}
\label{eq: long ineq that gives nice hat bound for fs3 case, case a}
\infty 
\neq 
\hat{y}_-
\;\;>\;
-y -{{\textstyle{\sum\limits_{i=1}^{n_{{\textsc{d}}}}} y_i^{\textsc{d}}}}
-{{\textstyle{\sum\limits_{i=1}^{n_{{\textsc{g}}}-1}} y_{i+}^{\textsc{g}}}}
\;\ge\;
-y -{{\textstyle{\sum\limits_{i=1}^{n_{{\textsc{d}}}}} y_i^{\textsc{d}}}}
-{{\textstyle{\sum\limits_{i=1}^{n_{{\textsc{g}}}-1}} y_{i-}^{\textsc{g}}}}
\;>\;\;
\hat{y}_+
\neq
\infty,
\end{equation}
with the outer, strict, inequalities resulting from
Claim~\ref{claim: y- and y+ outside long+ and long-},
and the middle, non-strict, inequality resulting from 
(\ref{subsubcase of n_g - 1 guys satisfying y_i- ge y_-i+}).
Thus, for any $y \in {\mathbb{Q}}$, we have
\begin{equation}
\mathcal{L}^{\circ}\mkern-1.5mu(\hat{Y}[y])
\mkern1mu \cup \mkern1mu
\varphi_{\mkern-1.7mu n_{{\textsc{g}}}\mkern-.5mu *}^{{\mathbb{P}}}(\mathcal{L}^{\circ}\mkern-1.5mu(Y_{n_{{\textsc{g}}}}))
=\mkern1.5mu
(\left<\hat{y}_-, +\infty\right] \cup \left[-\infty, \hat{y}_+\right>)
\mkern1.5mu\cup\mkern1.5mu
\left<-\infty, y_{n_{{\textsc{g}}}+}^{\textsc{g}}\right>
\end{equation}
by inductive assumption,
and so
(\ref{eq: L-space for y rational when union with Yhat})
implies that $y \in \mathcal{L}(Y)$
if and only if $y_{n_{{\textsc{g}}}+}^{\textsc{g}} > \hat{y}_-$,
which, by
Claim
\ref{claim: non-solid-torus Yhat case, ineqs imply ineqs},
occurs if and only if $y \in [y_-, +\infty]$.
On the other hand, in case~(b),
with $\mathcal{L}^{\circ}(\hat{Y}[y]) = \left<\hat{y}_-, +\infty\right>$,
(\ref{eq: L-space for y rational when union with Yhat})
again implies, for any $y \in {\mathbb{Q}}$, that
$y \in \mathcal{L}(Y)$
if and only if $y_{n_{{\textsc{g}}}+}^{\textsc{g}} > \hat{y}_-$,
which again occurs if and only if $y \in [y_-, +\infty]$.
Thus, whether case (a) or (b) holds, we always have
$\infty \in \mathcal{L}(Y) = [y_-, +\infty] = [[y_-, y_+]]$.

If $(\textsc{fs3}\mkern-.5mu)$ holds with
$I_{+\infty} \neq \emptyset$
and $I_{-\infty} = \emptyset$, then an argument precisely analogous
to that in the preceding paragraph shows that
$\infty \in \mathcal{L}(Y) = [-\infty, y_+] = [[y_-, y_+]]$.

Lastly, suppose that
$(\textsc{fs3}\mkern-.5mu)$ holds with
$I_{+\infty} = I_{-\infty} = \emptyset$.
This, as well as the fact that
$\infty \in [[y_{i-}^{\textsc{g}}, y_{i+}^{\textsc{g}}]]$
for each $i \in \{1, \ldots, n_{{\textsc{g}}}\}$,
implies that
$\infty \neq y_{i-}^{\textsc{g}} \ge  y_{i+}^{\textsc{g}} \neq \infty$
for all $i \in \{1, \ldots, n_{{\textsc{g}}}\}$,
which, by 
Claim~\ref{claim: y- and y+ outside long+ and long-},
implies that
$\infty \neq y_- > y_+ \neq \infty$.
Moreover, applying
Claim~\ref{claim: y- and y+ outside long+ and long-}
to
(\ref{eq: defs of hat y- hat y+ in no solid torus case})
as in case (a) above
yields
(\ref{eq: long ineq that gives nice hat bound for fs3 case, case a}),
so that we also have
$\infty \neq \hat{y}_- > \hat{y}_+ \neq \infty$.
By inductive assumption, we then have
\begin{equation}
\mathcal{L}^{\circ}\mkern-1.5mu(\hat{Y}[y])
\mkern1mu \cup \mkern1mu
\varphi_{\mkern-1.7mu n_{{\textsc{g}}}\mkern-.5mu *}^{{\mathbb{P}}}(\mathcal{L}^{\circ}\mkern-1.5mu(Y_{n_{{\textsc{g}}}}))
=\mkern1.5mu
(\left<\hat{y}_-, +\infty\right] \mkern-1.5mu\cup\mkern-1.5mu \left[-\infty, \hat{y}_+\right>)
\mkern1.5mu\cup\mkern1.5mu
(\left<y_{n_{{\textsc{g}}}-}^{\textsc{g}}, +\infty\right]
\mkern-2.5mu\cup\mkern-2.5mu \left[-\infty, y_{n_{{\textsc{g}}}+}^{\textsc{g}}\right>)
\end{equation}
for any $y \in {\mathbb{Q}}$,
and so
(\ref{eq: L-space for y rational when union with Yhat})
implies that $y \in \mathcal{L}(Y)$
if and only if 
$y_{n_{{\textsc{g}}}+}^{\textsc{g}} > \hat{y}_-$
or
$y_{n_{{\textsc{g}}}-}^{\textsc{g}} < \hat{y}_+$
which, by
Claim
\ref{claim: non-solid-torus Yhat case, ineqs imply ineqs},
occurs if and only if $y \in [y_-, +\infty] \cup [-\infty, y_+]$.
We therefore have
$\infty \in \mathcal{L}(Y) = [[y_-, y_+]] = [y_-, +\infty] \cup [-\infty, y_+]$,
completing our proof.
\end{proof}

\begin{prop}
\label{prop: FS1, FS2, and NFS2}
Suppose $\infty \notin \mathcal{L}(Y)$.
Then $\mathcal{L}(Y)\neq \emptyset$
if and only if 
either 
$(\textsc{nfs2}\mkern-.5mu)$ holds for $Y$,
in which case $\mathcal{L}(Y) = \{y_-\} = \{y_+\}$,
or 
$(\textsc{fs1}\mkern-.5mu)$ or $(\textsc{fs2}\mkern-.5mu)$
holds for $Y$,
in which case $\mathcal{L}(Y) = [[y_-, y_+]]$.
\end{prop}
\begin{proof}
We first observe that if any daughter subtree of $Y$
fails to be Floer simple, then $\mathcal{L}(Y) = \emptyset$.
That, if we choose $Y_{n_{{\textsc{g}}}}$ to be non Floer simple,
then by
(\ref{eq: L-space for y rational when union with Yhat}),
$\mathcal{L}^{\circ}(Y_{n_{{\textsc{g}}}}) \mkern-2mu=\mkern-2mu \emptyset$
implies $\mathcal{L}(Y) \cap {\mathbb{Q}} \mkern-2mu=\mkern-2mu \emptyset$,
which, since $\infty \notin \mathcal{L}(Y)$, implies $\mathcal{L}(Y) = \emptyset$.
Thus, we henceforth assume
the daughter subtrees 
$Y_1$, \ldots, $Y_{n_{{\textsc{g}}}}$ are all Floer simple.

Let $I_{\pm\infty}, I_{\cap} \subset \{1, \ldots, n_{{\textsc{g}}}\}$ denote the sets
\begin{align}
I_{\pm\infty}\!
\mkern2mu&:=\mkern2mu
\left\{i: [[y_{i-}^{\textsc{g}}, y_{i+}^{\textsc{g}}]]
=\left<-\infty, +\infty\right>\right\},
     \\ \nonumber
I_{\cap}
\mkern1mu&:=\mkern2mu
\left\{i: [[y_{i-}^{\textsc{g}}, y_{i+}^{\textsc{g}}]]
=[y_{i-}^{\textsc{g}}, +\infty]\cap
[-\infty, y_{i+}^{\textsc{g}}]\right\}.
\end{align}
Then by 
(\ref{eq: condition that infty in each Yg})
from the proof of
Proposition~\ref{prop: infty in L(Y) case of main thm and prop, non solid torus Y hat},
we know that
\begin{equation}
\label{eq: one guy without infinity}
I_{\pm\infty} \cup I_{\cap} \neq \emptyset.
\end{equation}

Suppose $I_{\pm\infty} \mkern-3mu\neq\mkern-2mu \emptyset$.
We claim that in this case, $\mathcal{L}(Y) \neq \emptyset$
if and only if 
$(\textsc{fs1}\mkern-.5mu)$ holds for $Y$,
in which case $\mathcal{L}(Y) = [[y_-, y_+]]$.
First note that
$Y$ fails to satisfy any of the conditions for
nonempty $\mathcal{L}(Y)$ in
Proposition~\ref{prop: characterization of Floer simple graph manifolds}
except possibly
$(\textsc{fs1}\mkern-.5mu)$.
Suppose that $Y$ satisfies 
$(\textsc{fs1}\mkern-.5mu)$.
Then for each
$i \in \{1, \ldots, n_{{\textsc{g}}}-1\}$,
we have $\infty \in [[y_{i-}^{\textsc{g}}, y_{i+}^{\textsc{g}}]]^{\circ}$,
implying
$\infty \neq y_{i-}^{\textsc{g}} \ge y_{i+}^{\textsc{g}} \neq \infty$,
which by Claim~\ref{claim: y- and y+ outside long+ and long-}
implies $\infty \neq \hat{y}_- > \hat{y}_+ \neq \infty$
for all $y \in {\mathbb{Q}}$.
By inductive assumption, we then have
$\infty \mkern-2mu\in\mkern-2mu 
\mathcal{L}^{\circ}\mkern-1.8mu(\mkern.2mu\hat{Y}\mkern-.2mu[y])$
for all $y \in {\mathbb{Q}}$.  Thus, since
(\ref{eq: L-space for y rational when union with Yhat}) tells us that
\begin{equation}
\label{eq: gluing condition in pm infty case}
\mathcal{L}^{\circ}\mkern-1.8mu(\mkern.2mu\hat{Y}\mkern-.2mu[y])
\cup \left<-\infty, +\infty\right> = {\mathbb{Q}} \cup \{\infty\}
\;\;\;
\iff
\;\;\;
y\in \mathcal{L}(Y),
\end{equation}
we have $\mathcal{L}(Y) = {\mathbb{Q}} = \left<-\infty, +\infty\right> = [[y_-, y_+]]$.
Suppose instead we are given that $\mathcal{L}(Y) \neq \emptyset$.
Then (\ref{eq: gluing condition in pm infty case})
tells us that
$\infty \mkern-2mu\in\mkern-2mu 
\mathcal{L}^{\circ}\mkern-1.8mu(\mkern.2mu\hat{Y}\mkern-.2mu[y])$
for all $y \mkern-2mu\in\mkern-2mu \mathcal{L}(Y)$.
By inductive assumption this implies,
for each $y \mkern-2mu\in\mkern-2mu \mathcal{L}(Y)$,
that $(\textsc{fs3}\mkern-.5mu)$ holds for $\hat{Y}[y]$,
with $\infty \mkern-2mu\in\mkern-2mu [[\hat{y}_-, \hat{y}_+]]^{\circ}$.
Since $\infty \neq \{\hat{y}_-, \hat{y}_+\}$, we know that $\hat{y}_-$
and $\hat{y}_+$ cannot have infinite summands, and so
$I_{-\infty}^{\hat{Y}[y]} =I_{+\infty}^{\hat{Y}[y]} = \emptyset$.
This, in addition to the fact that 
that $(\textsc{fs3}\mkern-.5mu)$ holds for $\hat{Y}[y]$,
implies that $\infty \in [[y_{i-}^{\textsc{g}}, y_{i+}^{\textsc{g}}]]^{\circ}$
for all $i \in \{1, \ldots, n_{{\textsc{g}}}-1\}$,
and thus
$(\textsc{fs1}\mkern-.5mu)$ holds for $Y$.

Next, consider the case in which
$I_{\pm\infty} \mkern-3mu=\mkern-2mu \emptyset$,
so that by
(\ref{eq: one guy without infinity}), we have
$I_{\cap} \mkern-2mu\neq\mkern-2mu \emptyset$.
Assume, without loss of generality, that
$n_{{\textsc{g}}} \in I_{\cap}$.
Then since $\infty \notin 
[[y_{n_{{\textsc{g}}}-}^{\textsc{g}}, y_{n_{{\textsc{g}}}+}^{\textsc{g}}]]^{\circ}$,
the L-space gluing condition in 
(\ref{eq: L-space for y rational when union with Yhat}) again tells us that
$\infty \mkern-2mu\in\mkern-2mu 
\mathcal{L}^{\circ}\mkern-1.8mu(\mkern.2mu\hat{Y}\mkern-.2mu[y])$
for all $y \in \mathcal{L}(Y)$.
Just as in the preceding paragraph, we deduce from this
that $\infty \in [[y_{i-}^{\textsc{g}}, y_{i+}^{\textsc{g}}]]^{\circ}$
for all $i \in \{1, \ldots, n_{{\textsc{g}}}-1\}$, and that this
implies that
$\mathcal{L}^{\circ}\mkern-1.8mu(\mkern.2mu\hat{Y}\mkern-.2mu[y])
= [[\hat{y}_-, \hat{y}_+]]$, with
$\infty \neq \hat{y}_- > \hat{y}_+ \neq \infty$,
for all $y \in {\mathbb{Q}}$.
Since $y_-$ and $y_+$ have only finite summands, 
we also know that
$y_-, y_+ \in {\mathbb{Q}}$.
We therefore have
\begin{align}
y \in \mathcal{L}(Y)
\;\;
&\iff\;\;
\left(\left<\hat{y}_-, +\infty\right]
\mkern-2mu\cup\mkern-2mu\left[-\infty, \hat{y}_+\right>\right)
\mkern1mu\cup\mkern1mu
\left(\left<y_{n_{{\textsc{g}}}-}^{\textsc{g}}, +\infty\right]
\mkern-3mu\cap\mkern-3mu\left[-\infty, y_{n_{{\textsc{g}}}+}^{\textsc{g}}\right>\right)
= {\mathbb{Q}} \cup \{\infty\}
       \\ \nonumber
\;\;
&\iff\;\;
y_{n_{{\textsc{g}}}+}^{\textsc{g}}\mkern-1mu > \hat{y}_-
\;\;\text{and}\;\;
y_{n_{{\textsc{g}}}-}^{\textsc{g}}\mkern-1mu < \hat{y}_+
       \\ \nonumber
\;\;
&\iff\;\;
y \in [y_-, +\infty] \cap [-\infty, y_+],
\end{align}
where the first line is due to 
(\ref{eq: L-space for y rational when union with Yhat}),
and the third line is due to
Claim~\ref{claim: non-solid-torus Yhat case, ineqs imply ineqs}.
Thus, if $y_- > y_+$, then $\mathcal{L}(Y) = \emptyset$;
if $y_- = y_+$, then 
$(\textsc{nfs2}\mkern-.5mu)$ holds for $Y$,
with $\mathcal{L}(Y) = \{y_-\} = \{y_+\}$;
and if $y_- < y_+$, then
$(\textsc{fs2}\mkern-.5mu)$ holds for $Y$,
with $\mathcal{L}(Y) = [[y_-, y_+]] = [y_-, y_+]$.
\end{proof}

\subsection{Orientable Base: 
Cases involving solid tori $\boldsymbol{\hat{Y}[y]}$}
\label{ss: orientable base, cases involving solid torus Y hat}
In this final subsection of the proof of
Theorem~\ref{thm: l space interval for graph manifolds}
and Proposition~\ref{prop: characterization of Floer simple graph manifolds},
we consider all cases
of $Y$ for which $\hat{Y}[y]$,
defined in 
(\ref{eq: union for induction}),
is a solid torus for some $y \in {\mathbb{Q}} \cup \{\infty\}$.
Recall that we have fixed a
tree manifold $Y$ with torus boundary, 
$b_1(Y) \mkern-2mu=\mkern-2mu 1$, tree height $k\mkern-2mu>\mkern-2mu0$,
and $n_{{\textsc{g}}}\mkern-2mu>\mkern-2mu0$ daughter subtrees,
as parameterized in 
Section~\ref{subsection: conventions for graph manifolds}.
Since $\hat{Y}[y]$ contains no incompressible tori, we must have $n_{{\textsc{g}}}-1 = 0$.
Thus, for any $y \in {\mathbb{Q}}$,
$\hat{Y}[y]$ is Seifert fibered over the disk,
and is a solid torus if and only if it has one or fewer exceptional fibers.
This occurs for $y \in {\mathbb{Z}}$ if and only if
the set
$\{y^{\textsc{d}}_1, \ldots, y^{\textsc{d}}_{n_{{\textsc{d}}}}\}$
contains $\le 1$ nonintegers.
Since $\hat{Y}[y]$ and $Y\!$ are invariant under reparameterizations
$(y^{\textsc{d}}_1, \ldots, y^{\textsc{d}}_{n_{{\textsc{d}}}})
\mapsto (y^{\textsc{d}}_1 + z_1, \ldots, y^{\textsc{d}}_{n_{{\textsc{d}}}} + z_{n_{{\textsc{d}}}})$
with $\sum_{i=1}^{n_{{\textsc{d}}}} z_i \mkern-2mu =\mkern-2mu 0$ and each 
$z_i \mkern-1.8mu\in\mkern-1.8mu {\mathbb{Z}}$,
or under the operation of forgetting a fiber complement with Dehn filling
of slope $\pi_i(-\tilde{h}_i) \mkern-2mu=\mkern-2mu 0$,
we may assume, without loss of generality, that $n_{{\textsc{d}}} \mkern-2mu=\mkern-2mu 1$.
Thus $n_{{\textsc{g}}}\mkern-3mu=\mkern-1.5mu n_{{\textsc{d}}} \mkern-3mu=\mkern-2mu 1$.
\smallskip

If the unique daughter subtree $Y_1$ satisfies
$\mathcal{L}(Y_1) \mkern-2mu=\mkern-2mu \emptyset$, then
the gluing
Propositions~\ref{prop: gluing prop for LNTF equivalent manifolds}
and
\ref{prop: l-space gluing for compressible boundary}
imply
that $\mathcal{L}(Y) = \emptyset$,
as predicted for $Y$ by
Proposition~\ref{prop: characterization of Floer simple graph manifolds}.

We therefore henceforth assume $\mathcal{L}(Y_1) \neq \emptyset$.
Since $Y$ has tree height $k$, $Y_1$ has tree height $k-1$.
Thus, by inductive assumption
as laid out in 
Section~\ref{ss: inductive set-up for main result},
$Y_1$ satisfies 
Theorem~\ref{thm: l space interval for graph manifolds}
and Proposition~\ref{prop: characterization of Floer simple graph manifolds},
with
\begin{equation}
\mathcal{L}(Y_1)
=
\begin{cases}
\{y^{\textsc{g}}_{1-}\}
=\{y^{\textsc{g}}_{1+}\}
   &
Y_1\;\text{not Floer simple},
    \\
[[y^{\textsc{g}}_{1-}
,y^{\textsc{g}}_{1+}]]
   &
Y_1\;\text{Floer simple}.
\end{cases}
\end{equation}

We proceed, once again, by fixing some notation.
Recall the definitions of $y_-$ and $y_+$:
\begin{align}
\label{eq: defs of y- and y+ in possibly solid torus case}
y_-
&:= \max_{k>0} 
-{{\textstyle{\frac{1}{k}}}}(1 + \lfloor y_1^{\textsc{d}} k \rfloor +
(\lceil y_{1+}^{\textsc{g}} k \rceil - 1) ),
        \\ \nonumber
y_+
&:= \min_{k>0} 
-{{\textstyle{\frac{1}{k}}}}(-1 + \lceil y_1^{\textsc{d}} k \rceil +
(\lfloor y_{1-}^{\textsc{g}} k \rfloor + 1) ),
\end{align}
where, as always,
we define $y_-$ or $y_+$ to be infinite if any of its summands is infinite.

For any $y \neq \infty$, $\hat{Y}[y]$ is Seifert fibred over the disk,
hence is Floer simple, allowing us to write
$[[\hat{y}_-, \hat{y}_+]] := \mathcal{L}(\hat{Y}[y])$.
Moreover, when $\hat{Y}[y]$ is {\em{not}} a solid torus, we have
\begin{equation}
\{y, y_1^{\textsc{d}}\} \cap {\mathbb{Z}} = \emptyset;
\;\;\;\;\;
\begin{array}{l}
\hat{y}_-
= \max\limits_{k>0} -{{\textstyle{\frac{1}{k}}}}
(1 + \lfloor y k \rfloor + \lfloor y_1^{\textsc{d}} k \rfloor),
        \\
\hat{y}_+
= \min\limits_{k>0} -{{\textstyle{\frac{1}{k}}}}
(-1 + \lceil y k \rceil + \lceil y_1^{\textsc{d}} k \rceil),
\end{array}
\end{equation}
but when  $\hat{Y}[y]$ {\it{is}} a solid torus, we have
\begin{equation}
\{y, y_1^{\textsc{d}}\} \cap {\mathbb{Z}} \neq \emptyset;
\;\;\;\;\;\;
\hat{y}_- 
= \hat{y}_+
= -y-y_1^{\textsc{d}}.\;\;\;
\end{equation}

We next prove the analogs of Claims
\ref{claim: non-solid-torus Yhat case, ineqs imply ineqs}
and
\ref{claim: y- and y+ outside long+ and long-}
from 
Section~\ref{ss: orientable base, no solid tori}
\begin{claim}
\label{claim: Yhat solid torus sometimes, translating y to yhat ineqs}
For any 
$y_1^{\textsc{d}},  y \in {\mathbb{Q}}$, we have
\begin{align}
\label{eq: y ge y_-  iff two cases}
y \in [y_-, +\infty]
\;
&\iff
\;
\begin{cases}
y_{1+}^{\textsc{g}} 
> 
\,\hat{y}_-
&\mkern-10mu
\;\;\;
\text{if}\;\,\{y, y_1^{\textsc{d}}\} \cap {\mathbb{Z}} = \emptyset
 \\
y_{1+}^{\textsc{g}} 
\ge
\,\hat{y}_-
&\mkern-10mu
\;\;\;
\text{if}\;\,\{y, y_1^{\textsc{d}}\} \cap {\mathbb{Z}} \neq \emptyset
\end{cases}
\;\;\;\;\;\;\;\;\;
\text{if}\;\;
y_{1+}^{\textsc{g}} \in {\mathbb{Q}};
          \\
\label{eq: y le y_+  iff two cases}
y \in [-\infty, y_+]
\;
&\iff
\;
\begin{cases}
y_{1-}^{\textsc{g}} 
< 
\,\hat{y}_+
&\mkern-10mu
\;\;\;
\text{if}\;\,\{y, y_1^{\textsc{d}}\} \cap {\mathbb{Z}} = \emptyset
 \\
y_{1-}^{\textsc{g}} 
\le
\,\hat{y}_+
&\mkern-10mu
\;\;\;
\text{if}\;\,\{y, y_1^{\textsc{d}}\} \cap {\mathbb{Z}} \neq \emptyset
\end{cases}
\;\;\;\;\;\;\;\;\;
\text{if}\;\;
y_{1-}^{\textsc{g}} \in {\mathbb{Q}}.
\end{align}
\end{claim}
\begin{proof}[Proof of Claim]
To understand
(\ref{eq: y ge y_-  iff two cases}), note that due to
(\ref{eq: closed interval to ceiling floor}),
we have $y \ge y_-$ if  and only if
\begin{equation}
\label{eq: floor y- k ge solid torus}
\lfloor y k \rfloor \ge  
-(\lfloor y_1^{\textsc{d}} k \rfloor +
\lceil y_{1+}^{\textsc{g}} k \rceil )
\;\;\;\text{ for all } k > 0.
\end{equation}
When
$\{y, y_1^{\textsc{d}}\} \cap {\mathbb{Z}} = \emptyset$,
(\ref{eq: open interval to ceiling floor})
implies that
(\ref{eq: floor y- k ge solid torus})
holds if and only if 
$y^{\textsc{g}}_{1+} > \hat{y}_-$.
On the other hand, if 
$\{y, y_1^{\textsc{d}}\} \cap {\mathbb{Z}} \neq \emptyset$,
then 
$\lfloor y k \rfloor + \lfloor y_1^{\textsc{d}} k \rfloor
\mkern-1.5mu=\mkern-1.5mu \lfloor (y + y_1^{\textsc{d}}) k \rfloor$,
and so applying
(\ref{eq: closed interval to ceiling floor})
to 
(\ref{eq: floor y- k ge solid torus}) yields
\begin{equation}
y + y_1^{\textsc{d}}
\;\ge\; \max_{k>0}-{{\textstyle{\frac{1}{k}}}}\lceil y_{1+}^{\textsc{g}} k  \rceil
\;=\; -y_{1+}^{\textsc{g}},
\end{equation}
which is equivalent to the inequality 
$y^{\textsc{g}}_{1+} \ge \hat{y}_-$,
completing the proof of 
(\ref{eq: y ge y_-  iff two cases}).
One can then obtain
(\ref{eq: y le y_+  iff two cases})
by replacing
$y$, $y_{1+}^{\textsc{g}}$, and $y_1^{\textsc{d}}$
in 
(\ref{eq: y ge y_-  iff two cases})
with 
$-y$, $-y_{1-}^{\textsc{g}}$, and $-y_1^{\textsc{d}}$,
respectively.
\end{proof}

\begin{claim}
\label{claim: Yhat solid torus, y- y+ outside longs with equality sometimes}
If $y_1^{\textsc{d}}, y_{1+}^{\textsc{g}} \mkern-3.5mu\in\mkern-1.5mu {\mathbb{Q}}$, then
\begin{align*}
-y_1^{\textsc{d}} - y_{1+}^{\textsc{g}}
\mkern2mu&\le\mkern4mu y_-
\mkern2mu\le\mkern2mu -\lfloor y_1^{\textsc{d}} + y_{1+}^{\textsc{g}} \rfloor,
\;\;\text{with}
     \\
-y_1^{\textsc{d}} - y_{1+}^{\textsc{g}}
\mkern2mu&=\mkern4mu
y_-
\mkern20mu\iff\mkern22mu
y_1^{\textsc{d}} \mkern-.8mu\in\mkern-.8mu {\mathbb{Z}}
\mkern13.5mu\text{or}\mkern14mu
y_1^{\textsc{d}} \mkern-.8mu+\mkern-.2mu y_{1+}^{\textsc{g}} 
\mkern-2.5mu\in\mkern-.8mu {\mathbb{Z}}.
\mkern73mu
\end{align*}
If $y_1^{\textsc{d}}, y_{1-}^{\textsc{g}} \mkern-3.5mu\in\mkern-1.5mu {\mathbb{Q}}$, then
\begin{align*}
\mkern24mu
-\lceil y_1^{\textsc{d}} + y_{1-}^{\textsc{g}} \rceil
\mkern2mu\le\mkern4mu y_+
\mkern2mu&\le\mkern2mu 
-y_1^{\textsc{d}} - y_{1-}^{\textsc{g}},
\;\;\text{with}
     \\
y_+ 
\mkern2mu&=\mkern2mu
-y_1^{\textsc{d}} \mkern-.8mu-\mkern-.2mu y_{1-}^{\textsc{g}}
\mkern20mu\iff\mkern22mu
y_1^{\textsc{d}} \mkern-.8mu\in\mkern-.8mu {\mathbb{Z}}
\mkern13.5mu\text{or}\mkern14mu
y_1^{\textsc{d}} \mkern-.8mu+\mkern-.2mu y_{1-}^{\textsc{g}} 
\mkern-2.5mu\in\mkern-.8mu {\mathbb{Z}}.
\end{align*}
\end{claim}
\begin{proof}[Proof of Claim]
Just as in the proof of
Claim~\ref{claim: y- and y+ outside long+ and long-},
we define
\begin{equation}
y'_-(k) :=
\mfrac{1}{k}\mkern-1.5mu\left([y_1^{\textsc{d}} k]
- [-y_{1\mkern-.3mu+}^{\textsc{g}} k]\mkern-.5mu\right)
\end{equation}
for each $k \in {\mathbb{Z}}_{>0}$, so that
\begin{equation}
\label{eq: y- in terms of y'-(k)}
y_- 
\mkern1mu=\mkern1mu
-y_1^{\textsc{d}} \mkern-.8mu-\mkern-.2mu y_{1\mkern-.3mu+}^{\textsc{g}}\!
\mkern1.5mu+\mkern4mu \max_{k>0} \mkern1.5mu y'_-(k).
\end{equation}
Demanding $\mkern1.5muy_1^{\textsc{d}}, y_{1+}^{\textsc{g}} \mkern-2mu\in {\mathbb{Q}}$, we again set
$s_+ := \min\left\{k \mkern-1mu\in\mkern-1mu {\mathbb{Z}}_{>0} 
\mkern1mu|\mkern3mu 
y_1^{\textsc{d}}k, y_{1\mkern-.5mu +}^{\textsc{g}}\mkern-.5mu k \in {\mathbb{Z}}\right\}$.

Since $y'_-(s_{\mkern-1mu+}) \mkern-2mu=\mkern-2mu 0$, we already know that
$y_- \mkern-2mu\ge\mkern-2mu -y_1^{\textsc{d}} - y_{1+}^{\textsc{g}}$
for all $y_1^{\textsc{d}}, y_{1+}^{\textsc{g}} \mkern-2mu\in {\mathbb{Q}}$.
If $y_i^{\textsc{d}} \in {\mathbb{Z}}$,
then $y'_-(k) \mkern-2mu\le\mkern-2mu 0$ for all $k \in {\mathbb{Z}}_{>0}$,
implying $y_- \mkern-2mu\ge\mkern-2mu -y_1^{\textsc{d}} - y_{1+}^{\textsc{g}}$.
Suppose that 
$y_- \mkern-2mu= -y_1^{\textsc{d}} - y_{1+}^{\textsc{g}}$
for some $y_i^{\textsc{d}} \in {\mathbb{Q}} \setminus {\mathbb{Z}}$. 
Then since $y'_-(k) \le 0$ for all $k \in {\mathbb{Z}}_{>0}$, we have
\begin{align}
\label{eq: y(1) and y(s-1) both 0}
0
&\mkern1.5mu\ge\mkern1.5mu
y'_-\mkern-1.2mu(1) \mkern1.5mu+\mkern1.5mu (s_{\mkern-1mu+}\mkern-3mu-\mkern-1.5mu1)
\mkern.3mu y'_-\mkern-1.2mu(s_{\mkern-1mu+}\mkern-3mu-\mkern-1.5mu1)
      \\
&\mkern1.5mu=\mkern1.5mu
([y_1^{\textsc{d}}] + [-y_1^{\textsc{d}}])
-([-y_{1+}^{\textsc{g}}] + [y_{1+}^{\textsc{g}}])
   \nonumber   \\  \nonumber
&\mkern1.5mu\ge\mkern1.5mu
0.
\end{align}
Thus, the top line of 
(\ref{eq: y(1) and y(s-1) both 0}) must be an equality,
which, since $y'_-(k) \le 0$ for all $k \in {\mathbb{Z}}_{>0}$, implies
$y'_-\mkern-1.2mu(1) =
\mkern.3mu y'_-\mkern-1.2mu(s_{\mkern-1mu+}\mkern-3mu-\mkern-1.5mu1) = 0$.
In particular, the fact that $y'_-\mkern-1.2mu(1) = 0$
implies that $\mkern1.5muy_1^{\textsc{d}} + y_{1+}^{\textsc{g}} \mkern-2mu\in {\mathbb{Z}}$.
Conversely, if $\mkern1.5muy_1^{\textsc{d}} + y_{1+}^{\textsc{g}} \mkern-2mu\in {\mathbb{Z}}$,
then $y'_-(k) \equiv 0$ for all $k \in {\mathbb{Z}}_{>0}$, implying
$y_- \mkern-2mu= -y_1^{\textsc{d}} - y_{1+}^{\textsc{g}}$.

We have shown that
$y_- \mkern-2.5mu\ge\mkern-1mu -y_1^{\textsc{d}} - y_{1+}^{\textsc{g}}$,
with equality if and only if
$y_1^{\textsc{d}} \mkern-.8mu\in\mkern-.8mu {\mathbb{Z}}\mkern.8mu$
or
$\mkern.8mu y_1^{\textsc{d}} \mkern-.8mu+\mkern-.2mu y_{1+}^{\textsc{g}} 
\mkern-2.5mu\in\mkern-.8mu {\mathbb{Z}}$,
and so it remains to show that 
$y_- \le 
-\lfloor y_1^{\textsc{d}} + y_{1+}^{\textsc{g}} \rfloor$.
Note that for each $k\in {\mathbb{Z}}_{>0}$, we have
\begin{align}
y'_-(k)
\mkern1.5mu&=\mkern1.5mu
\begin{cases}
\frac{1}{k}[\mkern.5mu y_1^{\textsc{d}}k 
\mkern2mu+\mkern2mu y_{1\mkern-1mu +}^{\textsc{g}}\mkern-.9mu k]
&\;
y_{1\mkern-1mu +}^{\textsc{g}}\mkern-.5mu k \in {\mathbb{Z}}
     \\
\frac{1}{k}([y_1^{\textsc{d}}k] 
\mkern2mu+\mkern2mu [y_{1\mkern-1mu+}^{\textsc{g}}\mkern-.9muk] - 1)
&\;
y_{1\mkern-1mu +}^{\textsc{g}}\mkern-.5mu k \notin {\mathbb{Z}}
\end{cases}
      \\
\mkern1.5mu&\le\mkern1.5mu
\mfrac{1}{k}[(y_1^{\textsc{d}} \mkern2mu+\mkern2mu y_{1\mkern-1mu +}^{\textsc{g}}) k]
     \nonumber \\ \nonumber
\mkern1.5mu&\le\mkern1.5mu
[\mkern.3mu y_1^{\textsc{d}} \mkern2mu+\mkern2mu y_{1\mkern-1mu +}^{\textsc{g}}].
\end{align}
Thus, by (\ref{eq: y- in terms of y'-(k)}), we have
\begin{equation}
y_- \le 
-y_1^{\textsc{d}} \mkern-.5mu-\mkern.5mu y_{1\mkern-1mu +}^{\textsc{g}}
\mkern-.5mu+\mkern2.5mu
[\mkern.3mu y_1^{\textsc{d}} \mkern-.5mu+\mkern.5mu y_{1\mkern-1mu +}^{\textsc{g}}]
=
-\lfloor y_1^{\textsc{d}} \mkern-.5mu+\mkern.5mu y_{1\mkern-1mu +}^{\textsc{g}} \rfloor.
\end{equation}

A similar argument proves all of the analogous results for $y_+$.
\end{proof}

For the proof that
Theorem~\ref{thm: l space interval for graph manifolds}
and 
Proposition~\ref{prop: characterization of Floer simple graph manifolds}
hold for $Y$, we divide our argument into three main cases, 
first according to whether or not
$Y_1$ is Floer simple,
and then according to whether $\infty \in \mathcal{L}(Y)$.

\begin{prop}
Suppose $\mathcal{L}(Y_1) \mkern-2mu\neq\mkern-2mu \emptyset$ with 
$Y_1\mkern-1.5mu$ not Floer simple.
Then $\mathcal{L}(Y) \mkern-2mu\neq\mkern-2mu \emptyset$ if and only if
condition 
$(\textsc{nfs1}\mkern-.5mu)$ or $(\textsc{nfs4}\mkern-.5mu)$ from
Proposition~\ref{prop: characterization of Floer simple graph manifolds} holds,
in which case $\mathcal{L}(Y) \mkern-2mu=\mkern-2mu \{y_-\} \mkern-2mu=\mkern-2mu \{y_+\}$.
\end{prop}
\begin{proof}
For brevity, set $y_1^{\textsc{g}} \mkern-1.5mu:=\mkern-1mu y_{1-}^{\textsc{g}}
\mkern-4mu=\mkern-1mu y_{1+}^{\textsc{g}}$,
so that $\varphi_{1*}^{{\mathbb{P}}}(\mathcal{L}(Y_1)) = \{y_1^{\textsc{g}}\}$.
We first note 
that if $\hat{Y}[y]$ has incompressible boundary,
in which case $\hat{Y}[y]$ is a non-solid-torus graph manifold,
then 
Proposition~\ref{prop: gluing prop for LNTF equivalent manifolds}
implies that the union 
$Y(y) = \hat{Y}[y] \cup_{\varphi_1} \mkern-4.5mu Y_1$
is not an L-space.
$\hat{Y}[y]$ has compressible boundary
if and only if
$y_1^{\textsc{d}} \mkern-1.5mu\in\mkern-1.5mu {\mathbb{Z}}$ or 
$y \mkern-1.5mu\in\mkern-1.5mu {\mathbb{Z}} \mkern-1.5mu\cup\mkern-2.5mu \{\infty\}$.
Thus, for any 
$y \mkern-1.5mu\in\mkern-1.5mu {\mathbb{Q}} \mkern-1.5mu\cup\mkern-2.5mu \{\infty\}$,
Proposition~\ref{prop: l-space gluing for compressible boundary}
implies
\begin{equation}
\label{eq: L-space rule for Y_1 not Floer simple}
y \in \mathcal{L}(Y)
\mkern33mu\iff\mkern35mu
\hat{y}_-\mkern-2mu = \hat{y}_+\mkern-3mu = y_1^{\textsc{g}}
\mkern20mu\text{and}\mkern20mu
y_1^{\textsc{d}} \mkern-1.2mu\in\mkern-1.2mu {\mathbb{Z}}
\mkern9mu\text{or}\mkern8.5mu
y \mkern-1.2mu\in\mkern-1.2mu {\mathbb{Z}} \mkern-1.5mu\cup\mkern-2.5mu \{\infty\}.
\end{equation}

Condition $(\textsc{nfs4}\mkern-.5mu)$
holds if and only if $y_1^{\textsc{g}} \mkern-2mu=\mkern-2mu \infty$.
Since
$\hat{y}_- \mkern-2mu=\mkern-2mu \hat{y}_+ 
\mkern-2mu=\mkern-2mu \infty$
if and only if $y \mkern-2mu=\mkern-2mu \infty$,
we then have $\mathcal{L}(Y) \mkern-2mu=\mkern-2mu \{\infty\} 
\mkern-2mu=\mkern-2mu \{y_-\} \mkern-2mu=\mkern-2mu \{y_-\}$ when 
$y_1^{\textsc{g}} \mkern-2mu=\mkern-2mu \infty$.

We henceforth demand $y_1^{\textsc{g}} \mkern-3mu\in\mkern-1mu {\mathbb{Q}}$.
If $y_1^{\textsc{d}} \mkern-2mu\in\mkern-2mu {\mathbb{Z}}$,
then $(\textsc{nfs1}\mkern-.5mu)$ holds, and we have
\begin{equation}
\hat{y}_- \mkern-1mu=\mkern1.5mu \hat{y}_+ \mkern-1mu=\mkern1.5mu -y_1^{\textsc{d}}-y,
\;\;\;\;\;\;\;\;\;
y_- \mkern-1mu=\mkern1mu y_+ \mkern-1mu=\mkern1mu -y_1^{\textsc{d}} -y_1^{\textsc{g}}.
\end{equation}
Statement (\ref{eq: L-space rule for Y_1 not Floer simple}) then tells us that
$y \mkern-2mu\in\mkern-2mu \mathcal{L}(Y)$ if and only if
$-y_1^{\textsc{d}}-y \mkern-2mu=\mkern-2mu y_1^{\textsc{g}}$,
which occurs if and only if $y = -y_1^{\textsc{d}} -y_1^{\textsc{g}} = y_- = y_+$,
which means that
$\mathcal{L}(Y) = \{y_-\} = \{y_+\}$.

Lastly, suppose that 
$y_1^{\textsc{d}} \mkern-2mu\in\mkern-.5mu 
{\mathbb{Q}} \mkern-1mu\setminus\mkern-1mu {\mathbb{Z}}$,
with $y_1^{\textsc{g}} \mkern-3mu\in\mkern-1mu {\mathbb{Q}}$.
Since
$\hat{y}_- \mkern-3mu=\mkern-1mu \hat{y}_+ \mkern-3mu=\infty
\mkern-2mu\neq\mkern-1mu y_1^{\textsc{g}}$
when $y \mkern-2mu=\mkern-2mu \infty$,
(\ref{eq: L-space rule for Y_1 not Floer simple}) tells us that
$\infty \mkern-2mu\notin\mkern-2mu \mathcal{L}(Y)$.
Thus $y \mkern-2mu\in\mkern-2mu \mathcal{L}(Y)$ 
if and only if $y \mkern-2mu\in\mkern-2mu {\mathbb{Z}}$ and 
$\hat{y}_-\mkern-2mu = \hat{y}_+\mkern-3mu = y_1^{\textsc{g}}$.
Since
$\hat{y}_- \mkern-3mu=\mkern-1mu \hat{y}_+ \mkern-3.5mu=\mkern-2mu 
-y_1^{\textsc{d}} \mkern-3mu-\mkern-1.2mu y$
when
$y \mkern-2mu\in\mkern-2mu {\mathbb{Z}}$, this means
$y \mkern-1.5mu\in\mkern-1.5mu \mathcal{L}(Y)$
if and only if
$-y_1^{\textsc{d}} \mkern-3mu-\mkern-1.2mu y_1^{\textsc{g}} \mkern-2mu=\mkern-2mu y 
\mkern-1.5mu\in\mkern-1.5mu{\mathbb{Z}}$.
That is,
\begin{equation}
\mathcal{L}(Y) = 
\begin{cases}
\!-y_1^{\textsc{d}} \mkern-1.5mu-\mkern-.5mu y_1^{\textsc{g}}\mkern-1mu
&\;
-y_1^{\textsc{d}} \mkern-1.5mu-\mkern-.5mu y_1^{\textsc{g}}\mkern-1mu \in {\mathbb{Z}}
     \\
\mkern1mu\emptyset
&\;
-y_1^{\textsc{d}} \mkern-.8mu-y_1^{\textsc{g}}\mkern-1mu \notin {\mathbb{Z}}.
\end{cases}
\end{equation}
Thus, if
$y_1^{\textsc{d}} \mkern-2mu\in\mkern-.5mu 
{\mathbb{Q}} \mkern-1mu\setminus\mkern-1mu {\mathbb{Z}}$ and
$y_1^{\textsc{g}} \mkern-3mu\in\mkern-1mu {\mathbb{Q}}$,
then $\mathcal{L}(Y)$ is nonempty if and only if
$(\textsc{nfs1}\mkern-.5mu)$ holds,
in which case, since
$y_1^{\textsc{d}} \mkern-1.5mu+\mkern-.5mu y_1^{\textsc{g}}\mkern-1mu \in {\mathbb{Z}}$,
Claim~\ref{claim: Yhat solid torus, y- y+ outside longs with equality sometimes}
implies
$y_- \mkern-2mu=\mkern-2mu y_+ \mkern-2mu=\mkern-2mu 
-y_1^{\textsc{d}} \mkern-1.5mu-\mkern-.5mu y_1^{\textsc{g}}$,
so that
$\mathcal{L}(Y) \mkern-1.5mu=\mkern-1.5mu \{y_-\} \mkern-1.5mu=\mkern-1.5mu \{y_+\}$.

\end{proof}

We henceforth demand that $Y_1$ be Floer simple, in which case
Propositions~\ref{prop: gluing prop for LNTF equivalent manifolds}
and
\ref{prop: l-space gluing for compressible boundary}
tell us that for any $y\in {\mathbb{Q}} \cup \{\infty\}$,
the union
$Y(y) = \hat{Y}[y] \cup_{\varphi_1} \mkern-4mu Y_1$
is an L-space if and only if
\begin{equation}
\label{eq: L-space criterion in hatY[y] sometimes solid torus case}
\begin{cases}
[[\hat{y}_-, \hat{y}_+]]^{\circ}
\cup
[[y_{1-}^{\textsc{g}}, y_{1+}^{\textsc{g}}]]^{\circ} = {\mathbb{Q}} \cup \{\infty\}
&
\;\;\;\;
\text{if}\;
\hat{Y}[y]\;
\text{has incompressible boundary}
       \\
\hat{y}_- = 
\hat{y}_+
\in [[y_{1-}^{\textsc{g}}, y_{1+}^{\textsc{g}}]]
&
\;\;\;\;
\text{if}\;
\hat{Y}[y]\;
\text{has compressible boundary}.
\end{cases}
\end{equation}
Note that $\hat{Y}[y]$ has compressible boundary if and only if
$\{y, y_1^{\textsc{d}}\} \mkern-1.5mu\cap\mkern-1.5mu {\mathbb{Z}} 
\mkern-1.5mu\neq\mkern-1.5mu \emptyset$ or 
$y \mkern-2mu=\mkern-2mu \infty$,
with the former condition accounting for
the case of solid torus $\hat{Y}[y]$,
and the latter condition accounting for the case in which
$\hat{Y}[\infty]$ is either a solid torus (when $y_1^{\textsc{d}} \in {\mathbb{Z}}$)
or the connected sum of a solid torus with a lens space 
(when $y_1^{\textsc{d}} \notin {\mathbb{Z}}$).

\begin{prop}
\label{prop: infty in L(Y) case for Yhat sometimes torus}
Suppose $Y_1$ is Floer simple.
Then $\infty \mkern-2mu\in\mkern-2mu \mathcal{L}(Y)$
if and only if
condition $(\textsc{fs3}\mkern-.5mu)$ from
Proposition~\ref{prop: characterization of Floer simple graph manifolds} holds,
in which case 
$\mathcal{L}(Y) \mkern-2mu=\mkern-2mu [[y_-,y_+]]$.
\end{prop}
\begin{proof}

The first part of the statement is immediate.
That is, since $\hat{Y}[\infty]$ has compressible boundary, with
$\hat{y}_-\mkern-3mu = \hat{y}_+\mkern-3mu =\mkern-2mu \infty$,
we have $\infty \in \mathcal{L}(Y)$ if and only if
$\infty \in [[y_{1-}^{\textsc{g}}, y_{1+}^{\textsc{g}}]]$,
which occurs if and only if
$(\textsc{fs3}\mkern-.5mu)$ holds.
For the remainder of the proof, we assume
$(\textsc{fs3}\mkern-.5mu)$ holds.

Consider the case in which
$y_{1-}^{\textsc{g}} \mkern-4mu\neq\mkern-1.5mu  y_{1+}^{\textsc{g}}$.
Since $\infty \mkern-2mu\in\mkern-2mu [[y_{1-}^{\textsc{g}}, y_{1+}^{\textsc{g}}]]$,
$[[y_{1-}^{\textsc{g}}, y_{1+}^{\textsc{g}}]]$
takes one of the forms
(a) $[y_{1-}^{\textsc{g}}, +\infty]$,
(b) $[-\infty, y_{1+}^{\textsc{g}}]$,
or
(c) $[y_{1-}^{\textsc{g}}, +\infty] \cup [-\infty, y_{1+}^{\textsc{g}}]$,
in which cases
$[[y_-, y_+]]$ takes the respective forms
(a) $\mkern-1mu[-\infty, y_+]$, (b) $\mkern-1mu[y_-, +\infty]$,
(c) $\mkern-1mu[y_-, +\infty] \mkern-1mu\cup\mkern-1mu [-\infty, y_+]$,
with Case~(c) due to the fact that
$\infty \mkern-2mu\neq\mkern-2mu y_{1-}^{\textsc{g}} 
\mkern-3mu>\mkern-2mu y_{1+}^{\textsc{g}} 
\mkern-3mu\neq\mkern-2mu \infty$
implies 
$\infty \mkern-2mu\neq\mkern-2mu y_- 
\mkern-2mu>\mkern-2mu y_+ \mkern-2mu\neq\mkern-2mu \infty$
by 
Claim~\ref{claim: Yhat solid torus, y- y+ outside longs with equality sometimes}.
Thus, for $y \in {\mathbb{Q}}$,
the condition $y \in [[y_-, y_+]]$
is respectively equivalent to the right-hand conditions of
(a) (\ref{eq: y ge y_-  iff two cases}),
(b) (\ref{eq: y le y_+  iff two cases}), or
(c) (\ref{eq: y ge y_-  iff two cases}) or (\ref{eq: y le y_+  iff two cases})
from 
Claim~\ref{claim: Yhat solid torus sometimes, translating y to yhat ineqs},
each of which conditions, given the respective form of 
$[[y_{1-}^{\textsc{g}}, y_{1+}^{\textsc{g}}]]$,
is equivalent to 
(\ref{eq: L-space criterion in hatY[y] sometimes solid torus case}),
which holds if and only if $Y(y)$ is and L-space,
and we conclude that $\mathcal{L}(Y) = [[y_-, y_+]]$.

Next, suppose that $(\textsc{fs3}\mkern-.5mu)$ holds
with $y_{1-}^{\textsc{g}} \mkern-3.9mu=\mkern-1mu  y_{1+}^{\textsc{g}}
\mkern-3.9mu=\mkern-.1mu:\mkern-1mu y_1^{\textsc{g}}$.
Since 
$[[y_{1-}^{\textsc{g}}, y_{1+}^{\textsc{g}}]]
= [[y_{1-}^{\textsc{g}}, y_{1+}^{\textsc{g}}]]^{\circ}
= {\mathbb{Q}} \cup \{\infty\} \setminus y_1^{\textsc{g}}$,
the L-space condition in 
(\ref{eq: L-space criterion in hatY[y] sometimes solid torus case})
takes the following form.  For any $y \mkern-1.5mu\in\mkern-1.5mu {\mathbb{Q}}$,
\begin{equation}
\label{eq: y1G-  = y_1G+ version of L-space condition}
y \in \mathcal{L}(Y)
\;\;\iff\;\;
\begin{cases}
y_1^{\textsc{g}} \in [[\hat{y}_-, \hat{y}_+]]^{\circ}
& \{y_1^{\textsc{d}}, y\} \cap {\mathbb{Z}} = \emptyset
    \\
y \neq -y_1^{\textsc{d}} - y_1^{\textsc{g}}
& \{y_1^{\textsc{d}}, y\} \cap {\mathbb{Z}} \neq \emptyset.
\end{cases}
\end{equation}
The derivation of the $\{y_1^{\textsc{d}}, y\} \cap {\mathbb{Z}} = \emptyset$ case
from (\ref{eq: L-space criterion in hatY[y] sometimes solid torus case})
is immediate.
In the $\{y_1^{\textsc{d}}, y\} \cap {\mathbb{Z}} \neq \emptyset$ case,
(\ref{eq: L-space criterion in hatY[y] sometimes solid torus case}) tells us that
$y \in \mathcal{L}(Y)$ if and only if 
$\hat{y}_- \mkern-2.5mu\neq\mkern-.5mu y_1^{\textsc{d}}$.
Since 
$\hat{y}_- \mkern-3.5mu= \hat{y}_+ 
\mkern-3.5mu= - y_1^{\textsc{d}}\mkern-1mu - y$,
this occurs if and only if $y \neq -y_1^{\textsc{d}} - y_1^{\textsc{g}}$,
completing the proof of (\ref{eq: y1G-  = y_1G+ version of L-space condition}).

Now, since  $(\textsc{fs3}\mkern-.5mu)$ demands that
$\infty \in [[y_{1-}^{\textsc{g}},  y_{1+}^{\textsc{g}}]]$,
implying $y_1^{\textsc{g}} \mkern-2.5mu\in\mkern-1mu {\mathbb{Q}}$,
Claim~\ref{claim: Yhat solid torus, y- y+ outside longs with equality sometimes}
tells us that
$y_+
\mkern.5mu\le
-y_1^{\textsc{d}} \mkern-1.5mu-\mkern-1.1mu y_1^{\textsc{g}} 
\mkern2mu\le\mkern2mu
y_-$,
with $y_- \mkern-3mu=\mkern-1.5mu y_+$ if and only if
$\mkern.5mu y_1^{\textsc{d}} \in {\mathbb{Z}}\mkern.5mu$ or 
$\mkern.5mu y_1^{\textsc{d}} \mkern-1.5mu+\mkern-.8mu y_1^{\textsc{g}} \in {\mathbb{Z}}$. 
If $\mkern.5mu y_1^{\textsc{d}} \in {\mathbb{Z}}\mkern.5mu$,
then (\ref{eq: y1G-  = y_1G+ version of L-space condition})
implies
$\mathcal{L}(Y) = 
{\mathbb{Q}} \cup \{\infty\} \setminus
\{-y_1^{\textsc{d}} - y_1^{\textsc{g}}\}
= [[y_-, y_+]]$.
Suppose instead that
$y_1^{\textsc{d}} \notin {\mathbb{Z}}$.
If 
$y_1^{\textsc{d}} \mkern-1.5mu+\mkern-.8mu y_1^{\textsc{g}} \in {\mathbb{Z}}$,
then
for each $y \in [[y_-, y_+]] \cap {\mathbb{Q}}$,
either $y \notin {\mathbb{Z}}$,
in which case,
the fact that $y \in [y_-, +\infty] \cup [-\infty, y_+]$
makes
Claim~\ref{claim: Yhat solid torus sometimes, translating y to yhat ineqs}
tell us that
$y_1^{\textsc{g}} \in [[\hat{y}_-, \hat{y}_+]]^{\circ}$,
so that
(\ref{eq: y1G-  = y_1G+ version of L-space condition})
implies $y \mkern-1.5mu\in\mkern-1.5mu \mathcal{L}(Y)$;
or $y \mkern-1.5mu\in\mkern-1.5mu {\mathbb{Z}}$, in which case
(\ref{eq: y1G-  = y_1G+ version of L-space condition})
tells us that $y \mkern-1.5mu\in\mkern-1.5mu \mathcal{L}(Y)$
if and only if 
$y \neq -y_1^{\textsc{d}} - y_1^{\textsc{g}} = y_- = y_+$.
Combining these two results makes $\mathcal{L}(Y) = [[y_-, y_+]]$.

Lastly, consider the case in which 
$y_{1-}^{\textsc{g}} \mkern-4mu=\mkern-1mu  y_{1+}^{\textsc{g}}
\mkern-4mu=:\mkern-1mu  y_1^{\textsc{g}}
\mkern-1mu\in\mkern-.5mu {\mathbb{Q}}$,
$\mkern.5mu y_1^{\textsc{d}} \mkern-.5mu\notin\mkern-.5mu {\mathbb{Z}}$, and 
$\mkern.5mu y_1^{\textsc{d}} \mkern-.9mu+ y_1^{\textsc{g}} 
\mkern-1mu\notin\mkern-1mu {\mathbb{Z}}$.
Since in this case,
$y \in {\mathbb{Z}}$ automatically implies 
$y \neq -y_1^{\textsc{d}} \mkern-.9mu- y_1^{\textsc{g}}$,
we deduce that
the second line of 
(\ref{eq: y1G-  = y_1G+ version of L-space condition}) is vacuous.
That is, $y \in \mathcal{L}(Y)$ for all $y \in {\mathbb{Z}}$.
Accordingly, since
Claim~\ref{claim: Yhat solid torus, y- y+ outside longs with equality sometimes}
tells us that
\begin{equation}
\label{eq: no integers in me}
-\lceil y_1^{\textsc{d}} \mkern-1.5mu + \mkern-1.1mu y_1^{\textsc{g}} \rceil
\mkern2mu\le\mkern3mu
y_+
\mkern1.5mu<\mkern1mu
-y_1^{\textsc{d}} \mkern-1.5mu-\mkern-1.1mu y_1^{\textsc{g}} 
\mkern2mu<\mkern3mu
y_-
\mkern1.5mu\le\mkern1mu
-\lfloor y_1^{\textsc{d}} \mkern-1.5mu + \mkern-1.1mu y_1^{\textsc{g}} \rfloor,
\end{equation}
we observe that the complement of $[[y_-, y_+]]$ contains no integers,
and so
${\mathbb{Z}} \subset [[y_-, y_+]]$.
On the other hand, for any 
$y \in {\mathbb{Q}} \setminus {\mathbb{Z}}$,
(\ref{eq: y1G-  = y_1G+ version of L-space condition}) tells us that
$y \in \mathcal{L}(Y)$
if and only if
$y_1^{\textsc{g}} \in [[\hat{y}_-, \hat{y}_+]]^{\circ}$,
which, by 
Claim~\ref{claim: Yhat solid torus sometimes, translating y to yhat ineqs},
occurs if and only if
$y \in [y_-, +\infty] \cup [-\infty, y_+] = [[y_-, y_+]]$.
Thus, once again, $\mathcal{L}(Y) = [[y_-, y_+]]$,
completing the proof of the proposition.
\end{proof}

\begin{prop}
Suppose $Y_1$ is Floer simple and $\infty \mkern-1mu\notin\mkern-1mu \mathcal{L}(Y)$.
Then $\mathcal{L}(Y) \mkern-1.8mu\neq\mkern-1mu \emptyset$ if and only if
either
condition 
$(\textsc{nfs2}\mkern-.5mu)$ from
Proposition~\ref{prop: characterization of Floer simple graph manifolds} holds,
in which case $\mathcal{L}(Y) \mkern-2mu=\mkern-2mu \{y_-\} \mkern-2mu=\mkern-2mu \{y_+\}$,
or condition
$(\textsc{fs1}\mkern-.5mu)$ or
$(\textsc{fs2}\mkern-.5mu)$ holds,
in which case
$\mathcal{L}(Y) \mkern-2mu=\mkern-2mu [[y_-,y_+]]$.
\end{prop}
\begin{proof}
Since $Y_1$ is Floer simple, and since
$\infty \mkern-1mu\notin\mkern-1mu \mathcal{L}(Y)$
implies that
$(\textsc{fs3}\mkern-.5mu)$ fails to hold, 
we know that
$\infty \notin [[y_{1-}^{\textsc{g}}, y_{1+}^{\textsc{g}}]]$.
Thus, 
$[[y_{1-}^{\textsc{g}}, y_{1+}^{\textsc{g}}]]
= \left<-\infty, +\infty\right>$
or
$[[y_{1-}^{\textsc{g}}, y_{1+}^{\textsc{g}}]]
=[y_{1-}^{\textsc{g}}, y_{1+}^{\textsc{g}}]$.

Suppose that 
$[[y_{1-}^{\textsc{g}}, y_{1+}^{\textsc{g}}]]= 
\left<-\infty, +\infty\right>$,
which occurs if and only if
condition $(\textsc{fs1}\mkern-.5mu)$ holds.
For any $y \in {\mathbb{Q}}$, we have
$\infty \neq \hat{y}_- \ge  \hat{y}_+ \neq \infty$,
implying condition 
(\ref{eq: L-space criterion in hatY[y] sometimes solid torus case}) holds,
making $Y\mkern-1.5mu(y)$ an L-space.
Thus, since 
$y_{1-}^{\textsc{g}} = y_{1+}^{\textsc{g}} = \infty$
implies
$y_- = y_+ = \infty$,
we have
$\mathcal{L}(Y) = {\mathbb{Q}} = 
\left< -\infty, +\infty \right>
=[[y_-, y_+]]$.

Lastly, suppose
$[[y_{1-}^{\textsc{g}}, y_{1+}^{\textsc{g}}]]
=[y_{1-}^{\textsc{g}}, y_{1+}^{\textsc{g}}]$.
For any $y \in {\mathbb{Q}}$,
we have $y \in [y_-, +\infty] \cap [-\infty, y_+]$
if and only if the right-hand conditions of
(\ref{eq: y ge y_-  iff two cases})
and (\ref{eq: y le y_+  iff two cases}) from
Claim~\ref{claim: Yhat solid torus sometimes, translating y to yhat ineqs}
hold, which, since
$[[y_{1-}^{\textsc{g}}, y_{1+}^{\textsc{g}}]] = 
[y_{1-}^{\textsc{g}}, +\infty] \cap [y_{1-}^{\textsc{g}}, y_{1+}^{\textsc{g}}] \neq \emptyset$,
occurs if and only if
(\ref{eq: L-space criterion in hatY[y] sometimes solid torus case})
holds, which happens precisely when $Y(y)$ is an L-space.
Thus, $\mathcal{L}(Y) = [y_-, +\infty] \cap [-\infty, y_+]$, or in other words,
\begin{equation}
\mathcal{L}(Y)=
\label{eq: three cases for L-space interval in Yhat torus case}
\begin{cases}
[[y_-, y_+]]
& y_- < y_+
   \\
\{y_-\} = \{y_+\}
& y_- = y_+
   \\
\emptyset
& y_- > y_+
\end{cases}.
\end{equation}
That is, 
$\mathcal{L}(Y)$ is nonempty
if and only if either
$(\textsc{fs2}\mkern-.5mu)$ holds,
in which case $\mathcal{L}(Y) = [[y_-, y_+]]$,
or $(\textsc{nfs2}\mkern-.5mu)$ holds,
in which case
$\mathcal{L}(Y) \!=\! \{y_-\} \!=\! \{y_+\}$.
\end{proof}

\smallskip
The combined results of Sections
\ref{ss: non-orientable base case},
\ref{ss: orientable base, no solid tori}, and
\ref{ss: orientable base, cases involving solid torus Y hat}
prove that
Theorem~\ref{thm: l space interval for graph manifolds}
and Proposition~\ref{prop: characterization of Floer simple graph manifolds}
hold for
any graph manifold $Y$ with torus boundary,
$b_1 = 1$, tree height $k>0$,
and $n_{{\textsc{g}}}>0$ daughter subtrees,
given the inductive assumptions,
laid out in Section
\ref{ss: inductive set-up for main result},
that
Theorem~\ref{thm: l space interval for graph manifolds}
and Proposition~\ref{prop: characterization of Floer simple graph manifolds}
hold for
any graph manifold with torus boundary,
$b_1 = 1$, and either
tree height $k$ and $\le n_{{\textsc{g}}}-1$
daughter subtrees, or tree height $\le k-1$.

For graph manifolds $Y$ with torus boundary, $b_1 = 1$,
and tree height $k$, inducting on the
number of daughter subtrees $n_{{\textsc{g}}}$ yields the
result that any graph manifold with torus boundary,
$b_1 = 1$, and tree height $k$ satisfies
Theorem~\ref{thm: l space interval for graph manifolds}
and Proposition~\ref{prop: characterization of Floer simple graph manifolds}.
Inducting on tree height $k$ then completes the proof of
Theorem~\ref{thm: l space interval for graph manifolds}
and Proposition~\ref{prop: characterization of Floer simple graph manifolds}.

\qed

\subsection{Some technical results for ${\boldsymbol{y_-}}$ and ${\boldsymbol{y_+}}$}

We conclude this section with the proof of some
basic facts about $y_-$ and $y_+$ for later use.

Recall that $y_-$ and $y_+$ are defined by
$\mkern.6muy_- \mkern-2.5mu:= \max_{k>0} y_-(k)\mkern.6mu$ and 
$\mkern.6muy_+ \mkern-2.5mu:= \min_{k>0} y_+(k)$, where
\begin{align}
\label{eq: technical section, defs of y-(k) etc}
y_-(k)
&:=
-\mfrac{1}{k}\mkern-2.5mu\left( \mkern2.5mu 1
\mkern1.5mu+\mkern.1mu
{{\textstyle{\sum\limits_{i=1}^{n_{{\textsc{d}}}}  
\mkern-2.5mu\left\lfloor y^{\textsc{d}}_i k \right\rfloor}}}
\mkern1.5mu+\mkern1mu
{{\textstyle{\sum\limits_{i=1}^{n_{{\textsc{g}}}}}}}\mkern-3.5mu
\left(  \left\lceil y^{\textsc{g}}_{i+}k \right\rceil \mkern-2.2mu-\mkern-1.3mu 1 
\right)\mkern-2mu\right),
\\ \nonumber
y_+(k)
&:=
-\mfrac{1}{k}\mkern-2.5mu\left( \mkern-1.5mu -1
\mkern1.5mu+\mkern.1mu
{{\textstyle{\sum\limits_{i=1}^{n_{{\textsc{d}}}}  
\mkern-2.5mu\left\lceil y^{\textsc{d}}_i k \right\rceil}}}
\mkern1.5mu+\mkern1mu
{{\textstyle{\sum\limits_{i=1}^{n_{{\textsc{g}}}}}}}\mkern-3.5mu
\left(  \left\lfloor y^{\textsc{g}}_{i-}k \right\rfloor \mkern-2.2mu+\mkern-1.3mu 1 
\right)\mkern-2mu\right).
\end{align}
Let $k_-, k_+ \in {\mathbb{Z}}_{>0}$ denote the lowest values of 
$k$ for which these extrema occur.  That is, set
\begin{equation}
\label{eq: def for k- in terms of value at which maximum first occurs}
k_-:= \min \{k \in {\mathbb{Z}}_{>0}\mkern.5mu|\mkern2mu y_-(k) = y_-\},
\;\;\;\;\;
k_+:= \min \{k \in {\mathbb{Z}}_{>0}\mkern.5mu|\mkern2mu y_+(k) = y_+\}.
\end{equation}
We then have the following result.

\begin{prop}
\label{prop: k- and k+ are denominators of y- and y+}
If $\mkern.8muY\mkern-4mu$ is not a solid torus,
and $\mkern.5muy_-, y_+ \mkern-3.5mu\in\mkern-1.5mu {\mathbb{Q}}$, then
$k_-\!$ and $k_+$ are the respective denominators of $y_-\!$ and 
$\mkern.8mu y_+\mkern-1mu$.
That is,
\begin{equation}
\label{eq: k- is a denominator}
k_-:= \min \{k \in {\mathbb{Z}}_{>0}\mkern1mu|\mkern3mu y_-k \in {\mathbb{Z}}\},
\;\;\;\;\;
k_+:= \min \{k \in {\mathbb{Z}}_{>0}\mkern1mu|\mkern3mu y_+k \in {\mathbb{Z}}\}.
\end{equation}
\end{prop}
\begin{proof}
Since this question is unaffected by an 
overall translation of $y_-$ by an integer,
we assume without loss of generality that 
$y_i^{\textsc{d}} \in \left<0, 1\right>$ for all $i \in \{1, \ldots, n_{{\textsc{d}}}\}$
and that
$y_{i+}^{\textsc{g}} \in \left[0, 1\right>$ for all $i \in \{1, \ldots, n_{{\textsc{g}}}\}$.
In addition, we permute the daughter subtrees $Y_1, \ldots, Y_{n_{{\textsc{g}}}}$
so that 
$y_{i+}^{\textsc{g}} \in \left<0, 1\right>$ for all $i \in \{1, \ldots, \bar{n}_{{\textsc{g}}}\}$,
and 
$y_{i+}^{\textsc{g}} = 0$ for all $i \in \{\bar{n}_{{\textsc{g}}}+1,\ldots, n_{{\textsc{g}}}\}$,
for some $\bar{n}_{{\textsc{g}}} \le n_{{\textsc{g}}}$.

Setting $N_{{\textsc{g}}} := n_{{\textsc{g}}} - \bar{n}_{{\textsc{g}}}$, we note that if $N_{{\textsc{g}}} > 0$, then we obtain
\begin{align}
\label{eq: def of y-(k) for k- = 1 proof}
y_-(k)
&:=
\mfrac{1}{k}\mkern-2.5mu\left( N_{{\textsc{g}}}\mkern-2.5mu-\mkern-1.5mu 1
\mkern1.5mu-\mkern.1mu
{{\textstyle{\sum\limits_{i=1}^{n_{{\textsc{d}}}}  
\mkern-2.5mu\left\lfloor y^{\textsc{d}}_i k \right\rfloor}}}
\mkern1.5mu-\mkern1mu
{{\textstyle{\sum\limits_{i=1}^{\bar{n}_{{\textsc{g}}}}}}}\mkern-3.5mu
\left(  \left\lceil y^{\textsc{g}}_{1+}k \right\rceil \mkern-2.2mu-\mkern-1.3mu 1 
\right)\mkern-2mu\right),
    \\ \nonumber
&\le 
\mfrac{1}{k}( N_{{\textsc{g}}}\mkern-2.5mu-\mkern-1.5mu 1)
\;
\le N_{{\textsc{g}}}\mkern-2.5mu-\mkern-1.5mu 1 \; =\; y_-(1)
\end{align}
for all $k >0$, making $y_- = y_-(1)$.
On the other hand, if $N_{{\textsc{g}}} = 0$, then
the top line of 
(\ref{eq: def of y-(k) for k- = 1 proof}) implies $y_-(k) \le -\frac{1}{k} < 0$ for all $k >0$,
which, since $y_- \in {\mathbb{Z}}$, implies $y_- \le -1 = y_-(1)$.
Thus, in either case, we have $k_- = -1$, and a similar argument
shows $k_+ = 1$ when $y_+ \in {\mathbb{Z}}$.

If $n_{{\textsc{g}}} \mkern-4mu=\mkern-1.5mu 0$ and 
$n_{{\textsc{d}}} \mkern-3.5mu\le\mkern-1.5mu 1$, then $Y$ is a solid torus,
a case excluded by hypothesis.
Suppose $n_{{\textsc{d}}} \mkern-4mu=\mkern-2mu 0$ and $n_{{\textsc{g}}} \mkern-4mu=\mkern-2mu 1$.
If $y_{1+}^{\textsc{g}} \mkern-4mu\in\mkern-1.5mu {\mathbb{Z}}$, then the above argument shows
$y_- \mkern-4mu\in\mkern-1.5mu {\mathbb{Z}}$ and $k_-\mkern-4mu =\mkern-1.5mu 1$.
If 
$y_{1+}^{\textsc{g}} \mkern-4mu =\mkern-.1mu:
{r_{1\mkern-.7mu+}^{\textsc{g}}\mkern-.5mu}/{s_{1\mkern-.7mu+}^{\textsc{g}}} 
\mkern-1.5mu\in \left<0, 1\right>\mkern.8mu$
with $\mkern.8mu{r_{1\mkern-.7mu+}^{\textsc{g}}\mkern-.5mu},
{s_{1\mkern-.7mu+}^{\textsc{g}}} \in {\mathbb{Z}}_{>0}\mkern.8mu$ 
relatively prime, then for all $k > 0$, we have
\begin{equation}
y_-(k) = -y_{1\mkern-.7mu+}^{\textsc{g}} 
\mkern-4.5mu-\mkern-1.5mu \mfrac{1}{k}[-y_{1+}^{\textsc{g}}]
\le  -y_{1\mkern-.7mu+}^{\textsc{g}} = y_-(s_{1+}^{\textsc{g}}).
\end{equation}
Thus $y_- \mkern-3.5mu=\mkern-1mu -y_{1\mkern-.7mu+}^{\textsc{g}} 
\mkern-3.5mu=\mkern-1mu y_-(s_{1+}^{\textsc{g}})$, and since
$y_-(k) \mkern-1mu<\mkern-1mu y_-$ for all
$k \mkern-1mu\in\mkern-1mu 
\{1, \ldots, s_{1+}^{\textsc{g}}\mkern-4.5mu-\mkern-2mu 1\}$, we also have
$k_- \mkern-3mu=\mkern-1.5mu s_{1+}^{\textsc{g}}$, which is the denominator of $y_-$.
A similar argument shows that
(\ref{eq: k- is a denominator}) also holds for $k_+$ when 
$n_{{\textsc{d}}} \mkern-4mu=\mkern-2mu 0$ and $n_{{\textsc{g}}} \mkern-4mu=\mkern-2mu 1$.

Finally, suppose we exclude all cases
considered in the preceding paragraph, and
all cases in which $y_- \in {\mathbb{Z}}$.
Since $y_- \in{\mathbb{Q}}$ by hypothesis,
the second paragraph implies we also have
$y_{i+}^{\textsc{g}} \in {\mathbb{Q}} \setminus {\mathbb{Z}}$
for all $i \in \{1, \ldots, n_{{\textsc{g}}}\}$.
Since the problem is still unaffected by an overall integer
translation of $y_-$,
we demand without loss of generality that
$y_i^{\textsc{d}}, y_{j+}^{\textsc{g}} \in \left<0, 1\right>$
for all $i \in \{1, \ldots, n_{{\textsc{d}}}\}$ and $j \in \{1, \ldots, n_{{\textsc{g}}}\}$.
Thus, after removing one more regular fiber neighborhood $\nu(f_0)$
from the JSJ component containing $\partial Y$,
and Dehn filling this complement with slope $y_0^{\textsc{d}} := -1$,
we may appeal to Theorem 3 from Jankins and Neumann \cite{JankinsNeumann},
which is equivalent to the following statement.

If $y_- := \max_{k>0}y_-(k)$, where
\begin{equation}
y_-(k)
:=
1-\mfrac{1}{k}\mkern-2.5mu\left( \mkern2.5mu 1
\mkern1.5mu+\mkern.1mu
{{\textstyle{\sum\limits_{i=1}^{n_{{\textsc{d}}}}  
\mkern-2.5mu\left\lfloor y^{\textsc{d}}_i k \right\rfloor}}}
\mkern1.5mu+\mkern1mu
{{\textstyle{\sum\limits_{i=1}^{n_{{\textsc{g}}}}}}}\mkern-3.5mu
\left(  \left\lceil y^{\textsc{g}}_{i+}k \right\rceil \mkern-2.2mu-\mkern-1.3mu 1 
\right)\mkern-2mu\right)
\end{equation}
(with the initial $1$ coming from $-y_0^{\textsc{d}}$), and if
$k_-$ is defined as in 
(\ref{eq: def for k- in terms of value at which maximum first occurs}),
then there is a positive integer $c < k_-$ with $\gcd(c, k_-) = 1$,
a permutation $\pi$ on 
$n_{{\textsc{d}}} \mkern-3mu+\mkern-.5mu n_{{\textsc{g}}} \mkern-3mu+\mkern-1.9mu 1$ elements,
and an $n_{{\textsc{d}}} \mkern-3mu+\mkern-.5mu n_{{\textsc{g}}} \mkern-3mu+\mkern-1.9mu 1$-tuple 
$a_* = (c,\mkern1.5mu k_- \mkern-3mu-\mkern-1.5mu c,\mkern1.5mu 1,\mkern1.5mu \ldots, 1)$,
such that
\begin{equation}
\lfloor y_i^{\textsc{d}} k_- \rfloor + 1 
=  a_{\pi(i)},\;\;\;\;\;
\lceil y_{j+}^{\textsc{g}} k_- \rceil 
=  a_{\pi(n_{{\textsc{d}}} + j)},\;\;\;\;\;
y_- 
= \frac{a_{\pi(n_{{\textsc{d}}} + n_{{\textsc{g}}}+ 1)}}{k_-}
\end{equation}
for all $i\in\{1, \ldots, n_{{\textsc{d}}}\}$ and $j\in\{1, \ldots, n_{{\textsc{g}}}\}$.
In particular, $\gcd(a_{\pi(n_{{\textsc{d}}} + n_{{\textsc{g}}}+ 1)}, k_-) = 1$,
making $k_-$ satisfy 
(\ref{eq: k- is a denominator}).
A similar argument shows that 
(\ref{eq: k- is a denominator})
holds for $k_+$, completing the proof.
\end{proof}

The above result is useful for proving the 
following proposition, but first, we define
\begin{equation}
N_{{\textsc{g}}}
\mkern-2mu:\mkern-.1mu=\mkern-1mu 
\left| \{i : y_{i+}^{\textsc{g}} \mkern-3mu\in\mkern-1mu {\mathbb{Z}} \} \right|
+
\left| \{i : y_{i-}^{\textsc{g}} \mkern-3mu\in\mkern-1mu {\mathbb{Z}} \} \right|,
\;\;\;\;\;
\bar{n}_{{\textsc{d}}}:=
\left| \{i : y_i^{\textsc{d}} \mkern-3mu\in\mkern-1mu {\mathbb{Q}} \setminus {\mathbb{Z}} \} \right|.
\end{equation}

\begin{prop}
\label{prop: y- = y+ means integer}
Suppose that $\mkern1mu\bar{n}_{{\textsc{d}}} + N_{{\textsc{g}}} > 0$, 
with $Y\!$ not a solid torus.
If $y_- \mkern-2mu= y_+ \in {\mathbb{Q}}$, then
$\mkern1muy_- \mkern-2mu= y_-(1) = y_+(1) = y_+ \in \mkern1mu{\mathbb{Z}}$.
\end{prop}

\begin{proof}

Since $y_- \mkern-3mu= y_+$, we have
\begin{equation}
0 
\mkern5mu=\mkern5mu y_ -\mkern-3mu- y_+
\mkern1mu=\mkern5mu \max_{k_1>0} y_-(k_1) - \min_{k_2>0} y_+(k_2)
\mkern5mu=\mkern1mu \max_{k_1, k_2 >0}\! 
\left(y_-(k_1)\mkern-.5mu - y_+(k_2)\right),
\end{equation}
with $y_-(k)$ and $y_+(k)$ as defined in
(\ref{eq: technical section, defs of y-(k) etc}).
Defining $k_-, k_+ \in {\mathbb{Z}}_{>0}$
as in
(\ref{eq: def for k- in terms of value at which maximum first occurs}),
we observe that since $y_- \mkern-3mu= y_+$,
Proposition~\ref{prop: k- and k+ are denominators of y- and y+}
implies $k_- \mkern-3mu= k_+$.
In particular, the set of $(k_1, k_2) \in {\mathbb{Z}}_{>0}\times {\mathbb{Z}}_{>0}$
for which $y_-(k_1)\mkern-.5mu - y_+(k_2)$ is maximized
has nontrivial intersection with the set
of $(k_1, k_2) \in {\mathbb{Z}}_{>0}\times {\mathbb{Z}}_{>0}$ for which $k_1 = k_2$.
We therefore have
\begin{align}
0 
\mkern3mu
&=\mkern1mu 
\max_{k >0}\mkern.8mu
\left(y_-(k)\mkern-.2mu + y_+(k)\right)
   \nonumber    \\
&=\max_{k>0}
\mfrac{1}{k}\mkern-2.5mu\left( -2
\mkern1.5mu+\mkern.1mu
{{\textstyle{\sum\limits_{i=1}^{n_{{\textsc{d}}}}}}}
\left( \left\lceil y^{\textsc{d}}_i k \right\rceil
-\left\lfloor y^{\textsc{d}}_i k \right\rfloor
\right)
\mkern1.5mu-\mkern1mu
{{\textstyle{\sum\limits_{i=1}^{n_{{\textsc{g}}}}}}}\mkern-3.5mu
\left( \left\lceil y^{\textsc{g}}_{i+}k \right\rceil
- 1 +  \left\lceil -y^{\textsc{g}}_{i-}k \right\rceil - 1
\right)\right)
   \nonumber  \\
&=
-{{\textstyle{\sum\limits_{i=1}^{\tilde{n}_{{\textsc{g}}}}}}}\mkern-3.5mu
\left\lfloor \tilde{y}^{\textsc{g}}_{i} \right\rfloor
+
\max_{k>0}
\mfrac{1}{k}\mkern-2.5mu\left( N_{{\textsc{g}}}\mkern-2.5mu - 2
\mkern1.5mu+\mkern.1mu
{{\textstyle{\sum\limits_{i=1}^{n_{{\textsc{d}}}}}}}
\left( \left\lceil y^{\textsc{d}}_i k \right\rceil
-\left\lfloor y^{\textsc{d}}_i k \right\rfloor
\right)
\mkern1.5mu-\mkern1mu
{{\textstyle{\sum\limits_{i=1}^{\tilde{n}_{{\textsc{g}}}}}}}\mkern-3.5mu
\left( \left\lceil [\tilde{y}^{\textsc{g}}_{i}]k \right\rceil - 1
\right)\right),
\end{align}
where the second line uses the fact that
$-\lfloor q\rfloor = \lceil -q\rceil$ for all $q \in {\mathbb{Q}}$,
and where in the third line, if we set 
$\tilde{n}_{{\textsc{g}}} := 2n_{{\textsc{g}}} \mkern-2.5mu- N_{{\textsc{g}}}$, then
$\tilde{y}^{\textsc{g}}_* \in {\mathbb{Q}}^{\tilde{n}_{{\textsc{g}}}}$ is
the $\tilde{n}_{{\textsc{g}}}$-tuple obtained from deleting the $N_{{\textsc{g}}}$
integer-valued entries from the $2n_{{\textsc{g}}}$-tuple
$(y_{1+}^{\textsc{g}}, \ldots, y_{n_{{\textsc{g}}}+}^{\textsc{g}},
-y_{1-}^{\textsc{g}}, \ldots, -y_{n_{{\textsc{g}}}-}^{\textsc{g}})$.
The third line also makes use of the
notation $[\cdot]: {\mathbb{Q}} \to \left[0, 1\right>$,
$q \mapsto [q]:= q - \lfloor q \rfloor$.

Since $\mkern1mu\bar{n}_{{\textsc{d}}} + N_{{\textsc{g}}} > 0$,
we know that $\mkern1mu\bar{n}_{{\textsc{d}}} + N_{{\textsc{g}}} -1 \ge 0$.  Thus
for all $k > 0$, we have
\begin{align}
y_-(k)+\mkern.2mu y_+(k)
&<
-{{\textstyle{\sum\limits_{i=1}^{\tilde{n}_{{\textsc{g}}}}}}}\mkern-3.5mu
\left\lfloor \tilde{y}^{\textsc{g}}_{i} \right\rfloor
+
\mfrac{1}{k}
(\bar{n}_{{\textsc{d}}} \mkern-1.3mu+\mkern-.5mu N_{{\textsc{g}}} \mkern-2mu-\mkern-1mu 1)
   \\ \nonumber
&\le
-{{\textstyle{\sum\limits_{i=1}^{\tilde{n}_{{\textsc{g}}}}}}}\mkern-3.5mu
\left\lfloor \tilde{y}^{\textsc{g}}_{i} \right\rfloor
\mkern3.2mu+\mkern4mu
\bar{n}_{{\textsc{d}}} \mkern-1.3mu+\mkern-.5mu N_{{\textsc{g}}} \mkern-2mu-\mkern-1mu 1
   \\ \nonumber
&= \mkern2.5muy_-(1) \mkern1.5mu+\mkern1.5mu y_+\mkern-1mu(1) 
\mkern1.5mu+\mkern.5mu 1 \mkern2mu\in{\mathbb{Z}},
\end{align}
which, since 
$\max_{k >0}\mkern-1mu
\left(y_-\mkern-.4mu(k)\mkern-.2mu + y_+\mkern-.9mu(k)\right) \in {\mathbb{Z}}$,
implies that
$y_-(k)\mkern-.2mu + y_+\mkern-.9mu(k)$ is maximized at $k \mkern-1.5mu=\mkern-1.5mu 1$.
Thus, 
$y_-\mkern-.4mu(k_1)$ and $y_+(k_2)$ are respectively maximized
and minimized at $k_1 = 1$ and $k_2 = 1$,
completing the proof of the proposition.

\end{proof}

\section{Cabling}
\label{section: cabling}
The $(p,q)$-cable $Y^{(p,q)}\mkern-2mu\subset\mkern-2mu k$ of a knot complement 
$Y \mkern-2mu:=\mkern-2mu X \mkern-1.5mu\setminus\mkern-1.5mu \nu(K) \subset X$
is given by 
the knot complement $Y^{(p,q)} \mkern-3mu:=\mkern-2mu 
X \mkern-1.5mu\setminus\mkern-1.5mu \nu(K^{(p,q)})$,
where $K^{(p,q)}\mkern-1mu \subset X$ is the image of 
the $(p,q)$-torus knot embedded in the boundary of $Y\!$.
I recently made the mundane, and almost certainly not novel,
observation that
one can realize any cable of $Y \subset X$ by gluing on
an appropriate Seifert fibered space.

\subsection{Cabling via gluing}
Suppose $Y := X \setminus \nu(K)$ is the knot complement
of an arbitrary knot $K \subset X$ in an arbitrary closed oriented three-manifold $X$.
We construct the $(p,q)$-cable $Y^{(p,q)} \subset X$ of $Y \subset X$ as follows.

Let $\mu \in H_1(\partial Y)$ denote the meridian of $K$,
and let $\lambda\in H_1(\partial Y)$ denote a choice of longitude, so that
$X = Y(\mu)$ and $\mu \cdot \lambda = 1$.
Choosing $p^*, q^* \in Z$ such that $pp^* - qq^* = 1$, let
$Y_{(-q^*\mkern-3mu,\mkern1.5mu p)}$ denote the
regular fiber complement
\begin{equation}
Y_{(-q^*\mkern-3mu,\mkern1.5mu p)}
:= M_{S^2}(0\mkern.4mu; \,-q^*\mkern-2.5mu/p) \mkern1.5mu\setminus\mkern1.5mu \nu(f),
\end{equation}
so that 
$Y_{(-q^*\mkern-3mu,\mkern1.5mu p)}$
is a solid torus whose compressing disk has boundary of slope $\frac{q^*}{p}$.

To construct the $(p,q)$-cable $Y^{\mkern-.8mu(p,q)} \subset X$,
we form the union
\begin{equation}
Y^{\mkern-.8mu(p,q)} := \hat{Y}_{(-q^*\mkern-3mu,\mkern1.5mu p)} \cup_{\varphi} Y,\;\;\;\;
\varphi: \partial Y \to -\partial_1 \hat{Y}_{(-q^*\mkern-3mu,\mkern1.5mu p)},
\end{equation}
where $\hat{Y}_{(-q^*\mkern-3mu,\mkern1.5mu p)}$ is a regular fiber complement
in $Y_{(-q^*\mkern-3mu,\mkern1.5mu p)}$, with
$\partial_1 \hat{Y}_{(-q^*\mkern-3mu,\mkern1.5mu p)} \mkern-2mu:= \mkern-1.5mu
\partial Y_{(-q^*\mkern-3mu,\mkern1.5mu p)}$,
and where $\varphi$ is chosen to
induce the map $\varphi_*$ on homology defined by
\begin{equation}
\varphi_*(\mu) := -q^* \mkern-2mu\tilde{f}_1 +\mkern.5mu p\tilde{h}_1,\;\;
\varphi_*(\lambda):= p^*\mkern-2.2mu\tilde{f}_1 -\mkern.7mu q\tilde{h}_1.
\end{equation}

\begin{prop}
\label{prop: graph manifold for cabling works}
$Y^{\mkern-.8mu(p,q)} \mkern-1mu\subset\mkern-1mu Y^{\mkern-.8mu(p,q)}\mkern-.8mu(0) = X$
is the $(p,q)$-cable of the knot complement $Y \mkern-3mu\subset\mkern-2mu X$.
\end{prop}
{\noindent{\textit{Proof.}}
To verify that $Y^{\mkern-.8mu(p,q)}\mkern-.8mu(0) = X$,
note that $Y^{\mkern-.8mu(p,q)}\mkern-.8mu(0)$ is a
union of $Y$ with the solid torus $Y_{(-q^*\mkern-3mu,\mkern1.5mu p)}$,
such that $\mu$ is sent to the slope $\frac{q^*}{p}$ bounding
the compressing disk of $Y_{(-q^*\mkern-3mu,\mkern1.5mu p)}$.
In other words, $Y^{\mkern-.8mu(p,q)}\mkern-.8mu(0)$ is the Dehn filling
$Y\mkern-1.8mu(\mu) =: X$.}

Since $Y^{\mkern-.8mu(p,q)} \mkern-2mu=\mkern-2mu X \setminus \nu(f)$
is the complement of the regular fiber $f$,
we next must verify that, in the boundary of the solid torus 
($\textsc{st}$) to which $Y$ is glued,
the regular fiber is of class 
$pm_{\textsc{st}} + ql_{\textsc{st}} \in H_1(\partial Y_{(-q^*\mkern-3mu,\mkern1.5mu p)})$
in terms of the basis
$l_{\textsc{st}} := \varphi_*(\mu)$,
$m_{\textsc{st}} := \varphi_*(\lambda)$
specified by the meridian $\mu$ and longitude $\lambda$ of $K$.
Indeed, we have
\begin{equation}
\mkern110mu
pm_{\textsc{st}} + ql_{\textsc{st}}
= p\varphi_*(\lambda) + q\varphi_*(\mu)
=(pp^* - qq^*) \tilde{f}_1 = \tilde{f}_1.
\mkern100mu\qed
\end{equation}

\subsection{L-space intervals for cables}
Supposing $Y \mkern-2mu\subset\mkern-2mu X$ is Floer simple and boundary 
incompressible,
write
$a_-\mu + b_-\lambda$ and $a_+\mu + b_+\lambda$
for respective representatives in $H_1(\partial Y)$ 
of the left-hand and right-hand endpoints 
of the L-space interval $\mathcal{L}(Y) \mkern-1.5mu\subset\mkern-1.5mu {\mathbb{P}}(H_1(\partial Y))$.
Now, $\varphi$ is an orientation-reversing map, but since we change
from a positively-oriented basis to a negatively-oriented basis,
the induced map $\varphi_*^{{\mathbb{P}}}$ is orientation-preserving.
We therefore have
\begin{equation}
\varphi_*^{{\mathbb{P}}}(\mathcal{L}(Y)) = [[y_{1-}, y_{1+}]],
\;\;\;\;\;\;
y_{1\pm} := \frac{a_{\pm}q^* - b_{\pm}p^*}{a_{\pm}p - b_{\pm}q}
=\mkern1mu \frac{\mkern.5muq^*\mkern-2.5mu}{p} 
\mkern-5mu\left(1 - \frac{b_{\pm}}{q^*(a_{\pm}p-b_{\pm}q)}\right).
\end{equation}
For $k\in {\mathbb{Z}}_{>0}$, define
\begin{equation}
y_-(k)
:= -\mfrac{1}{k}\mkern-3mu\left(
\left\lfloor
\mkern-2mu-\mfrac{\mkern1.5muq^*\mkern-3.5mu}{p}k\mkern-1.5mu
\right\rfloor
+
\left\lceil y_{1+}k\mkern-1.5mu
\right\rceil
\right),
\;\;\;
y_+(k)
:= -\mfrac{1}{k}\mkern-3mu\left(
\left\lceil
\mkern-2mu-\mfrac{\mkern1.5muq^*\mkern-3.5mu}{p}k\mkern-1.5mu
\right\rceil
+
\left\lfloor
y_{1-}k\mkern-1.5mu
\right\rfloor
\right).
\end{equation}
which simplifies to
\begin{equation}
\label{eq: useful def of cabling y-(k) and y+(k)}
y_-(k)
:= \mfrac{1}{k}\mkern-3mu\left(
\left\lceil
\mkern-2mu\mfrac{\mkern1.5muq^*\mkern-3.5mu}{p}k\mkern-1.5mu
\right\rceil
-
\left\lceil y_{1+}k\mkern-1.5mu
\right\rceil
\right),
\;\;\;
y_+(k)
:= \mfrac{1}{k}\mkern-3mu\left(
\left\lfloor
\mkern-2mu\mfrac{\mkern1.5muq^*\mkern-3.5mu}{p}k\mkern-1.5mu
\right\rfloor
-
\left\lfloor
y_{1-}k\mkern-1.5mu
\right\rfloor
\right).
\end{equation}
As usual, we also define
\begin{equation}
\label{eq: cabling def for y- y_+}
y_- := \max_{k>0} y_-(k),\;\;\;\;\; y_+ := \min_{k>0} y_+(k).
\end{equation}

Since $Y$ is Floer simple and boundary incompressible,
and since
every Dehn filling of
$\hat{Y}_{(-q^*\mkern-3mu,\mkern1.5mu p)}$
along $\partial Y^{\mkern-.8mu(p,q)}$
is Floer simple,
we can still invoke 
Theorem~\ref{thm: l space interval for graph manifolds}
and Proposition~\ref{prop: characterization of Floer simple graph manifolds}
to compute the L-space interval for $Y^{\mkern-.8mu(p,q)}$
(see
Corollary~\ref{cor: generalization of JN to Floer}
for justification of this generalization).
Thus, in terms of the Seifert basis $(\tilde{f}, -\tilde{h})$, 
$\hat{Y}_{(-q^*\mkern-3mu,\mkern1.5mu p)}$ has L-space interval
$[[y_-, y_+]]$ if it is Floer simple, and $\{y_-\} = \{y_+\}$ if
it has an isolated L-space filling.

It is often more natural, however, to express this L-space interval
in terms of the surgery basis for the cabled knot.
Recall that
$\hat{Y}_{(-q^*\mkern-3mu,\mkern1.5mu p)} = X \setminus \nu(f)$.
The natural surgery basis associated to the complement of the
regular fiber is given by the meridian $\mu^{(p,q)} := -\tilde{h}$
and longitude $\lambda^{(p,q)} = \tilde{f}$, yielding the following result.

\begin{theorem}
\label{thm: general cabling}
Suppose $Y \mkern-2mu=\mkern-2mu X\mkern-2mu\setminus\mkern-2mu \nu(K)$
is a boundary incompressible
Floer simple knot complement
with L-space interval 
$\mathcal{L}(Y) = [[\frac{a_-}{b_-}, \frac{a_+}{b_+}]]$ in terms of
the surgery basis $\mu, \lambda \in H_1(\partial Y)$ for $K$, with 
$\mu$ the meridian of $K$ and $\lambda$ a choice of longitude.
Then in terms of the surgery basis produced by cabling,
the $(p,q)$-cable $Y^{(p,q)} \mkern-2mu\subset\mkern-2mu X$ of 
$Y \mkern-2mu\subset\mkern-2mu X$
has L-space interval
\begin{equation}
\mathcal{L}(Y^{(p,q)}) =
\begin{cases}
\emptyset
&\mkern15mu
\infty \mkern-1.5mu\neq y_{1-} \mkern-4.5mu<\mkern-.5mu y_{1+} 
\mkern-2mu\neq\mkern-1.5mu \infty
\mkern9mu\text{and}\mkern11mu
y_- \mkern-3.5mu>\mkern-2.5mu y_+
     \\
\{1/y_-\} = \{1/y_+\}
&\mkern15mu
\infty \mkern-1.5mu\neq y_{1-} \mkern-4.5mu<\mkern-.5mu y_{1+} 
\mkern-2mu\neq\mkern-1.5mu \infty
\mkern9mu\text{and}\mkern11mu
y_- \mkern-3.5mu=\mkern-2.5mu y_+
     \\
[[1/y_+, 1/y_-]]
&\mkern15mu
\text{otherwise}.
\end{cases}
\end{equation}
\end{theorem}

\subsection{Knots in ${\boldsymbol{S^3}}$}

As an illustration, we apply the above result to an arbitrary
boundary incompressible Floer simple knot complement 
$Y \mkern-2mu:=\mkern-2mu S^3 \mkern-2mu\setminus\mkern-2mu \nu(K)$ in $S^3\!$.
The surgery basis for a knot complement in $S^3$
conventionally takes $\lambda$ to be the rational longitude,
which in $S^3$ is Seifert framed.
The meridian $\mu$ of $K$ is automatically dual to this $\lambda$.

Without loss of generality (up to replacing $K$ with its mirror image),
we demand that $K$ be {\em{positive}},
by which we mean that there exist positive $u, v \in {\mathbb{Z}}$ such that
the Dehn filling $Y(u\mu + v\lambda)$ is an L-space.
In terms of the projectivization map 
$x\mu + y\lambda \mapsto x/y \in {\mathbb{Q}} \cup \{\infty\}$,
it is easy to show (see ``example'' in \cite[Section 4]{lslope})
that $Y$ has L-space interval
\begin{equation}
\mathcal{L}(Y) = [N, +\infty],\;\;\;\;
N:= 2g(K) - 1 = \deg(\Delta(K)) - 1,
\end{equation}
where $g(K)$ and $\Delta(K)$ are the genus and Alexander polynomial of $K$.

Choosing $q^*, p^* \in {\mathbb{Z}}$ such that
$pp^* - qq^* = 1$, and demanding $0 \le q^* \mkern-2mu< p$, we then have
\begin{equation}
\varphi_*^{{\mathbb{P}}}(\mathcal{L}(Y)) = [[y_{1-}, y_{1+}]],
\;\;\;\;\;
y_{1-}:= \mfrac{Nq^* - p^*}{Np - q}
= \mfrac{q^*}{p} + \mfrac{1}{p(q-Np)},
\;\;\;\;\;
y_{1+}:= \mfrac{q^*}{p}.
\end{equation}
From
(\ref{eq: useful def of cabling y-(k) and y+(k)}) and
(\ref{eq: cabling def for y- y_+}),
we immediately compute that $y_- = 0$. For $y_+ $, we have

\begin{equation}
y_+ = \min_{k>0}y_+(k),
\;\;\;\;
y_+(k)
:= \mfrac{1}{k}\mkern-3mu\left(
\left\lfloor
\mkern-2mu\mfrac{\mkern1.5muq^*\mkern-3.5mu}{p}k\mkern-1.5mu
\right\rfloor
-
\left\lfloor
\left(\mfrac{q^*}{p} + \mfrac{1}{p(q-Np)}\right)k
\right\rfloor
\right).
\end{equation}

If $q-Np \mkern-.5mu<\mkern-1.5mu 0$, then 
$y_+\mkern-1.3mu(k) \mkern-2mu\ge\mkern-2mu 0$ for all 
$k \mkern-2mu>\mkern-2mu 0$, and so $y_+\mkern-1.3mu(k)$ is minimized
at $y_+ \mkern-3mu=\mkern-1mu y_+\mkern-1.3mu(1) \mkern-1.5mu=\mkern-1.5mu 0$. 
Since
$\infty \mkern-2mu\neq y_{1\mkern-.8mu-} \mkern-4mu<\mkern-.5mu y_{1\mkern-1mu+} 
\mkern-3mu\neq\mkern-1.5mu \infty$ in this case, we then have
$\mathcal{L}(Y^{(p,q)}) = \{0\}$.
If $q-Np = 0$, then $y_+ \mkern-2.5mu=\mkern-2mu \infty$,
yielding $\mathcal{L}(Y^{(p,q)}) \mkern-2mu=\mkern-2mu [0, \infty]$.
If $q-Np < 0$, then
$y_+(k) \le 0$ for all $k>0$, and it is straightforward
to show that $y_+(k)$ is minimized at the lowest value of $k_+>0$
for which $\lfloor \frac{q^*}{p}k_+ \rfloor \neq \lfloor y_{1-}k_+ \rfloor$. 
Since $y_{1-} - \frac{q^*}{p} < \frac{1}{p}$,
a necessary condition for this to occur is to have
\begin{equation}
\left(y_{1-} - \mfrac{q^*}{p}\right)k_+ \ge \mfrac{1}{p},
\;\;\;
\text{which implies}\;\;\;
k_+ \ge q - Np.
\end{equation}
Since $y_{1+}(q-Np) = p^* - Nq^* \in {\mathbb{Z}}$,
setting $k_+ = q - Np$ is also sufficient:
\begin{equation}
y_+ = y_+\mkern-1.5mu(k_+) = \mfrac{1}{q-Np}
\left((p^* - Nq^*-1) - (p^* - Nq^*)\right) = -\mfrac{1}{q-Np},
\end{equation}
and we have $\mathcal{L}(Y^{(p,q)}) = [[0, -1/(q-Np)]]$.

As a final step, we re-express $\mathcal{L}(Y^{(p,q)})$
in terms
of the conventional basis for knot complements in $S^3$.
We again use the meridian
$\mu^{(p,q)} := -\tilde{h}$, but the rational longitude is
\begin{equation}
l^{(p,q)} = 
-\left(-\mfrac{q^*}{p} + \varphi_*^{{\mathbb{P}}}([\lambda])
\right)
=
-\left(-\mfrac{q^*}{p} + \mfrac{p^*}{q}\right)
= -\mfrac{1}{pq},
\end{equation}
for which we choose the representative
$\lambda_{{\mathbb{Q}}} := \tilde{f} + pq\tilde{h}$
to achieve
$\mu^{(p,q)} \cdot \lambda_{{\mathbb{Q}}}=1$.
Performing the requisite change of basis on
$\mathcal{L}(Y^{(p,q)})$ for the three cases
described in the preceding paragraph then
recovers the following result of
Hedden \cite{Heddencableii} and Hom \cite{Homcable}.
\begin{cor}
$Y^{(p,q)} \subset S^3$ has L-space interval
\begin{equation}
\mathcal{L}(Y^{(p,q)}) = 
\begin{cases}
\{\infty\}
&2g(K) - 1 >
\frac{q}{p}
   \\
[\mkern.5mu pq \mkern-1.5mu-\mkern-1.5mu p 
\mkern-1.5mu-\mkern-1.5mu q \mkern-1.5mu+\mkern-1.5mu 2g(K)p , \;\infty\mkern.2mu]
&2g(K) - 1 \le
\frac{q}{p}.
\end{cases}
\end{equation}
\end{cor}

\subsection{Knots in L-spaces}
It is possible to prove an analogous result for
boundary incompressible
Floer simple knot complements in arbitrary L-spaces.

To simplify the statement of such a result, we
discard cables with $p=0$, $p=1$, or $q=0$,
since the zero-cable of a knot complement $Y \subset X$
is just the connected sum of $X$ with the unknot complement in $S^3$;
the $1/q$-cable is just a change of framing;
and the $1/0$-cable, which changes the framing by zero,
is the identity cable.  We then have the following.

\begin{theorem}
Suppose that $p, q \in {\mathbb{Z}}$ with $p>1$ and $\gcd(p,q) = 1$,
and that $Y = X\setminus \nu(K)$ is a boundary incompressible
Floer simple knot complement
in an L-space $X$, 
with L-space interval 
$\mathcal{L}(Y) = [[\frac{a_-}{b_-}, \frac{a_+}{b_+}]]$, written in
terms of the surgery basis $\mu, \lambda \in H_1(\partial Y)$ for $K$, with 
$\mu$ the meridian of $K$ and $\lambda$ a choice of longitude.
Then in terms of the surgery basis produced by cabling,
the $(p,q)$-cable $Y^{(p,q)} \mkern-2mu\subset\mkern-2mu X$ of 
$Y \mkern-2mu\subset\mkern-2mu X$
has L-space interval
\begin{equation}
\mathcal{L}(Y^{(p,q)}) =
\begin{cases}
\{\infty\}
&\mkern15mu
\mfrac{a_-\mkern-2mu}{b_-\mkern-2mu} \in
\left[\mfrac{p^*}{q^*}, \infty\right],\;
\mfrac{a_+\mkern-2mu}{b_+\mkern-2mu} \in
\left[\mfrac{q-p^*}{p-q^*}, \mfrac{q}{p} \right> \cup \{\infty\}
     \\
[[1/y_+, 1/y_-]]
&\mkern15mu
\text{otherwise},
\end{cases}
\end{equation}
where $pp^* - qq^* = 1$ with $0 < q^* < p$.
\end{theorem}
\begin{proof}
If $Y^{(p,q)}$ is Floer simple, then
Theorem~\ref{thm: general cabling}
implies $\mathcal{L}(Y^{(p,q)}) = [[1/y_+, 1/y_-]]$
in terms of the surgery basis produced by cabling.

Observe that $Y^{(p,q)}$ is not Floer simple
if and only if  $y_- = y_+ = 0$.
That is, if $Y^{(p,q)}$ is not Floer simple,
then the meridional filling $X$ is the only L-space filling,
and conversely if $y_- = y_+ = 0$, then
since $0 \notin [[0, 0]]$, we
know that $Y^{(p,q)}$ is not Floer simple.

Choose $p^*, q^* \in {\mathbb{Z}}$ so that $pp^* - qq^* = 1$ with $0 < q^* < p$.
Then for $y_{1\pm} \neq \infty$, one has
\begin{align}
\begin{cases}
y_-\mkern-2mu= -\lceil y_{1+} \rceil + 1
&
1-[-y_{1+}] \ge \frac{q^*}{p}
 \\
y_-\mkern-2mu\notin {\mathbb{Z}}
&
1-[-y_{1+}] < \frac{q^*}{p},
\end{cases}
\;\;\;\;\;\;\;
\begin{cases}
y_+\mkern-2mu= 
-\lfloor y_{1-} \rfloor
&
[y_{1-}] \le \frac{q^*}{p}
 \\
y_+\mkern-2mu\notin {\mathbb{Z}}
&
[y_{1-}] > \frac{q^*}{p}.
\end{cases}
\end{align}
That is, since
\begin{equation}
y_+
= -\lfloor y_{1-} \rfloor +
\min_{k>0}
\mfrac{1}{k}\mkern-3mu\left(
\left\lfloor
\mkern-2mu\mfrac{\mkern1.5muq^*\mkern-3.5mu}{p}k\mkern-1.5mu
\right\rfloor
-
\left\lfloor
[y_{1-}]k\mkern-1.5mu
\right\rfloor
\right),
\end{equation}
the right-hand summand vanishes when 
$[y_{1-}] \le \frac{q^*}{p}$,
but when
$[y_{1-}] > \frac{q^*}{p}$, the right-hand summand is not
minimized at $k=1$.  Thus $y_+ \neq y_+(1)$, and
Proposition
\ref{prop: k- and k+ are denominators of y- and y+}
tells us that
$y_+ \notin {\mathbb{Z}}$.
A similar argument holds for $y_-$.

We therefore have
\begin{equation}
y_- = 0
\;\;\;\iff\;\;\;
{{\textstyle{\mfrac{q^*}{p}}}} \le y_{1+} \le 1
\;\;\;\iff\;\;\;
\mfrac{q-p^*}{p-q^*} \le \mfrac{a_+}{b_+} < \mfrac{q}{p}
\mkern10mu\text{or}\mkern10mu
\mfrac{a_+}{b_+} = \infty,
\end{equation}
and similarly,
\begin{equation}
\mkern3mu
y_+ = 0
\;\;\;\iff\;\;\;
0 \le y_{1-} \le {{\textstyle{\mfrac{q^*}{p}}}}
\;\;\;\iff\;\;\;
\mfrac{p^*}{q^*} \le \mfrac{a_-}{b_-} \neq \infty
\mkern10mu\text{or}\mkern10mu
\mfrac{a_-}{b_-} = \infty.
\mkern12mu
\end{equation}
Thus $y_- = y_+ =0$ if and only if
\begin{equation}
\mfrac{a_-}{b_-} \in
\left[\mfrac{p^*}{q^*}, \infty\right]
\;\;\;\;\text{and}\;\;\;\;
\mfrac{a_+}{b_+} \in
\left[\mfrac{q-p^*}{p-q^*}, \mfrac{q}{p} \right> \cup \{\infty\}.
\end{equation}
\end{proof}

\section{Observations}
\label{s: Observations}

Our demonstration of extended L/NTF-equivalence for graph manifolds in 
Theorem~\ref{thm: lntf equivalence}
gives a (mildly) alternate proof of
the Theorem \ref{thm: intro version of l = ntf} statement
that a graph manifold is an L-space if and only if it fails
to admit a co-oriented taut foliation.

From a practical standpoint, however, the main utility of
Theorem~\ref{thm: lntf equivalence}
for us was its implication that
the gluing result in
Proposition~\ref{prop: gluing prop for LNTF equivalent manifolds}
holds for all graph manifolds:
\begin{cor}
\label{final gluing result for graph manifolds}
If $Y_1$ and $Y_2$ are non-solid-torus 
graph manifolds with torus boundary,
then the union $Y_1 \mkern-1mu\cup_{\varphi}\! Y_2$, with gluing map
$\varphi \mkern-1mu:\mkern-1mu \partial Y_1 \to -\partial Y_2$,
is an L-space if and only if
\begin{equation*}
\varphi_*^{{\mathbb{P}}}(\mathcal{L}^{\circ}(Y_1)) \cup \mathcal{L}^{\circ}(Y_2)
={\mathbb{P}}(H_1(\partial Y_2)).
\end{equation*}
\end{cor}

Corollary~\ref{final gluing result for graph manifolds}
has two advantages over the more general L-space gluing criterion of
Proposition~
\ref{prop: L-space gluing thm from lslope}:
it removes the condition that
$\varphi_*^{{\mathbb{P}}}(\mathcal{L}^{\circ}(Y_1)) \cap \mathcal{L}^{\circ}(Y_2)$
be nonempty,
and it allows one to prove that $Y_1 \cup Y_2$ is not an L-space
in cases in which boundary incompressible $Y_1$ and $Y_2$ are not Floer simple.

\subsection{Generalization of Theorem \ref{thm: l space interval for graph manifolds}}
Nevertheless, while the L-space gluing result analogous to 
Proposition~\ref{prop: L-space gluing thm from lslope} proved by
Hanselman and Watson in \cite{HanWat}
replaces the hypothesis of Floer simplicity by
a more technical condition,
their gluing result does not impose the hypothesis of nonempty
$\varphi_*^{{\mathbb{P}}}(\mathcal{L}^{\circ}(Y_1)) \cap \mathcal{L}^{\circ}(Y_2)$
required by the gluing result of J. Rasmussen and the author in \cite{lslope}.
In \cite{HRRW}, the four authors discuss how these two gluing
results can be combined to prove a gluing result analogous to
Proposition~\ref{prop: L-space gluing thm from lslope}
which requires Floer simplicity but not
nonempty
$\varphi_*^{{\mathbb{P}}}(\mathcal{L}^{\circ}(Y_1)) \cap \mathcal{L}^{\circ}(Y_2)$.
Thus, the only real hypothesis we have circumvented is that of Floer simplicity.
If we replace the condition that the $Y_i$ glued to $\hat{M}$ be
graph manifolds with the condition that they be Floer simple, 
then we can extend the domain of validity of 
Proposition~\ref{prop: characterization of Floer simple graph manifolds}
and
Theorem~\ref{thm: l space interval for graph manifolds}
as follows.
\begin{cor}
\label{cor: generalization of JN to Floer}
Theorem
~\ref{thm: l space interval for graph manifolds}
holds for any boundary incompressible Floer simple three-manifolds
$Y_1, \ldots, Y_{n_{{\textsc{g}}}}$, provided that
$Y$ satisfies the criteria in
Proposition~\ref{prop: characterization of Floer simple graph manifolds}
for $\mathcal{L}(Y)$ to be nonempty.
\end{cor}

\subsection{Generalized Solid Tori}
A recent result of Gillespie \cite{Gillespietorus}
states that a compact oriented three-manifold $Y$ with torus boundary
satisfies $\mathcal{L}(Y) \mkern-1.5mu= {\mathbb{P}}(H_1(\partial Y)) \mkern-1mu\setminus\mkern-1mu \{l\}$
if and only if $Y\!$ has genus 0 and an L-space filling.
Such manifolds, called {\em{generalized solid tori}}
in \cite{lslope}, are of independent interest.

In the proof of
Theorem~\ref{thm: l space interval for graph manifolds}
and 
Proposition~\ref{prop: characterization of Floer simple graph manifolds},
we find many generalized solid tori with the regular
fiber class as rational longitude, but there are limited
circumstances in which other generalized solid tori appear.
In fact, we can prove the following.
\begin{theorem}
\label{thm: graph manifold generalized solid torus}
If $Y\!$ is a graph manifold with torus boundary, $b_1(Y) \mkern-2.5mu=\mkern-2.5mu 1$, and rational
longitude other than the regular fiber, then
$Y$ is a generalized solid torus if and only if
it is homeomorphic to an iterated
cable of the regular fiber complement in $S^1\mkern-2.5mu\times\mkern-.8mu S^2\mkern-4mu$.
\end{theorem}
\begin{proof}
Suppose $Y$ is a generalized solid torus graph manifold, with rational longitude
$l$ not coinciding with the regular fiber.
Since $Y$ is Floer simple with $\infty \in \mathcal{L}^{\circ}(Y)$,
$Y$ must satisfy 
$(\textsc{fs3}\mkern-.5mu)$ from 
Proposition~\ref{prop: characterization of Floer simple graph manifolds},
and since $y_-, y_+ \mkern-2mu\neq\mkern-2mu \infty$, we
must have
$\infty \neq y_{i-}^{\textsc{g}} \ge y_{i+}^{\textsc{g}} \neq \infty$.
Claim~\ref{claim: y- and y+ outside long+ and long-} from 
Section~\ref{ss: orientable base, no solid tori}
then implies $y_- \mkern-3.5mu>\mkern-2mu y_+$ unless 
$n_{{\textsc{g}}} \mkern-4mu=\mkern-2mu 1$ with $n_{{\textsc{d}}} \mkern-3mu\le\mkern-2mu 1$
or $n_{{\textsc{g}}} \mkern-4mu=\mkern-2mu 0$ with $n_{{\textsc{d}}} \mkern-3mu\le\mkern-2mu 2$.

If $n_{{\textsc{g}}}\mkern-2mu=\mkern-2mu1$, then according to 
Claim
\ref{claim: Yhat solid torus, y- y+ outside longs with equality sometimes}
from Section~\ref{ss: orientable base, cases involving solid torus Y hat},
either $y_1^{\textsc{d}} \in {\mathbb{Z}}$, in which case $Y$ is homeomorphic
to $Y_1$ and we should
replace $Y$ with $Y_1$ and begin again;
or $y_{-1}^{\textsc{g}} = y_{+1}^{\textsc{g}} =: y_{1}^{\textsc{g}}$
with $y_1^{\textsc{d}} + y_1^{\textsc{g}} \in {\mathbb{Z}}$,
in which case
$Y_1$ is a generalized solid torus, and
Proposition
\ref{prop: graph manifold for cabling works}
implies 
$Y \subset Y(l)$ is a cable of $Y_1 \subset Y_1(l)$.

If $n_{{\textsc{g}}} = 0$, then $Y$ is Seifert fibered with $y_- = y_+$,
and so either from \cite{lslope} or from a mildly modified version of
Claim~\ref{claim: Yhat solid torus, y- y+ outside longs with equality sometimes},
we deduce that
either $y_1^{\textsc{d}} + y_2^{\textsc{d}} \in \mathbb{Z}$,
in which case 
Proposition
\ref{prop: graph manifold for cabling works}
implies $Y\!$ is a cable of the regular fiber complement in
$S^1 \mkern-4mu\times\! S^2$;
or $\{y_1^{\textsc{d}}, y_2^{\textsc{d}}\} \mkern-1mu\cap\mkern-1mu {\mathbb{Z}} \mkern-2mu\neq\mkern-2mu \emptyset$,
in which case $\mkern-1muY\mkern-2.5mu$ is a solid torus, hence homeomorphic to the
regular fiber complement in $S^1 \mkern-3.5mu\times\mkern-2.5mu S^2\!\!$.

The converse is an immediate corollary of
Proposition~\ref{prop: graph manifold for cabling works}
and
Theorem~\ref{thm: l space interval for graph manifolds}.
\end{proof}

We also have the following result for arbitrary generalized solid tori.
\begin{prop}
If $Y$ is a generalized solid torus, then
any cable of $Y \subset Y(l)$ is a generalized solid torus.
\end{prop}
\begin{proof}
If $Y$ is boundary compressible, then it is the connected sum of a solid
torus with lens spaces, and
Theorem~\ref{thm: graph manifold generalized solid torus}
implies that any cable of a solid torus within its longitudinal filling
is a generalized solid torus.
If $Y$ is boundary incompressible, then the result is an immediate corollary of
Theorem~\ref{thm: general cabling}.
\end{proof}

Similarly, for any class of manifolds for which the gluing result in
Proposition~\ref{prop: intro gluing result for graph manifolds}
holds without the requirement of Floer simplicity---such as graph manifolds---one
has the result that if $Y$ has an isolated L-space filling, {\em{i.e.}},
if $\mathcal{L}(Y) = \{\mu\}$ for some $\mu \in {\mathbb{P}}(H_1(\partial Y))$,
then any cable of $Y \subset Y(\mu)$ has $Y(\mu)$ as an isolated L-space filling.

\subsection{Isolated L-space fillings}
A Seifert fiber complement in an L-space Seifert fibered manifold
could justifiably be called
the prototypical Floer simple manifold,
just as a lens space is the prototypical L-space.
It is therefore striking that we encounter
isolated L-space fillings as
regular fiber complements in graph manifolds.
Fortunately, this still does not prevent
L-space graph manifolds from admitting Floer simple
Seifert fiber complements.

Given a closed graph manifold $X$, we shall call an 
exceptional fiber $f_{\textsc{e}} \subset X$
{\em{invariantly exceptional}} if the JSJ componenent $\hat{Y} \subset X$
containing $f_{\textsc{e}}$ has more than one exceptional fiber.
To motivate this name, note that if $X$ has more than one JSJ component and
$\hat{Y}$ has only one exceptional fiber,
say, of slope $y_1^{\textsc{d}} = y_0$, then 
since the punctured solid torus has nonunique Seifert structure,
$X$ is homemorphic to
a graph manifold in which 
$y_1^{\textsc{d}}$ is replaced with
$0$ and $\varphi_{1*}^{{\mathbb{P}}}$ is replaced with $\varphi_{1*}^{{\mathbb{P}}} \!+ y_0$.

\begin{theorem}
Every invariantly exceptional fiber complement 
in an L-space graph manifold is Floer simple.
\end{theorem}
\begin{proof}
Suppose $X$ is an L-space graph manifold.
If $X$ is Seifert fibered, then every Seifert fiber complement,
regular or otherwise, is Floer simple.

Suppose $X$ has more than one JSJ component,
and let $Y$ denote a non-Floer-simple 
complement of an invariantly exceptional fiber.
Since $\mathcal{L}(Y) \neq \emptyset$,
$Y$ non-Floer-simple implies $\mathcal{L}(Y) = \{y_-\} = \{y_+\}$,
with $y_{\pm} \in {\mathbb{Q}}$.
However, since $Y$ has at least one exceptional fiber,
Proposition~\ref{prop: y- = y+ means integer}
tells us that $y_- = y_+ \in {\mathbb{Z}}$, contradicting
the hypothesis that $Y$ is an exceptional fiber complement of $X$.
Thus the theorem holds.
\end{proof}

On the other hand, for a graph manifold with more than one JSJ component,
Seifert fibers are not the only knots yielding Floer simple knot complements,
due to the following result for arbitrary L-spaces.
\begin{prop}
\label{prop: floer simple knot across incompressible torus}
If an L-space $X$ decomposes as a union $Y_1 \cup_{\varphi}\! Y_2$ 
of Floer simple manifolds $Y_i$ along an incompressible torus $T \subset X$,
then there is a knot $K\subset X$
transversely intersecting $T$
for which knot the complement
$X \setminus \nu(K)$ is Floer simple.
\end{prop}
\begin{proof}
In fact, an analogous result holds for any toroidal L-space.

Suppose the above hypotheses hold.  Then since $Y_1$ and $Y_2$ 
are Floer simple with incompressible boundary,
Proposition~\ref{prop: L-space gluing thm from lslope} implies
\begin{equation}
\label{eq: y1 and y2 have have l-space intervals covering p1}
\varphi_*^{{\mathbb{P}}}(\mathcal{L}^{\circ}(Y_1)) \cup \mathcal{L}^{\circ}(Y_2) 
= {\mathbb{P}}(H_1(\partial Y_2)).
\end{equation}
Since $\varphi_*^{{\mathbb{P}}}(\mathcal{L}^{\circ}(Y_1))$
and $\mathcal{L}^{\circ}(Y_2)$ are open,
(\ref{eq: y1 and y2 have have l-space intervals covering p1}) implies
they intersect.
Choosing any 
$\mu_2 \in \varphi_*^{{\mathbb{P}}}(\mathcal{L}^{\circ}(Y_1)) \cap \mathcal{L}^{\circ}(Y_2)$,
set $\mu_1 := \varphi_*^{{\mathbb{P}}\mkern1mu -1}(\mu_2)$, and let 
$K_i$ denote the knot core of $Y_i(\mu_i) \setminus Y_i$.
As explained in more detail in the proof of 
\cite[Theorem 6.2]{lslope},
$X$ can be regarded as zero-surgery on
the knot $K_1 \# K_2 \subset Y_1(\mu_1) \# Y_2(\mu_2)$.
Thus, if we set $Y := Y_1(\mu_1) \# Y_2(\mu_2) \setminus \nu(K_1 \# K_2)$
and let $K$ denote the knot core of $X \setminus Y$,
then the knot complement $Y = X \setminus \nu(K)$ has at least
two distinct L-space fillings, hence is Floer simple.
Moreover, since $K$ is dual to $K_1 \# K_2$ under surgery,
$K$ intersects the separating torus transversely.
\end{proof}

\begin{cor}
If $X$ is an L-space graph manifold, then for every incompressible
torus $T \subset X$, there is a knot $K \subset X$ transversely
intersecting $T$ for which knot the complement $X \setminus \nu(K)$
is Floer simple.
\end{cor}
\begin{proof}
Choose an arbitrary incompressible torus $T \subset X$,
not necessarily one used in the minimal JSJ decomposition for $X$,
and write $X = Y_1 \cup_T Y_2$.
Since $X$ is an L-space, 
Corollary~\ref{final gluing result for graph manifolds}
implies each $Y_i$ has nonempty
$\mathcal{L}^{\circ}(Y_i)$, hence is Floer simple.
Thus, we can apply Proposition
\ref{prop: floer simple knot across incompressible torus}.
\end{proof}

This section has only cataloged the most obvious 
corollaries of the paper's main results.
We invite the reader to find more.

\bibliography{Lgraph}
\bibliographystyle{plain}

\end{document}

%% file: macros.tex
\newtheorem{claim}{Claim}
\newtheorem{lemma}{Lemma}[section]
\newtheorem{prop}[lemma]{Proposition}
\newtheorem{theorem}[lemma]{Theorem}
\newtheorem{cor}[lemma]{Corollary}
\newtheorem{definition}[lemma] {Definition}

\newtheorem{conj}[lemma]{Conjecture}

\newtheorem*{theorem*}{Theorem}
\newtheorem*{conj*}{Conjecture}

\DeclareMathOperator*{\Bigconnect}{{\mbox{\larger[2]{\#}}}}

\renewcommand{\ker}{{{\text {ker} \ }}}

\renewcommand{\epsilon}{{\varepsilon}}